\documentclass[10pt,reqno]{amsart}
\usepackage[utf8x]{inputenc}
\usepackage[english]{babel}
\usepackage[T1]{fontenc}
\usepackage{amssymb,amsmath}
\usepackage[hidelinks]{hyperref}
\usepackage{url}
\usepackage{indentfirst}
\usepackage{slashed}
\usepackage{enumerate,amsmath,amssymb, mathrsfs}

\usepackage{latexsym}
\usepackage{color}
\usepackage{accents}

\newcommand{\Acirc}{\accentset{\circ}{A}}
\allowdisplaybreaks

\newtheorem{prop}{Proposition}[section]
\newtheorem{thm}[prop]{Theorem}
\newtheorem{lem}[prop]{Lemma}
\newtheorem{coro}[prop]{Corollary}
\newtheorem{rema}[prop]{Remark}

\newtheorem{defi}[prop]{Definition}
\newtheorem{conj}[prop]{Conjecture}

\title[The Penrose inequality for asymptotically flat half-spaces]{The Riemannian Penrose inequality for  asymptotically flat manifolds with non-compact boundary}
\author{Thomas Koerber}
\address{ Albert-Ludwigs-Universit\"at Freiburg,
	Mathematisches Institut,
	Ernst-Zermelo-Straße 1
	D-79104 Freiburg, Germany}
\email{thomas.koerber@math.uni-freiburg.de}
\begin{document}
	\begin{abstract}
In this article, we prove the Riemannian Penrose inequality for asymptotically flat manifolds with non-compact boundary whose asymptotic region is modelled on a half-space. Such spaces were initially considered by Almaraz, Barbosa and de Lima in 2014. In order to prove the inequality, we develop a new approximation scheme for the weak free boundary inverse mean curvature flow, introduced by Marquardt in 2012, and establish the monotonicity of a free boundary version of the Hawking mass. Our result also implies a non-optimal Penrose inequality for asymptotically flat support surfaces in $\mathbb{R}^3$ and thus sheds some light on a conjecture made by Huisken.
	\end{abstract}
	\maketitle
	\section{Introduction}
Let $(M,g)$ be a complete, non-compact Riemannian three-manifold without boundary. We say that $(M,g)$ is asymptotically flat if there exists a compact set $\Omega\subset M$ such that every component $\hat M$ of $M-\Omega$ is diffeomorphic to $\mathbb{R}^3-B^3_1(0)$, if the metric expressed in terms of this chart satisfies  
\begin{align}
|g_{ij}(x)-{g_e}_{ij}|+|x||\partial_l g_{ij}(x)| \leq c|x|^{-1} \label{decay one}
\end{align}
and if the following lower bound for the Ricci-curvature holds
\begin{align}
\operatorname{Rc}\geq -c|x|^{-2}g. \label{decay two}
\end{align}
Here, $c>0$ is a positive constant, $x$ the position vector field of $\mathbb{R}^3$ and $g_e$ the Euclidean background metric. We  call $\hat M$ an end of $M$. Asymptotically flat manifolds play an important part in general relativity as they arise as  time-symmetric initial data sets for solutions of the Einstein field equations, see for instance \cite{chrusciel2010mathematical}. If the scalar curvature is integrable in $M$, it is well-known that each end $\hat M$ possesses a global geometric invariant called the ADM-mass and defined by
$$
m_{ADM}:=\lim_{r\to\infty} \frac{1}{16\pi}\int_{\mathbb{S}^3_r(0)}(\partial_jg_{ij}-\partial_ig_{jj})\frac{x_i}{|x|}\text{d}vol_e,
$$
see \cite{arnowitt1961coordinate,bartnik1986mass}. The physically natural condition of non-negative energy density translates to $(M,g)$ having non-negative scalar curvature $\operatorname{Sc}$, a condition known as the dominant energy condition. A classical rigidity result, proven by Schoen and Yau using minimal surface techniques and known as the positive mass theorem, see \cite{schoen1979proof}, states that the ADM-mass of each end is non-negative provided the dominant energy condition $\operatorname{Sc}\geq 0$ holds and that equality holds if and only if $(M,g)$ is isometric to the flat Euclidean space. For one, this result can be seen as a rigidity result in the theory of manifolds with non-negative scalar curvature. On the other hand, it  establishes a connection between the global mass of a space-like slice and its local geometry. Using a beautiful argument based on  classical results in gravitational physics, Penrose conjectured that in certain cases a stronger, quantitative version of the positive mass theorem should hold, see \cite{penrose1973naked,penrose1982some}. This conjecture which is now known as the Riemannian Penrose inequality was subsequently proven by Huisken and Ilmanen in a seminal paper, see \cite{huisken2001inverse}. In order to state the result, we define a non-compact subset  $M'\subset M$ to be an exterior region if $M'$ is non-compact and connected, $\partial M'$ consists of compact minimal surfaces and if $M'$ contains no other minimal surfaces. 
\begin{thm}
	Let $(M,g)$ be a complete asymptotically flat Riemannian three-manifold with non-negative scalar curvature and $M'\subset M$ be an exterior region with ADM-mass $m_{ADM}$. Then 
	\begin{align}
	m_{ADM}\geq \sqrt{\frac{|\Sigma|}{16\pi}}, \label{nb RPI}
	\end{align}
	where $ \Sigma\subset \partial M'$ is any connected component of $\partial M$. Equality holds if and only if $(M',g)$ is one half of the spatial Schwarzschild manifold.
	\label{nb RPI thm}
\end{thm}
In order to prove the theorem, Huisken and Ilmanen used the so-called weak inverse mean curvature flow starting at $\Sigma$ to sweep out the exterior region. They showed that the so-called Hawking mass is monotonous along this flow, initially equal to the right-hand side of (\ref{nb RPI}) and asymptotic to $m_{ADM}$. We will discuss the proof in more detail later on. We also remark that a stronger version, where $\Sigma$ is replaced by $\partial M'$ in $(\ref{nb RPI})$, was later on shown by Bray in \cite{bray2001proof}. \\
Naturally, one may ask if quantities such as the ADM-mass can also be defined for complete Riemannian three-manifolds $(M,g)$ with a non-compact boundary $\partial M$ and if appropriate versions of the positive mass theorem and the Penrose inequality remain true provided suitable conditions hold. To this end, we say that $(M,g)$ is an asymptotically flat half-space, possibly with multiple ends, if there exists a compact set $\Omega\subset M$ such that any component $\hat M$ of $M\setminus \Omega$ is diffeomorphic to $\mathbb{R}^3_+\setminus B_1^3(0)$ and if the diffeomorphism can be chosen in a way such that (\ref{decay one}) and (\ref{decay two}) are satisfied. Here, $\mathbb{R}^3_+$ denotes the upper half-space $\{x\in\mathbb{R}^3|x_3\geq0\}$. In order to proceed, let us fix some terminology. Given a set $U\subset M$, we call $\tilde \partial U:=\overline{\partial U\setminus \partial M}$ the interior boundary of $U$ and $\hat\partial U=\partial U\cap \partial M$ the exterior boundary of $U$. Moreover, we say that a compact and connected hypersurface $\Sigma$ is a free boundary surface if $\partial \Sigma\neq \emptyset$, $\Sigma\cap \partial M=\partial \Sigma$ and if $\Sigma$ meets $\partial M$ orthogonally. Contrary, we say that $\Sigma$ is closed if it is compact and connected, has no boundary and does not touch $\partial M$.  As before, a non-compact, connected subset $M'\subset M$ is called an exterior region if $\tilde \partial M'$ consists of closed and free boundary minimal surfaces  and if $M'$ contains no other closed or free boundary minimal surfaces.   In \cite{almaraz2014positive}, Almaraz, Barbosa and de Lima studied such asymptotically flat half-spaces, calling them asymptotically flat manifolds with non-compact boundary, and discovered a mass type quantity, again called the ADM-mass, which can be assigned to each end $\hat M$ and is given by
$$
\tilde m_{ADM}:=\lim_{r\to\infty} \frac{1}{16\pi}\bigg(\int_{\mathbb{S}^3_r(0)\cap \mathbb{R}^3_+}(\partial_jg_{ij}-\partial_ig_{jj})\frac{x_i}{|x|}\text{d}vol_e+\int_{\partial D^2_r(0)\times\{0\}} g_{i3}\frac{x_i}{|x|}\text{d}vol_e\bigg ).
$$
They showed that $\tilde m_{ADM}$ is a well-defined geometric invariant if the scalar curvature $\operatorname{Sc}$ and the mean curvature $H^{\partial M}$ of $\partial M$ are integrable. Moreover, they proved the positive mass type rigidity statement  $\tilde m_{ADM}\geq 0$ provided the dominant energy condition $\operatorname{Sc},H^{\partial M}\geq 0$ holds with equality if and only if $(M,g)$ is the flat Euclidean half-space. For a more precise statement, we refer to Section \ref{asymptotically flat half spaces}. At this point, it is natural to expect that a suitable version of the Penrose inequality holds for asymptotically flat half-spaces, too. In \cite{barbosa2018positive}, Barbosa and Meira verified the Riemannian Penrose inequality for asymptotically flat half-spaces if the exterior region $M'$ can be written as a certain graph over $\mathbb{R}^3_+$. On the other hand, in \cite{marquardt2017weak}, Marquardt studied the so-called weak free boundary inverse mean curvature flow supported on certain convex graphs and discovered a monotonous quantity in case the flow remains smooth. In this article, we extend the theory of the weak free boundary inverse mean curvature flow and show the following Penrose type inequality for asymptotically flat half-spaces. 
\begin{thm}
	Let $(M,g)$ be an asymptotically flat half-space satisfying (\ref{asymptotic behavior 0}), (\ref{asymptotic behavior 1})  as well as the dominant energy condition $\operatorname{Sc}, H^{\partial M}\geq 0$. Let $M'\subset M$ be an exterior region with ADM-mass $\tilde m_{ADM}$ and suppose that $\Sigma$ is a connected free boundary component of $\tilde \partial M'$. Then there holds 
	$$
	\tilde m_{ADM}\geq\sqrt{\frac{|\Sigma|}{{32\pi}}}
	$$ 
	with equality if and only if $(M',g)$ is one-half of the spatial Schwarzschild half-space. \label{main thm RPI}
\end{thm} 
Of course, it would be interesting to know if the inequality also holds for a closed component or even the full interior boundary $\tilde \partial M'$. Unfortunately, we will see in Section \ref{fb imcf general section} that there does not seem to be a reasonable monotonous quantity along the weak free boundary inverse mean curvature starting  at a closed surface. \\ Before we describe the proof of Theorem \ref{main thm RPI}, we discuss another application of our result. In his PhD thesis \cite{volkmann2015free}, Volkmann studied so-called asymptotically flat support surfaces $S\subset \mathbb{R}^3$ and defined an exterior mass $m_{ext}$ associated to each end of $S$. Assuming the dominant energy type condition $H^S\geq 0$ he then proceeded to prove non-negativity of $m_{ext}$ with equality if and only if $S$ is a flat plane, see Section \ref{asymptotically flat half spaces} for more details. It was subsequently conjectured by Huisken that an appropriate version of the Penrose inequality holds for such surfaces $S$, too. To this end, we say that $S'\subset S$ is an exterior surface if $S'$ is non-compact and connected, if there is a free boundary minimal surface $\Sigma\subset \mathbb{R}^3$ with respect to $S$ such that $\partial \Sigma=\partial S'$ and if there are no other free boundary minimal surfaces with respect to $S'$. The following conjecture is taken from \cite{volkmann2015free}.
\begin{conj}
	Let $S$ be an asymptotically flat support surface with $H^S\geq 0$ and $S'\subset S$ be an exterior surface with free boundary minimal surface $\Sigma$ and exterior mass $m_{ext}$. If $\Sigma$ is connected, there holds
	\begin{align}
	m_{ext}\geq\sqrt{\frac{|\Sigma|}{{\pi}}} 
	\label{HC PR}
	\end{align}
	with equality if and only if $S'$ a half-catenoid and $\Sigma$  the free boundary disc  contained in the symmetry plane of the catenoid. 
	\label{Huisken conjecture}
\end{conj}
It is natural to expect the catenoid to occur in the rigidity case as it is a minimal surface which exhibits a symmetry similar to the one of the spatial Schwarzschild space. Moreover, using a divergence structure of the exterior mass, Volkmann was able to verify this conjecture for certain graphical surfaces. Using the weak free boundary inverse mean curvature flow, we can show a  weaker inequality which we do not expect to be optimal. Namely, we will show that an exterior surface $S$  can be realised as the exterior boundary of an exterior region of an asymptotically flat half-space $M$ such that $m_{ext}=4m_{ADM}$, see Lemma  \ref{mass vs extrinsic mass}. Combining this with Theorem \ref{main thm RPI} yields the following.
\begin{coro}
	Let $S$ be an asymptotically flat support surface with $H^S\geq 0$ and $S'\subset S$ an exterior surface with free boundary minimal surface $\Sigma$ and exterior mass $m_{ext}$. If $\Sigma$ is connected, there holds
	$$
	m_{ext}\geq	\sqrt{\frac{|\Sigma|}{{2\pi}}}.
	$$
	\label{volkmann coro}
\end{coro}
The previous corollary suggests that the free boundary inverse mean curvature flow cannot be used to prove (\ref{HC PR}). In fact, the inverse mean  curvature flow with free boundary seems to behave in a rather erratic way in the absence of ambient curvature as we will see in Section \ref{fb imcf general section}. Proving or disproving (\ref{HC PR}) thus remains an open problem. 
\\
We now describe the proof of Theorem \ref{main thm RPI}. To this end, we first provide a quick summary of the weak inverse mean curvature flow and its application to the Riemannian Penrose inequality for asymptotically flat manifolds without boundary. Let $(M,g)$ be a complete Riemannian manifold without boundary. We say that a smooth family of surfaces $x:\Sigma_t\hookrightarrow M$ evolves by the inverse mean curvature flow if the following evolution equation holds
\begin{align}
\frac{dx}{dt}=\frac{\nu}{H},
\label{IMCF smooth evolution}
\end{align}
where $H$ and $\nu$ are the mean curvature and the outward normal of $\Sigma_t$, respectively. The inverse mean curvature flow was originally considered by Geroch in \cite{geroch1973energy}, see also \cite{jang1977positive}, who discovered that the quasi-local Hawking mass given by 
$$
m_H(\Sigma):=\frac{|\Sigma|^{\frac12}}{(16\pi)^{\frac32}}\bigg(16\pi-\int_\Sigma H^2\text{d}vol\bigg)
$$
is non-decreasing along the flow of a smooth and connected surface $\Sigma_t$ in a manifold with non-negative scalar curvature. Of course, the Hawking mass is equal to the right-hand side of (\ref{nb RPI}) if $\Sigma$ is a minimal surface. Another remarkable property of the inverse mean curvature flow is the exponential growth of the area and the fact that the surfaces $\Sigma_t$ become more and more round while sweeping out the ambient space. In fact, one can show that $m_H(\Sigma_t)\to m_{ADM}$ if the flow exists for all times which suggests the obvious strategy to evolve $\Sigma$ in Theorem \ref{nb RPI thm} by the inverse mean curvature flow in order to prove (\ref{nb RPI}). Although some global existence results for the inverse mean curvature flow were subsequently proven, see \cite{huisken1988expansion,urbas1990expansion,gerhardt1990flow}, it is well-known that the flow generally develops singularities and cannot be continued past the singular time. In order to overcome these difficulties, Huisken and Ilmanen introduced the concept of a weak solution where the leaves of the flow are given by the level sets of a function $u\in C_{loc}^{0,1}(M)$ satisfying a certain variational principle. If the flow is smooth, $u$ satisfies the following quasi-linear degenerate elliptic partial differential equation
\begin{align}
\overline{\operatorname{div}}\bigg(\frac{\overline{\nabla} u}{|\overline{\nabla} u|}\bigg)=|\overline{\nabla} u|. \label{IMCF level set equation}
\end{align}
Here, the bar indicates that the respective geometric quantity is taken with respect to the ambient metric $g$. In order to construct a weak solution, Huisken and Ilmanen used a so-called elliptic regularization scheme, proved the existence of a smooth solution of 
\begin{align}
\overline{\operatorname{div}}\bigg(\frac{\overline{\nabla} u_{\epsilon}}{\sqrt{|\overline{\nabla }u_{\epsilon}|^2+\epsilon^2}}\bigg)=\sqrt{|\overline{\nabla }u_{\epsilon}|^2+\epsilon^2} \label{IMCF level set equation epsilon}
\end{align}
with $\epsilon>0$ and obtained the weak solution $u$ in the limit. Equation (\ref{IMCF level set equation epsilon}) has a natural interpretation as the level set formulation of a translating solution of the smooth inverse mean curvature flow
on a cylinder $M\times\mathbb{R}$ for which the evolution of the Hawking mass can be computed explicitly. Huisken and Ilmanen were able to obtain Geroch's monotonicity in the limit and subsequently established (\ref{nb RPI}). We also remark that there have been many other beautiful applications of the inverse mean curvature flow, see for instance \cite{perez2011nearly,lambert2017geometric,wei2018minkowski}. \\ In order to prove the Riemannian Penrose inequality for asymptotically flat half-spaces $(M,g)$, that is, Theorem \ref{main thm RPI}, it is natural to adapt the strategy in \cite{huisken2001inverse} by considering the free boundary inverse mean curvature flow. This flow is given by a smooth family of  free boundary surfaces $\Sigma_t$ which again evolve by (\ref{IMCF smooth evolution}) and it also enjoys many desirable properties such as the exponential area growth. Moreover, a calculation due to Marquardt, see \cite{marquardt2017weak}, shows that the modified Hawking mass given by
$$
\tilde m_H(\Sigma):=\frac{(2|\Sigma|)^{\frac12}}{(16\pi)^{\frac32}}\bigg(8\pi-\int_\Sigma H^2\text{d}vol\bigg)
$$
is non-decreasing along a smooth flow of a connected free boundary surface  provided the dominant energy condition $\operatorname{Sc}\geq 0$ and $H^{\partial M}\geq 0$ holds. Unfortunately, the free boundary condition seems to lead to even more singularities. In \cite{lambert2016inverse}, Lambert and Scheuer give an example of a rotationally symmetric, strictly convex free boundary surface supported on a compact ellipsoid which develops a finite-time singularity without approaching a minimal surface. In Section \ref{fb imcf general section}, we provide an example of a rotationally symmetric, strictly convex free boundary disc supported on the non-compact catenoid which also develops a finite-time singularity. Furthermore, in \cite{marquardt2017weak} Marquardt developed the theory of the weak free boundary inverse mean curvature flow but was only able to establish the existence of weak solutions for subsets of $\mathbb{R}^n$ whose boundary is a convex graph over $\mathbb{R}^{n-1}$ where  $n\geq 3$. This difficulty is caused by the presence of the source term on the right-hand side of (\ref{IMCF level set equation epsilon}) which makes it almost impossible to construct suitable subsolutions respecting the boundary condition without making very specific geometric assumptions. In order to overcome these difficulties, we drop the inconvenient source term and consider the modified equation  
\begin{align}
\overline{\operatorname{div}}\bigg(\frac{|\overline{\nabla} u_{\epsilon}|}{\sqrt{|\overline{\nabla }u_{\epsilon}|^2+\epsilon^2}}\bigg)=|\overline{\nabla }u_{\epsilon}|. \label{IMCF level set equation epsilon modified}
\end{align}
We are then able to establish the existence of a weak solution in a very general setting and we can even allow the initial surface to be closed.
The price which we have to pay for this analytic convenience is a slightly less regular solution and less geometric interpretability. In fact, we will see that the modified approximate solution can now be interpreted as a translating solution of the degenerate parabolic equation
\begin{align}
\frac{dx}{dt}=\frac{\nu}{\sqrt{\epsilon^2+H^2}}.
\end{align}
and neither the modified Hawking mass $\tilde m_H$ nor the area functional evolve in a desirable way under this flow. However, we discover a useful approximate growth inequality for the approximate Willmore energy
$$
\frac14\int_\Sigma(H\sqrt{\epsilon^2+H^2}+\epsilon^2\log(\sqrt{\epsilon^2+H^2}+H)-\epsilon^2\log(\epsilon)) \text{d}vol.
$$
The convergence to the weak solution is strong enough such that we can pass this approximate inequality to the limit and, combining it with the exponential area growth of the weak solution, obtain the monotonicity of the modified Hawking mass $\tilde m_H$ for the weak solution. Following the ideas of Huisken and Ilmanen, we are then able to use a weak blowdown argument to establish that $\tilde m_H(\Sigma_t)\to m_{ADM}$ as $t\to\infty$ which completes the proof of Theorem \ref{main thm RPI}. \\
The rest of this article is organized as follows. In Section \ref{asymptotically flat half spaces}, we recall some facts about asymptotically flat half-spaces and asymptotically flat support surfaces in dimension three and show that the region bounded by an asymptotically flat support surface is an asymptotically flat half-space. We also study the topology of exterior regions. In Section \ref{fb imcf general section}, we introduce the  free boundary inverse mean curvature flow, give an example of a finite-time singularity and explain the definitions of the weak free boundary inverse mean curvature flow given in \cite{marquardt2017weak,huisken2001inverse}. We also recall some useful properties of the flow. In Section \ref{new approximation scheme}, we prove a-priori estimates for solutions of (\ref{IMCF level set equation epsilon modified}) and pass these solutions to the limit to obtain a weak solution. In Section \ref{geroch monotonicity section}, we prove a growth inequality for the approximate Willmore energy in the smooth setting and obtain the monotonicity of the modified Hawking mass in the limit using ideas from \cite{huisken2001inverse}. Finally, in Section \ref{asymptotic behavior section}, we study the asymptotic behaviour of the flow and prove that its leaves become close to a large hemisphere in $C^{1,\alpha}$. In particular, we will see that the modified Hawking mass approaches the ADM-mass. We then proceed to prove the main results and the rigidity statement. \\
\textbf{Acknowledgements.} The author would like to thank his advisor Guofang Wang for suggesting the problem and for many helpful discussions. He also would like to thank Julian Scheuer and Ben Lambert for their interest in this work.

\section{Asymptotically flat half-spaces and asymptotically flat support surfaces}
\label{asymptotically flat half spaces}
In this section, we give a precise definition of asymptotically flat half-spaces and asymptotically flat support surfaces. We also recall the definitions for the ADM-mass, the extrinsic mass and the modified Hawking mass. We then proceed to show that every asymptotically flat support surface can be identified with an asymptotically flat half-space and we also prove that an exterior region is diffeomorphic to a half-space with finitely many solid free boundary discs and finitely many solid closed spheres removed. For the rest of this article, let us recall that given a Riemannian manifold $(M,g)$ with boundary $\partial M$ and a subset $U\subset M$, we call $\tilde \partial U:=\overline{\partial U\setminus \partial M }$ the interior boundary of $U$ and $\hat \partial U:=\partial U\cap \partial M$ the exterior boundary of $U$. Given a connected and compact hypersurface $\Sigma\subset M$, we say that $\Sigma$ is a free boundary surface (with respect to $\partial M$) if $\partial \Sigma\neq \emptyset$ and if $\Sigma$ meets $\partial M$ along its boundary at a contact angle of $\pi/2$. If $\partial \Sigma=\emptyset$ and $\Sigma\cap \partial M=\emptyset$, we say that $\Sigma$ is closed. 

\subsection{Asymptotically flat half-spaces}

Let $(M,g)$ be a three-dimensional, complete Riemannian manifold with non-compact boundary $\partial M$. We say that $(M,g)$ is an asymptotically flat half-space if there exists a compact set $\Omega\subset M$ such that every component $\hat M$ of $M-\Omega$, which we call an end of $M$, is diffeomorphic to $\mathbb{R}_+^3-B^3_1(0)$ and if the metric expressed in terms of this chart satisfies\footnote{Asymptotically flat half-spaces can also be defined with different decay rates, where for instance the right-hand side of (\ref{asymptotic behavior 0}) is replaced by $c|x|^{-\beta}$ for some $1/2<\beta\leq 1$, see \cite{almaraz2014positive}. However, we will only consider the case $\beta=1$.}  
\begin{align}
|g_{ij}(x)-{g_e}_{ij}|+|x||\partial_l g_{ij}(x)| \leq c|x|^{-1} \label{asymptotic behavior 0}
\end{align}
as well as 
\begin{align}
\operatorname{Rc}\geq -c|x|^{-2}. \label{asymptotic behavior 1}
\end{align}
Let $\operatorname{Sc}$ be the scalar curvature of $(M,g)$ and $H^{\partial M}$ be the mean curvature of $\partial M$. If $\operatorname{Sc}\in L^1(M)$ and $ H^{\partial M}\in L^1(\partial M)$, then each end $\hat M$ possesses a global non-negative invariant called the ADM-mass and defined by
\begin{align}
m_{ADM}:=\lim_{r\to\infty} \frac{1}{16\pi}\bigg(\int_{\mathbb{S}^3_r(0)\cap \mathbb{R}^3_+}(\partial_jg_{ij}-\partial_ig_{jj})\frac{x_i}{|x|}\text{d}vol_e+\int_{\partial D^2_r(0)\times\{0\}} g_{i3}\frac{x_i}{|x|}\text{d}vol_e\bigg ).
\label{ADM mass}
\end{align}
Here, $x$ denotes the position vector field of $\mathbb{R}^3$ and the subscript $e$ indicates that the geometric quantity is computed with respect to the Euclidean background metric $g_e$. Contrary to the introduction, we will not indicate  quantities related to an ambient manifold with non-empty  boundary by a tilde. If the dominant energy condition $\operatorname{Sc}, H^{\partial M}\geq 0$ holds, then $m_{ADM}\geq 0$ and if equality holds for one end then $(M,g)$ is isometric to the flat half-space $\mathbb{R}^3_+$. We refer to \cite{almaraz2014positive} for more details on asymptotically flat half-spaces and a derivation of the dominant energy condition from the Gibbons-Hawking-York action. An alternative proof of the fact $m_{ADM}\geq 0$ is also given in \cite{chai2018positive} using free boundary minimal surface techniques. \\
 A special family of asymptotically flat half-spaces is given by the Schwarzschild half-spaces of mass $m_{ADM}>0$. They are defined by $(M_S,g_S):=(\mathbb{R}^3_+,\phi^4 g_e)$ where $\phi(x)=1+(|x|m_{ADM})^{-1}$. $(M_S,g_S)$ has two ends and one may check that the scalar curvature  vanishes, that the ADM-mass of each end is in fact $m_{ADM}$ and that $(M_S,g_S)$ exhibits a $\mathbb{Z}_2-$symmetry given by a spherical inversion with respect to the free boundary minimal disk $\mathbb{S}^2_{m_{ADM}}(0)\cap \mathbb{R}^3_+$. Using a calibration argument, it can also be checked that this free boundary minimal disc is area minimizing and in fact the only free boundary minimal disc in $(M_S,g_S)$. 
 \\We say that $M'\subset M$ is an exterior region if $M'$ is non-compact and connected, $\tilde \partial M'$ consists of closed and free boundary minimal surfaces and if there are no other closed or free boundary minimal surfaces in $M'$. $\{x\in M_S| |x|\geq m_{ADM}\}$ is therefore an example of an exterior region. Finally, given a hypersurface $\Sigma\subset M$, we denote the modified (free boundary) Hawking mass by
\begin{align}
m_H(\Sigma):=\frac{|\Sigma|^{\frac12}}{2(8\pi)^{\frac32}}\bigg(8\pi-\int_\Sigma H^2\text{d}vol\bigg) \label{modified hm}
\end{align}
and the (closed) Hawking mass by
$$
\hat m_H(\Sigma):=\frac{|\Sigma|^{\frac12}}{(16\pi)^{\frac32}}\bigg(16\pi-\int_\Sigma H^2\text{d}vol\bigg).
$$
We emphasize that this notation differs from the one used in the introduction as we will mostly be concerned with the modified Hawking mass in the rest of this article. Finally, we agree that $\partial M$ is oriented by the outward unit normal $\mu$. If $ \Sigma$ is a free boundary surface with respect to $M$, then its outward co-normal coincides with $\mu$ and we will use the notation $\mu$ in both contexts.
\subsection{Asymptotically flat support surfaces}
The second class of spaces which we will consider was introduced by Volkmann in \cite{volkmann2015free}. He defined so-called asymptotically flat support surfaces of $\mathbb{R}^3$ as follows. Let $S\subset\mathbb{R}^3$ be a complete, connected, two-sided, non-compact smooth surface oriented by the normal $\mu$. We say that 
$S$ is an asymptotically flat support surface if there exists a compact set $\Omega\subset\mathbb{R}^3$ such that every component $\hat S$ of $S-\Omega$, called an end of $S$, can, after a rotation, be written as the graph of a function $\psi\in C^2(\mathbb{R}^2\setminus D_{R_0}^2(0))$  satisfying\footnote{We note that Volkmann additionally requires the condition $|\psi(x)|\leq c|x|^{-\beta}$ for some $0<\beta<1$ which is however, up to translations, implied by (\ref{graph asymptotic behavior}).}
\begin{align}
|x||\nabla_e \psi(x)|+|x|^{2}|\nabla_e^2\psi(x)|\leq c \label{graph asymptotic behavior}
\end{align} 
for some constant $c$ and some radius $R_0>0$. As before, each end can be assigned an invariant number called the exterior mass and defined by
$$
m_{ext}:=\lim_{r\to\infty}\frac{1}{2\pi}\int_{\partial D^2_r(0)} \frac{x_i}{|x|}\frac{\partial \psi}{\partial x^i}\text{d}{vol}_e.
$$
If $H^S\in L^1(S)$, where $H^S$ denotes the mean curvature of $S$, then it can be shown that the exterior mass is well-defined. Moreover, Volkmann showed that the dominant energy condition $H^S\geq 0$ implies that $m_{ext}\geq 0$ with equality if and only if $\Sigma$ is a flat plane.  This result is very reminiscent of the positive mass theorems for asymptotically flat manifolds and half-spaces. Moreover, Volkmann proved many interesting properties of asymptotically flat support surfaces such as a rigidity statement for the existence of non-compact, properly embedded free boundary minimal surfaces with respect to $S$ which is very similar to a result obtained by Carlotto in the context of asymptotically flat manifolds, see \cite{carlotto2016rigidity}. \\ A special family of asymptotically flat support surfaces is given by the catenoids with mass $m_{ext}>0$ defined to be $S_C:=\{(x_1,x_2,\pm\operatorname{arcosh}(\sqrt{m_{ext}^{-2}(x_1^2+x^2)})) | x_1^2+x_2^2\geq m^2_{ext}\}.$ Such a catenoid is rotationally symmetric, has two planar ends, exhibits a $\mathbb{Z}_2$ symmetry given by the reflection across the plane $\{x_3=0\}$ and the exterior mass of each end is given by $m_{ext}$. Moreover, it bounds the free boundary minimal disc $D^2_{m_{ext}}(0)\times\{0\}$. These properties are very reminiscent of the spatial Schwarzschild manifold with a geodesic plane through the origin removed and provide more justification for the special part which the catenoid plays in Huisken's Conjecture \ref{Huisken conjecture}. As before, we say that a connected, non-compact subset $S'\subset S$ is an exterior surface if there is a free boundary minimal surface $\Sigma\subset \mathbb{R}^3$ with respect to $S$ such that $\partial \Sigma=\partial S'$ and if there are no other free boundary minimal surfaces with respect to $S'$. An example of an exterior surface is therefore the intersection of the catenoid $S_C$ with the upper half-space. \\
Let us recall that the normal of $S$ is given by $\mu$ and we call this direction outward. Since $S$ is orientable, it follows that $\mathbb{R}^3\setminus S$ consists of two components and we denote the one $\mu$ is pointing out of by $M^S$. As had been mentioned earlier, we are able to show that the manifold $(M^S,g_e)$ is an asymptotically flat half-space and that the exterior mass is equal to the ADM-mass up to a multiplicative factor.  
\begin{lem}
	Let $S$ be an asymptotically flat support surface. Then $(M^S,g_e)$ is an asymptotically flat half-space. The extrinsic mass $m_{ext}$ is well-defined if and only if $m_{ADM}$ is well-defined and there holds $m_{ext}= 4m_{ADM}$. Moreover, if $S'$ is an exterior surface, then there is an exterior region $M'$ such that $S'=\hat\partial M'$. 
	\label{mass vs extrinsic mass}
\end{lem}
\begin{proof}
	The main difficulty is to construct a chart which becomes asymptotically Euclidean in every direction and allows us to easily compute the ADM-mass. 
	We may pick $R_0>0$ such that for every component $\hat S$ of $S\setminus B^3_{R_0}(0)$ there exists a function $\psi$ defined on $\mathbb{R}^2\setminus D^2_{R_0}(0)$ such that after a rotation $\hat S$  is given by the graph of $\psi$. Moreover, we may assume that the outward normal $\mu$ satisfies $\mu\cdot e_3<0$, where $\{e_1,e_2,e_3\}$ denote the canonical basis vectors of $\mathbb{R}^3$. Let $r$ be the radial function of $\mathbb{R}^3$ and $\hat r$ the radial function of $\mathbb{R}^2$.  By assumption, there holds
	\begin{align}
	\hat r|\nabla \psi|+\hat r^2|\nabla^2 \psi|^2\leq c \label{asymptotic behavior}
	\end{align}
	for some constant $c>0$. In order to define a chart at infinity, we use a foliation by spherical regions where the velocity of the respective barycentres is determined by the exterior mass. More precisely, let  
	\begin{align}
	\rho(\hat r):=\frac{1}{2\pi\hat r}\int_{\partial D^2_{\hat r}(0)}\psi\text{d}vol_e.
	\end{align}
	and notice that \begin{align}|\rho'(\hat r)|=\bigg|\frac{1}{2\pi\hat r}\int_{\partial D^2_{\hat r}(0)} \partial_{\hat r} \psi\text{d}vol_e\bigg|\leq c\tilde r^{-1}. \label{rho der est}\end{align}
	Then, we consider the sphere $\mathbb{S}^2_{r}(\rho(r)e_3)$, where $r\geq R_0$, and denote the polar angle by $\varphi$ and the azimuthal angle by $\theta$. It follows from (\ref{asymptotic behavior}) that, after increasing $R_0$ if necessary, given $\varphi$ and $r$, there exists precisely one $\theta=\tilde\zeta(r,\varphi)$ such that $\mathbb{S}^2_r(\rho(r)e_3)$ meets $S$ at the azimuthal angle $\tilde\zeta$. Let $\zeta=2\pi^{-1}\tilde\zeta$ and define the spherical chart at infinity via
	\begin{align}
	\Phi(r,\theta,\varphi):=r(\sin\varphi\sin(\zeta(r,\varphi)\theta),\cos\varphi\sin(\zeta(r,\varphi)\theta),\cos(\zeta(r,\varphi))\theta))+\rho(r)e_3.
	\end{align}
	Let $\mathbb{S}_r:=\mathbb{S}^2_r(\rho(r)e_3)$ and define $\nu_r=(\sin\phi\sin(\zeta\theta),\cos\phi\sin(\zeta\theta),\cos(\zeta\theta))$ to be the outward normal of $\mathbb{S}_r$. It is easy to see that $\Phi$ maps $\mathbb{R}^3_+\setminus B^3_{R_0}(0)\cong(R_0,\infty)\times[0,\pi/2]\times[0,2\pi)/\sim$ onto $M^S\setminus \overline{B^3_{R_0}(\rho(R_0) e_3)}$. Here, $\sim$ denotes  the equivalence relation given by $(\theta,\varphi)\sim(\theta',\varphi')$ if and only if $\theta=\theta'=0$.
	On the other hand, there holds
	\begin{align*}
	\partial_r \Phi=\nu_r+r\partial_r\zeta\theta (\sin\varphi\cos(\zeta\theta),\cos\varphi\cos(\zeta\theta),-\sin(\zeta\theta))+\rho'(r)e_3
	\end{align*} 
	and consequently
	$
	\partial_r\Phi\cdot \nu_r =1+\rho'(r)e_3\geq 1-cr^{-1}>0,
	$
	provided $R_0$ is chosen sufficiently large. It follows that the map $\Phi$ is also injective. We will show below that the metric induced by $\Phi$ is positive definite which then implies that $\Phi$ is in fact a diffeomorphism. Before computing the induced metric, we continue to study the function $\zeta$. Given $r$ and $\varphi$, there exists precisely one point $p(r,\varphi)$ such that the polar angle of $p(r,\varphi)$ with respect to $\mathbb{S}_r$ is also $\varphi$ and $p(r,\varphi)\in\mathbb{S}_r\cap\operatorname{graph}(\psi)$. This allows for the definition of the function $r'(r,\varphi):=\hat r(\operatorname{pr}_{\mathbb{R}^2}(p(r,\varphi)))$. Clearly, $\tilde\zeta(r,\varphi)=\arccos((\psi(r',\phi)-\rho(r))r^{-1})$. After translating the graph, we may assume that $\rho(R_0)=0$. It then follows from (\ref{rho der est}) that $|\rho(r)|\leq c\log (r)$. On the other hand, it follows from (\ref{asymptotic behavior}) that $|\psi(\hat r,\varphi)-\rho(\hat r)|\leq c$ and consequently $-c(\log(\hat r)+1) \leq \psi(\hat r,\varphi ) \leq c(\log(\hat r)+1)$. Clearly, there holds $r'\leq r$. On the other hand, if $r'\leq r/2$, then it follows from the triangle inequality that
	$$
	|p(r,\varphi)-\rho(r)e_3|\leq c+r'< r,
	$$
	provided $R_0$ is sufficiently large. This is of course a contradiction and we deduce $r/2\leq r'\leq r$. Thus,
	$$
	|\psi(r',\varphi)-\rho(r)|\leq |\psi(r',\varphi)-\rho(r')|+|\rho(r)-\rho(r')|\leq c+c\log 2\leq c.
	$$
	As $\tilde\zeta(r,\varphi)=\arccos((\psi(r',\varphi)-\rho(r))r^{-1})$ we then infer from Taylor's theorem that \begin{align}|\tilde\zeta-\frac{\pi}{2}|\leq\frac{c}{r}.\end{align} 
	It then follows that
	\begin{align}
	r'=r\sin(\zeta(r,\varphi)\frac{\pi}{2})\geq r\sqrt{1-\cos^2(\tilde\zeta(r,\varphi))}\geq r\sqrt{1-cr^{-2}}\geq r-cr^{-1}.                                
	\end{align}
	Differentiating the previous inequality it is then easy to see that
	\begin{align}
	\partial_r r'=1+\mathcal{O}(r^{-2}), \qquad \partial_{\varphi} r'=\mathcal{O}(r^{-2}), \qquad |\nabla_e^2 r'|=\mathcal{O}(r^{-3})
	. \label{cai est1}	
	\end{align}
	Using (\ref{asymptotic behavior}) once again, it now follows that
	\begin{align}
	|\psi(r,\phi)-\psi(r',\phi)|\leq cr^{-2}. \label{cai est2}
	\end{align}
	In order to facilitate the following computations, we introduce the map $\tilde \Phi$ defined by $\tilde \Phi(r,\theta,\varphi):=\Psi(\zeta\theta,\varphi)-\Psi(\theta,\varphi)$ where $\Psi(\theta,\phi)=(\sin\phi\sin\theta,\cos\phi\sin\theta,\cos\theta)$ is the standard parametrization of the unit sphere. Then,
	\begin{align}
	\Phi=\operatorname{id}+\rho(r)e_3+r\tilde \Phi.
	\end{align} 
	Moreover, for any multi index $\hat l$ we define $\tilde\Phi_{\hat l}(r,\theta,\phi):=\partial_{\hat l}\Psi\circ(\zeta\theta,\varphi)-\partial_{\hat l}\Psi\circ(\theta,\varphi)$. 
	From (\ref{cai est1}) and (\ref{cai est2}) we deduce that
	\begin{align}
	\zeta(r,\theta)=1-\frac{2}{\pi}\frac{\psi(r,\varphi)-\rho(r)}{r}+\mathcal{O}(r^{-2})=:1+\upsilon(r,\varphi)+\mathcal{O}(r^{-2}).
	\label{cai taylor 1}
	\end{align}
	We notice that \begin{align}\int_{\partial D^2_r(0)}\upsilon(r,\varphi)\text{d}vol_e=\int_{\partial D^2_r(0)} \partial_{l_1}\upsilon(r,\varphi)\text{d}vol_e=\int_{\partial D^2_r(0)} \partial_{l_1}\partial_{l_2}\upsilon(r,\phi)\text{d}vol_e=0 \label{cai zero mean} \end{align} for any $r$ and $l_1,l_2\in\{r,\varphi\}$. Moreover, It  follows from (\ref{cai taylor 1}) and Taylor's theorem that 
	\begin{align}
	\tilde \Phi=\partial_\theta\Psi\upsilon\theta+\mathcal{O}(r^{-2}), \qquad \partial_\theta \tilde\Phi=\partial_\theta\partial_\theta\Psi\upsilon\theta+\mathcal{O}(r^{-2}), \qquad \partial_i\varphi \tilde\Phi_\varphi=\partial_i\varphi \partial_\theta\partial_\varphi \Psi\upsilon\theta+\mathcal{O}(r^{-2}), 
	\label{tilde phi taylor}
	\end{align}
	with obvious generalizations to higher derivatives. Here, $\{\partial_i\}$ denotes the standard coordinate frame of $\mathbb{R}^3$. Again for ease of notation, we define $\tilde \Psi(r,\theta,\varphi):=\Psi(\zeta\theta,\phi)$ and for any multi-index $\hat l$ we define  $\tilde \Psi_{\hat l}:=(\partial_{\hat l} \Psi)\circ(\zeta\theta,\varphi)$. These terms can be expanded using (\ref{cai taylor 1}), too. We proceed to compute using (\ref{cai taylor 1}) as well as (\ref{tilde phi taylor})
	\begin{align}
	\partial_i \Phi=&e_i+\partial_i r \rho' e_3+\partial_i r \tilde \Phi+r\partial_i\varphi \tilde \Phi_\varphi+r\partial_i\theta \tilde \Phi_\theta +r \partial_i\theta(\zeta-1)\tilde \Psi_\theta+r\theta \partial_i\zeta \tilde \Phi_\theta \notag \\
	=&e_i+\partial_i r\rho'e_3+\partial_ir\upsilon\theta \partial_\theta \Psi +r\partial_i\varphi\upsilon\theta \partial_\theta\partial_{\varphi}\Psi+r\partial_i\theta\theta\upsilon \partial_\theta\partial_\theta \Psi+r\partial_i\theta\upsilon \partial_\theta \Psi \notag \\& +\mathcal{O}(r^{-2}) \label{cai first der}
	\\=&e_i+\mathcal{O}(r^{-1}).
	\end{align}
	In a similar fashion, we obtain
	\begin{align}
	\partial_{ij}\Phi=&\partial_i r\partial_jr\rho'' e_3+\partial_i\partial_jr\rho'e_3+\partial_i\partial_jr\tilde\Phi+(\partial_ir\partial_j\varphi+\partial_jr\partial_i\varphi)\tilde \Phi_\varphi+(\partial_ir\partial_j\theta+\partial_jr\partial_i\theta)\tilde\Phi_\theta\notag\\&+(\partial_ir\partial_j\theta+\partial_jr\partial_i\theta)(\zeta-1)\tilde \Psi_\theta
	+(\partial_ir\partial_j\zeta+\partial_jr\partial_i\zeta)\theta \tilde \Psi_\theta+r\partial_i\partial_j\varphi \tilde\Phi_\varphi+r\partial_i\varphi\partial_j\varphi \tilde \Phi_{\varphi\varphi}\notag\\&+r(\partial_i\varphi\partial_j\theta+\partial_j\varphi\partial_i\theta)\tilde\Phi_{\varphi\theta}+r(\zeta-1)(\partial_j\varphi\partial_i\theta+\partial_j\varphi\partial_i\theta)\tilde \Psi_{\varphi\theta}
	+r(\partial_i\varphi\partial_j\zeta+\partial_j\varphi\partial_i\zeta)\theta\tilde \Psi_{\varphi\theta}\notag\\&+r\partial_i\partial_j\theta \tilde\Phi_\theta\notag+r\partial_i\theta\partial_j\theta\tilde\Phi_{\theta\theta}+r(\zeta-1)\partial_i\theta\partial_j\theta\tilde \Psi_{\theta\theta}+\theta(\partial_i\theta\partial_j\zeta+\partial_j\theta\partial_i\zeta)\tilde \Psi_{\theta\theta}\\&+r\partial_i\partial_j\theta(\zeta-1)\tilde \Psi_\theta\notag+r(\partial_i\theta\partial_j\zeta+\partial_j\theta\partial_i\zeta)\tilde \Psi_\theta+r\partial_i\theta(\zeta-1)\tilde \Psi_{\theta\theta}(\theta\partial_j\zeta+\partial_j\theta\zeta)\\&+r\theta\partial_i\partial_j\zeta\tilde \Phi_\theta\notag	+r\theta \partial_i\zeta \tilde \Psi_{\theta\theta}(\partial_j\theta\zeta+\theta\partial_j\zeta)
	\notag	\\=&:\partial_ir\partial_jr\rho'' e_3+\partial_i\partial_jr\rho'e_3+Q_{ij}(\theta,\varphi,r)\upsilon+Q^{l_1}_{ij}(\theta,\varphi,r)\partial_{l_1}\upsilon+Q^{l_1,l_2}_{ij}(\theta,\varphi,r)\partial_{l_1}\partial_{l_2}\upsilon \label{cai second der}
	\\=&\mathcal{O}(r^{-2}). 
	\end{align}
	Here, $Q_{ij},Q^{l_1}_{ij},Q^{l_1,l_2}_{ij}$ are collections of terms of order $\mathcal{O}(r^{-2})$ which are symmetric in $\varphi$. As before, $l_1, l_2$ are elements of  $\{r,\varphi\}$.  In the last inequality, we used Taylor's theorem again and the fact  $|\partial_i\upsilon|+|\partial_{i}\partial_j\upsilon|r\leq cr^{-2}$ which one readily verifies.  In this particular chart, there holds  ${g}_{ij}=g_e(\partial_i\Phi,\partial_j\Phi)$ and it follows that $(M^S,g_e)$ is an asymptotically flat half-space. We proceed to compute the mass. From (\ref{cai first der}) and (\ref{cai zero mean})   it follows that for $i\in\{1,2\}$ there holds ${g}_{i3}=\partial_ir\rho'+Q_i+\mathcal{O}(r^{-2})$ where $Q_i$ is a quantity satisfying 
	$$
	\int_{\partial D^2_r(0)\times\{0\}} x_iQ_i\text{d}vol_e=0.
	$$
	Thus, we obtain
	\begin{align}
	\int_{\partial D^2_r(0)\times\{0\}} {g}_{i3}\frac{x_i}{|x|}\text{d}vol_e=2\pi r \rho'+\mathcal{O}(r^{-1}).
	\label{cai boundary}
	\end{align}
	On the other hand, we may deduce from (\ref{cai first der}), (\ref{cai second der}), (\ref{cai zero mean}) and Fubini's theorem that
	\begin{align*}
	(\partial_j{g}_{ij}-\partial_i{g}_{jj})\partial_ir&=\partial_ir\partial_i\Phi\cdot\overline{\Delta}_e\Phi-\partial_ir\partial_{ij}\Phi\cdot\partial_j\Phi\\&=\partial_r\cdot (\partial_ir\partial_ir\rho''e_3+\partial_i\partial_ir\rho'e_3)-\partial_ir(\partial_ir\partial_3r\rho''+\partial_i\partial_3r\rho')+Q+\mathcal{O}(r^{-3})\\&=\frac{2\partial_3r}{r}\rho'+Q+\mathcal{O}(r^{-3}),
	\end{align*}
	where $Q$ is a collection of terms such that $\int_{\mathbb{S}^2_r(0)\cap \mathbb{R}^3_+(0)}Q\text{d}vol_e=0$. Thus, it follows that
	\begin{align}
	\int_{\mathbb{S}^3_r(0)\cap\mathbb{R}^3_+(0)}(\partial_jg_{ij}-\partial_ig_{jj})\frac{x_i}{|x|}\text{d}vol_e= 2\pi r\rho'+\mathcal{O}(r^{-1}).
	\end{align}
	Combining this with (\ref{cai boundary}) we have
	$$
	m_{ADM}=\frac{4\pi}{16\pi}\lim_{r\to\infty}r\rho'(r)=\frac{1}{4}\lim_{r\to\infty}\frac{1}{2\pi}\int_{\partial D_r(0)} \frac{x_i}{|x|}\partial_i \psi\text{d}vol_e= \frac14m_{ext}
	$$
	as claimed. The remaining assertions can be verified easily.
\end{proof}
\begin{rema}
	The extension of $S$ to an asymptotically flat half-space satisfying the dominant energy condition is of course not unique. In light of Conjecture \ref{Huisken conjecture} it would be interesting to know if a judicious injection of ambient curvature can produce an asymptotically flat half-space such that $m_{ADM}\leq 2^{-5/2} m_{ext}$.
\end{rema}
\subsection{The topology of an exterior region}

In this subsection, we study the topology of an exterior region. It will be essential for the monotonicity calculation in Section \ref{geroch monotonicity section} that the weak free boundary inverse mean curvature starting at a connected free boundary surface remains connected and does not detach from $\partial M$. In order to show this, we now establish the fact that an exterior region is simply connected and has a connected boundary. To this end, we follow the argument in Section 4 of \cite{huisken2001inverse}. Let $\Omega_1$ be the closure of the union of all smooth, immersed free boundary and closed minimal surfaces. $\Omega_1$ is a compact set since the region near each infinity is foliated by strictly mean convex hemispheres meeting $\partial M$ at an acute angle. The \textit{trapped region} $\Omega$ is then defined to be the union of $\Omega_1$ and all compact components of $M\setminus \Omega_1$ and we note that $\Omega$ is a compact set, too. We then define $M'$ to be the metric completion of a component of $M\setminus \Omega$\footnote{As in \cite{huisken2001inverse}, we take the metric completion rather than the closure as $\Omega_1$ might contain non-separating minimal surfaces.}. 
\begin{lem}Let $(M,g)$ be an asymptotically flat half-space of non-negative scalar curvature and $\partial M$ be mean convex and connected. Then $M'$ is an exterior region and $\tilde \partial M'$ consists of finitely many free boundary minimal discs and closed minimal spheres.  $M'$ is simply connected, does not contain any other immersed minimal surfaces (free boundary or closed) and has the topology of a half-space with finitely many solid balls removed. Finally, the closed components of $\tilde \partial M'$ minimize area in their homology class while the free boundary components minimize area in the homotopy class of their boundary curve with respect to $\hat \partial M'$.
	\label{topological structure}
\end{lem}	
\begin{proof}
	The proof is very similar to \cite{huisken2001inverse} and we only sketch the details. First, we fix an end $\hat M$ and consider a large boundary circle $\partial D^2_R(0)\times\{0\}$ in the asymptotic region of $\hat M$. If the circle is contractible, then it follows that the corresponding component of $\partial M$ has the topology of $\mathbb{R}^2$. If not, then, using the fact that $\partial M$ is mean convex, we can minimize area in the boundary homotopy class of    $\partial D^2_R(0)\times\{0\}$ and it follows from \cite{meeks1980topology,meeks1982existence} that there is an embedded free boundary minimal disc $\Sigma_0$ such that $\partial \Sigma_0$ lies in the homotopy class of $\partial D^2_R(0)\times\{0\}$. In the former case we take $\Sigma_0=\emptyset$ and denote the annulus (or the disc) bounded by $ \partial \Sigma_0 $ and $\partial D^2_R(0)\times\{0\}$ by $S$. 	
	In either case, if the component of $M-\Sigma_0$  corresponding to $\hat M$ has more than one end, then we can separate these ends by a closed minimal surface which is obtained by minimizing area in the homology class of the sphere $(\mathbb{S}^2_R(0)\cap \mathbb{R}^3_+)\cup S\cup \Sigma_0$. We observe that this sphere also acts as a barrier.\\ Next, as in \cite{huisken2001inverse}, we show that the set $\Omega$ can be generated by a smaller family of minimal surfaces. Namely, we let $c_0$ be the area of a strictly mean convex hemisphere in the asymptotic region of $\hat M$ that meets $\partial M$ at an acute angle and define $\mathcal{E}$ to be the family of all smooth, stable and embedded minimal surfaces, closed or free boundary, with area less or equal than $c_0$. Let $\Sigma_1$ be one of the surfaces in the definition of $\Omega_1$ and suppose that $\Sigma_1$ is not an element of $\mathcal{E}$. We then minimize area amongst all surfaces that shield $\Sigma_1$ from infinity in the chosen end. Clearly, the minimizer $\tilde \Sigma$ has area less or equal than $c_0$ and using a cut and paste comparison argument, see \cite{meeks1982classical}, we can show that $\tilde \Sigma$ cannot touch possible self-intersections of $\Sigma_1$. The Allard-type regularity result \cite{gruter1986allard,sternberg1991c1} then implies that $\tilde \Sigma$ is of class $C^{1,\alpha}$ up to the boundary and consequently a smooth minimal surface which is either closed or free boundary with respect to $\partial M$, see also \cite{gruter1987optimal}. Moreover, it follows that $\tilde \Sigma$ is stable and embedded. In any case, the strong maximum principle implies that $\tilde \Sigma$ and $\Sigma_1$ do not touch and consequently $\Sigma_1\cap \tilde \partial \Omega=\emptyset$. Hence, we may replace $\Sigma_1$ by $\tilde \Sigma\in\mathcal{E}$ in order to generate $\Omega$. \\Let $\Sigma_1,\Sigma_2\in \mathcal{E}$ be surfaces that meet $\tilde \partial \Omega$ and suppose that $\Sigma_1\cap \Sigma_2\neq\emptyset$. If $\Sigma_2$ is not a free boundary minimal surface it follows from the strict maximum principle that $\Sigma_1$ and $\Sigma_2$ must intersect transversally. If both are free boundary minimal surfaces, then they either intersect transversally, too, or meet at a boundary point. In any case, $\Sigma_1\cup \Sigma_2$ is not area minimizing and it follows from the strict maximum principle that the  minimizing hull $\tilde \Sigma$ of $\Sigma_1\cup \Sigma_2$, which is a smooth closed or free boundary minimal surface, is distinct from both $\Sigma_1,\Sigma_2$. It thus follows that neither $\Sigma_1$ nor $\Sigma_2$ intersect $\tilde{\partial} \Omega$, a contradiction. Even more, there is a constant $\epsilon>0$ such that $\operatorname{dist}_g(\Sigma_1,\Sigma_2)>\epsilon$ for any $\Sigma_1,\Sigma_2\in\mathcal{E}$ which meet $\tilde{\partial} \Omega$. For otherwise, there exist minimal surfaces $ \Sigma_3,\Sigma_i\in\mathcal{E}$ meeting $\tilde \partial \Omega$ such that $\operatorname{dist}( \Sigma_3,\Sigma_i)\to 0$. According to \cite{schoen1975curvature} for closed minimal surfaces and \cite{guang2016curvature,li2016min}   for embedded free boundary minimal surfaces, the set $\mathcal{E}$ is compact with respect to the $C^2-$topology. After passing to a subsequence we may assume that $\Sigma_i\to\tilde \Sigma$. One possibility is that $\tilde \Sigma$ touches $\Sigma_3$ at an interior point and it follows that $\Sigma_3=\tilde \Sigma$. But since $\Sigma_3$ and $\Sigma_i$ do not touch  this implies that either $\Sigma_3$ or $\Sigma_i$ do not touch $\tilde \partial \Omega$ for $i$ large, a contradiction. The other possibility is that $\tilde \Sigma$ touches $\Sigma_3$ at a boundary point. But then we can find, as before, another free boundary minimal surface $\hat \Sigma$ such that $\partial \hat \Sigma$ is homotopic to $\partial \tilde \Sigma \cup \partial \Sigma_3$ and the strong maximum principle implies that $ \Sigma_3, \tilde \Sigma$ and $\hat \Sigma$ do not touch. It follows that neither $\Sigma$ nor $\Sigma_i$ touch $\tilde \partial \Omega$ for $i$ large, which is again a contradiction. Hence, it follows that $\tilde \partial \Omega$ consists of finitely many minimal surfaces which are either closed or free boundary and that the non-compact components of $M- \Omega $ are free of other minimal surfaces. \\A straightforward adaptation of Proposition 2.2 in \cite{ros2008stability} to manifolds of non-negative scalar curvature implies that every free boundary minimal surface touching $\tilde\partial\Omega$ must be a topological disc. Since the end $\hat M$ is modelled on a half-space, it follows that the boundary of the corresponding component $M'$ of $M-\Omega$ has the topology of $\mathbb{R}^2$. Next, $M'$ must be simply connected because otherwise the universal cover $\pi:\check{M}\to M'$ of $M'$ has at least two ends which can be separated by either a free boundary minimal surface or a closed minimal surface $\Sigma_4$. According to the maximum principle,  $\Sigma_4$ does not intersect $\pi^{-1}(\tilde \partial M')$ and consequently $\pi(\Sigma_4)$ is an immersed minimal surface in $M'$ (closed or free boundary). This is a contradiction to the fact that $M'$ is free of minimal surfaces. It then follows by the loop theorem, see \cite{meeks1980topology,galloway1993topology}, that every closed component of $\tilde \partial \Omega$ must be a sphere. We can then argue as in \cite{huisken2001inverse,meeks1982embedded} to show that $M'$ has the claimed topology. Finally, in order to show that the components of $\tilde \partial M'$ are area minimizing, we minimize area in their respective homology or boundary homotopy class and since $M'$ is free of other minimal surfaces, the minimizer must be a subset of $\tilde \partial M'$.  
\end{proof}
\begin{rema}
	Contrary to the situation for asymptotically flat manifolds without boundary, see Lemma 4.1 in \cite{huisken2001inverse}, a curvature hypothesis is clearly needed for the statement of the previous lemma to be true. For instance,  one may take the the catenoid $S_C$ with mass one and consider the space $M:=((M^{S_C}\cap\mathbb{R}^3_+)\cup( {B^3_1(0)}\cap\mathbb{R}^3_-))\setminus C$, where $C$ is a thin, rotationally symmetric tentacle contained in the cylinder $D^2_{1/2}(0)\times \mathbb{R}$ such that $C\setminus\{x_3\geq -1/2\}={D}_{1/4}^2(0)\times[-1/2,0)\cup B^3_{1/4}(0)$. $\partial M$ is not mean convex and the interior boundary of the exterior region $M\cap \mathbb{R}^3_+$ is given by a flat annulus.
\end{rema}
\section{The free boundary inverse mean curvature flow}
\label{fb imcf general section}
In this section, we discuss the free boundary inverse mean curvature flow. Both in the weak and strong setting, it was introduced by Marquardt in \cite{marquardt2012inverse} based on ideas of Huisken, Ilmanen, Gerhard and Urbas, see \cite{huisken2001inverse,urbas1990expansion,gerhardt1990flow}. Many of the following concepts will make sense in a more general setting than the one provided by three-dimensional asymptotically flat half-spaces and for now, we will solely assume that $(M,g)$ is an $n-$dimensional, complete and connected Riemannian manifold with non-empty boundary $\partial M$. 
\subsection{The smooth case and examples of finite-time singularities}
Let $\Sigma\subset M$ be a possibly disconnected hypersurface such that each component is either closed or a free boundary surface and suppose that $\Sigma$ is oriented by the outward normal $\nu$ with corresponding mean curvature $H$. We say that $\Sigma_t$ flows by the free boundary inverse mean curvature flow with initial data $\Sigma$ if there is  a number $T>0$ and a  smooth family of embeddings $x:\Sigma_t \hookrightarrow M $, $0\leq t\leq T$, such that $\Sigma_0=\Sigma$, each component of $\Sigma_t$ is either closed or a free boundary surface and the following evolution equation holds 
\begin{align}
\frac{dx}{dt}=\frac{\nu}{H}. \label{imcf evolution equation}
\end{align}	
If $\Sigma$ is strictly mean convex, this flow is parabolic from which one may deduce short-time existence, see \cite{marquardt2012inverse}. The flow has the property that the area evolves exponentially, that is, $|\Sigma_t|=e^t|\Sigma|$. Moreover, it was observed by Marquardt in \cite{marquardt2017weak} that the modified Hawking mass (\ref{modified hm})
is non-decreasing along the flow provided $(M,g)$ satisfies the dominant energy condition $\operatorname{Sc}, H^{\partial M}\geq 0$ and  $\Sigma_t$ is a connected free boundary surface\footnote{At the end of this section, we will see that no reasonable monotonicity can be expected for a flow starting at a closed surface.}. As discussed in the introduction, this flow is therefore a useful tool to prove Theorem \ref{main thm RPI}.
\\ There are some known examples where the flow starting at a connected free boundary surface remains smooth for all times, see \cite{lambert2016inverse} and \cite{marquardt2013inverse}. However, singularities seem to develop in many other cases. For instance, in \cite{lambert2016inverse} Lambert and Scheuer give an example where $\partial M$ is a rotationally symmetric ellipsoid and a strictly convex, rotationally symmetric free boundary surface develops a finite-time singularity. We now also present an example of a finite-time singularity when $\partial M$ is non-compact.\\ Let $(M^{S_C},g_e)$ be the Euclidean domain bounded by the catenoid $S_C$ with exterior mass $m_{ext}=1$  and let $\tilde \Sigma$ be a large sphere which meets $S_C$ at an acute angle. The flat disc $D^2_1(0)$ is a free boundary surface with respect to $S_C$ and we can slightly move this disc upwards to obtain a strictly convex, rotationally symmetric free boundary surface $\Sigma$. Moreover, we may arrange that $\sup_{\Sigma}H^\Sigma<H^{\tilde \Sigma}=:H_0$. Now, suppose that the free boundary inverse mean curvature flow starting at $\Sigma$ and denoted by $\Sigma_t$ does not develop a finite-time singularity. It is easy to see that $\Sigma_t$ stays rotationally symmetric and graphical with respect to its projection onto $\mathbb{R}^2\times\{0\}$. Since the area of $\Sigma_t$ grows exponentially, it follows that there exists a first time $T>0$ when $\Sigma_T$ touches $\tilde \Sigma$. As $\tilde \Sigma$ meets $S_C$ at an acute angle, this must happen at an interior point $p\in M^{S_C}-S_C$. It follows that $H^{\Sigma_T}(p)\geq H^{\tilde \Sigma}(p)=H_0$. On the other hand, along the free boundary inverse mean curvature flow, the mean curvature satisfies the following partial differential evolution equation
\begin{equation}
\begin{aligned}
\partial_t H&=\frac{\Delta H}{H^2}-2\frac{|\nabla H|^2}{H^3}-\frac{|A|^2}{H^3}- \frac{\operatorname{Rc(\nu,\nu)}}{H} \text{ in } \Sigma_t, \\
\partial_\mu H&=HA^{\partial M}(\nu,\nu) \text{ on } \partial \Sigma_t,
\end{aligned}
\end{equation}
where $\mu$ is the outward co-normal of $\partial \Sigma_t$ and thus the outward normal of $\partial M$. By rotational symmetry and since $\partial M$ is a catenoid, it follows that $A^{\partial M}(\nu,\nu)<0$. Since $\operatorname{Rc}=0$ in our example, it follows from the maximum principle that $\operatorname{sup}_{\Sigma_t} H^{\Sigma_t}$ must be decreasing, a contradiction. \\This example can also be modified such that $\partial M$ is strictly convex. What is more, a flow starting from a closed surface seems to necessarily develop a finite-time singularity as the exponential area growth eventually forces the evolving surface to touch the boundary after which the flow cannot be continued in a classical sense. In order to overcome these singularities, we have to use a weak notion of the inverse mean curvature flow. 
\subsection{The weak formulation}
We now describe the weak formulation of the free boundary inverse mean curvature flow which was developed in \cite{marquardt2017weak,huisken2001inverse}. The weak inverse mean curvature flow is a level set flow. This compels the flow to be unidirectional and allows connected components to merge, fatten instantaneously and to change topology while preserving many of the useful properties of the smooth flow such as the exponential area growth and the monotonicity of the modified Hawking mass discussed in the previous subsection. \\ To begin with, let us assume that $\Sigma_t$ is a smooth flow and that there exists a smooth function $u:M\to\mathbb{R}$ such that $\Sigma_t=\{p\in M|u(p)=t\}$. Then the outward normal and mean curvature of $\Sigma_t$ are given by $\nu=\overline{\nabla}u/|\overline{\nabla}u|$ and   $H=\overline{\operatorname{div}}(\overline{\nabla}u/|\overline{\nabla}u|)$, respectively. On the other hand, there holds $u(x)=t$ if $x$ is the embedding of $\Sigma_t$ into $M$. Differentiating in time and using the flow equation (\ref{imcf evolution equation}) we thus find
\begin{align}
\overline{\operatorname{div}}\bigg(\frac{\overline{\nabla}u}{|\overline{\nabla}u|}\bigg)=|\overline{\nabla}u|.  \label{level set equation}
\end{align}
Moreover, if $\partial \Sigma_t$ is non-empty, then $\Sigma_t$ meets $\partial M$ orthogonally along its boundary and it follows that the outward normal $\mu$ of $\partial M$ is orthogonal to $\overline{\nabla }u$, that is,
\begin{align}
\partial_\mu u=0 \text{ on } \partial M. \label{level set bc}
\end{align}
Equation (\ref{level set equation}) is singular if $\overline{\nabla} u=0$ and degenerate in direction $\overline{\nabla} u$ otherwise. In fact, we have seen that solutions to the free boundary inverse mean curvature flow cannot remain smooth in general. In order to still make  sense of the level set formulation, Huisken and Ilmanen observed that (\ref{level set equation}) is the Euler-Lagrange equation of the following functional
\begin{align}
J_u(v):=\int_{\Omega} (|\overline{\nabla} v|+v|\overline{\nabla }u|)\text{d}vol.
\end{align}
Equation (\ref{level set bc}) is the natural boundary condition for a minimizer of this functional and we are thus lead to the following definition.
\begin{defi}
	Let	$u\in C^{0,1}_{loc}(M)$ and $U\subset M$ be open. We say that $u$ is a weak solution of the free boundary inverse mean curvature flow in $U$ if for any $v\in C_{loc}^{0,1}(M)$ such that $\{v\neq u\}\subset \subset U$ there holds 
	$$
	\int_{\{v\neq u\}} (|\overline{\nabla} u|+u|\overline{\nabla} u|)\text{d}vol\leq  \int_{\{v\neq u\}} (|\overline{\nabla} v|+v|\overline{\nabla} u|)\text{d}vol.
	$$
	Moreover, given an open set $E_0\subset U$, we say that $u$ is a weak solution of the free boundary inverse mean curvature flow with initial data $E_0$ if $\{u< 0\}={E_0}$ and if $u$ is a weak solution in $U-\overline {E_0}$.
	\label{weak solution}
\end{defi}
Evidently, one may obtain a trivial solution by setting $u=0$ outside of $E_0$ and one can also check that if $u$ is a solution, so is $u_{t}=\min\{t,u\}$ for every $t>0$. In order to exclude these somewhat unreasonable solutions, we will therefore require the sublevel sets $\{u<t\}$ to be precompact. If $(M,g)$ is an exterior region of an asymptotically flat half-space, this is equivalent to requiring $u$ to be proper in the sense that $u(p)\to\infty$ as $p\to\infty$. We will see in the next subsection that solutions with precompact sublevel sets enjoy various desirable properties. 
\\ For the rest of this article, we make the following definitions
$$
E_t:=\{u<t\},\qquad E_t^+= \operatorname{int}\{u\leq t\},\qquad \Sigma_t=\tilde \partial E_t, \qquad \Sigma_t^+:=\tilde \partial E_t^+.
$$
Whenever, $\Sigma_t\neq\Sigma_t^+$, the flow jumps over a positive volume and since $\{E_t|t\in[0,\infty)\}$ is a nested family, this can only happen for countably many $t\in Z$, where $Z$ denotes the set of these jump times. The sets $\Sigma_t$ allow for a more classical interpretation of the weak flow as we shall now see.
\subsection{Properties of weak solutions} 
In this subsection, we recall some concepts from geometric measure theory and collect properties of weak solutions which we will need in the sequel. All of the results have been, up to some minor modifications,  proven in   \cite{huisken2001inverse,marquardt2012inverse,marquardt2017weak} and the reader is referred to these references for an excellent exposition to this subject including details and heuristics. The regularity theory for weak solutions is based on three ingredients. First, one may check that the sublevel sets $E_t$ minimize the functional
\begin{align}
F\mapsto |\tilde \partial ^*F\cap \Omega|-\int_{\Omega\cap F} |\overline{\nabla} u|\text{d}{vol}
\label{variational principle}
\end{align}
amongst all finite perimeter sets $F$ such that $F\Delta E_t\subset\subset U\setminus \overline{E}_0$. Here, $\Omega$ is any compact subset of $U$ containing $ F \Delta E_t$. Second, the variational principle for $u$ implies that both $\Sigma_t$ and $\Sigma^+_t$ possess the generalized mean curvature $H=|\overline{\nabla} u|$ and are weakly orthogonal to $\partial M$ provided $\Sigma_t\cap \partial M\neq\emptyset$. We will explain these terms below. Third, as we will show in the next section, $|\overline{\nabla }u|$ is uniformly bounded in $M$. This implies in particular that the generalized mean curvature is in $L^\infty$. \\ We say that a set $\tilde \Sigma$ of locally finite $(n-1)-$dimensional Haussdorff measure which is a $C^1-$surface outside of a zero set with respect to the $(n-1)-$dimensional Haussdorf has  generalized mean curvature $H$ in $L^p$ and is weakly orthogonal to $\partial M$ if $H\in L^p(\tilde \Sigma)$ and 
\begin{align}
\int_{\tilde \Sigma} (\operatorname{div}X-H g(X,\nu))\text{d}vol=0
\end{align}
for any smooth vector field $X$ in $M$ that is tangential to $\partial M$ along $\tilde \Sigma\cap \partial M$ almost everywhere. Here, $\nu$ is a normal vector orienting $\tilde \Sigma$.  We note that if $\tilde \Sigma$ is a smooth closed or free boundary surface then the first variational formula for the area implies
\begin{align}
\int_{\tilde \Sigma} (\operatorname{div}X-H g(X,\nu))\text{d}vol=\int_{\partial \tilde \Sigma} g(X,\mu)\text{d}vol
\end{align}
for any vector field $X$ where $H$ is the ordinary mean curvature and $\mu$ the outward co-normal of $\tilde \Sigma$. The right-hand side is clearly zero if $X$ is  tangential to $\partial M$ along $\partial \tilde \Sigma$. \begin{rema}
	If $\Sigma_i$ is a sequence of closed or free boundary surfaces of class $C^1$ with generalized mean curvature $H^{\Sigma_i}$ such that $\limsup_{i\to\infty}|H^{\Sigma_i}|_{L^{\infty}(\Sigma_i)}<c$ and $\Sigma_i\to\tilde \Sigma$ in $C^1$, then it follows from the Riesz representation theorem that $\tilde \Sigma$ has generalized mean curvature $H$ in $L^\infty$ and is weakly orthogonal to $\partial M$. Moreover, $H^{\Sigma_i}\rightharpoonup H$ in $L^\infty$, $|H|_{L^\infty(\Sigma)}\leq \liminf_{i\to\infty}|H^{\Sigma_i}|_{L^{\infty}(\Sigma_i)}$
	and
	$$
	\int_{\tilde \Sigma} \zeta H^2\text{d}vol \leq \liminf_{i\to\infty} \int_{\Sigma_i} \zeta ({H^{\Sigma_i}})^2 \text{d} vol
	$$
	for every non-negative $\zeta\in C^0_c(M)$.
	\label{weak mean curvature rema}
\end{rema}We will also need the concept of a minimizing hull. We say that $\hat E$ is a minimizing hull in $U$ if for any $ F\subset U$ containing $\hat E$ such that $F\setminus \hat E\subset\subset U$ and any compact set $\Omega$ which contains $F\setminus \hat E$ there holds
$$
|\Omega \cap \tilde \partial^* \hat E|\leq |\Omega \cap \tilde \partial^*F|.
$$
Here, $\partial^*$ denotes the reduced boundary.  $\hat E$ is called a strictly minimizing hull if equality holds precisely if $F=\hat E$ almost everywhere. It can be shown that any measurable set $E$ is contained in a unique, open, strictly minimizing hull minimizing the interior perimeter amongst all locally finite perimeter sets  containing $E$. We will call this set the strictly minimizing hull of $E$ and denote it by $E'$. If $\tilde \partial E$ has generalized mean curvature $H^{\tilde\partial E}$, then one can show that $\tilde \partial E'$ has generalized mean curvature $H^{\tilde \partial E'}$ and
\begin{align}
H^{\tilde \partial E'}=H^{\tilde \partial E}\geq 0 \text{ on } \tilde\partial E\cap \tilde \partial E' \text{ and } H^{\tilde \partial E'}=0 \text{ on  } \partial E'\setminus\partial E. \label{mc minimizing hull}
\end{align} The next lemma summarizes the properties of weak solutions relevant to this article.
\begin{lem}
	Let $U\subset M$ be an open subset of an at most seven dimensional manifold $M$ with boundary, $E_0\subset U$ be an open minimizing hull and  $\Sigma=\tilde \partial E_0$. Assume that $u\in C_{loc}^{0,1}(U)$ is a weak solution with initial data $E_0$ and that all sublevel sets $E_t$ are precompact. Let $\alpha\in(0,1/2)$. The following holds.
	\begin{itemize}
		\item \textbf{Uniqueness.} The weak solution is unique amongst all solutions with precompact level sets.
		\item\textbf{Regularity.} Let $t>0$. The components of both $\Sigma_t$ and $\Sigma_t^+$ are closed or free boundary surfaces of class $C^{1,\alpha}$ with generalized mean curvature and outward unit normal given by $H=|\overline{\nabla}u|$ and $\nu=\overline{\nabla}u/|\overline{\nabla} u|$ almost everywhere, respectively. 
		\item\textbf{Estimates.} The $C^{1,\alpha}-$estimates of $\Sigma_t,\Sigma_t^+$ depend on $\alpha$, $|\overline{\nabla} u|_{L^\infty(M)}$, the $C^1-$data of $g$, the $C^2-$data of $\partial M$, the distance to $\partial U$ and the distance to $\Sigma$. If $\Sigma$ is of class $C^{1,\alpha}$, then the last dependency can be replaced by a dependency on the $C^{1,\alpha}-$estimates of $\Sigma$. 
		\item\textbf{Convergence.} Given $t>0$, there holds $\Sigma_{t'}\to \Sigma_t$ in $C^{1,\alpha}$ as $t'\nearrow t$ and $\Sigma_{t'}\to \Sigma^+_t$ in $C^{1,\alpha}$ as $t'\searrow t$. If $\Sigma$ is of class $C^{1,\alpha}$, one may also choose $t=0$. 
		\item\textbf{Minimizing Hull Property.} Let $t\geq 0$. $E_t$ is a minimizing hull in $U$ and $E^+_t$ is a strictly minimizing hull in $U$. Moreover, $E_t'=E_t^+$ and $|\Sigma_t|=|\Sigma_t^+|$. 
		\item\textbf{Exponential Area Growth.} There holds $|\Sigma_t|=e^t|\Sigma|$. 
		\item\textbf{Smooth start.} If $\Sigma$ is smooth and strictly mean convex, then there is a small constant $\epsilon>0$ such that $\Sigma_t$ is given by the leave of the smooth free boundary inverse mean curvature flow starting at $\Sigma$ at time $t$ for all $0<t<\epsilon$.
	\end{itemize}
	\label{lemma properties of weak solutions}
\end{lem}
\begin{proof}
	These results were proven in   \cite[Section 1 and 2]{huisken2001inverse}, \cite[Section 4 and 5]{marquardt2017weak} and \cite[Section 4.4]{marquardt2012inverse}. The interior regularity is based on the results proven in \cite{massari1974esistenza} whereas the regularity up to the boundary $\partial M$ relies on \cite{gruter1986allard,gruter1987optimal}. Marquardt only proves the regularity statement for almost every $t>0$, see Lemma 5.3. We may however argue as follows to extend the result to all $t>0$. In order to prove the statement for $\Sigma_t^+$ for instance, we pick a sequence $t_i\searrow t$ such that the regularity statement holds for every $t_i$. It follows from the Arzela-Ascoli theorem that $\Sigma_{t_i}\to \tilde \Sigma$ where $\tilde \Sigma\subset \Sigma^+_t$ is a closed or free boundary surface of class $C^{1,\alpha}$ for every $\alpha<1/2$. On the other hand, the exponential area growth implies that $|\tilde \Sigma|=|\Sigma_t^+|$ and it follows from the variational principle (\ref{variational principle}) and \cite{massari1974esistenza} that $\Sigma_t^+\setminus \partial M$ is contained in a $C^1-$hypersurface, see Lemma 5.1 in \cite{marquardt2017weak}. It follows that $\tilde \Sigma\setminus \partial M=\Sigma_t^+\setminus \partial M$ and $|\Sigma_t^+\setminus \tilde \Sigma|=0$. On the other hand, it follows from Remark \ref{weak mean curvature rema} that $\tilde \Sigma$ has generalized mean curvature in $L^\infty$ and is weakly orthogonal to $\partial M$ and the same consequently holds for $\Sigma_t^+$. We may then argue as in the proof of Lemma 5.3 of \cite{marquardt2017weak}.
\end{proof}
Heuristically, the flow can thus be described as follows. If $\Sigma_0$ is smooth, then the flow evolves smoothly until $\Sigma_t$ ceases to be a minimizing hull. At this point, the flow jumps over a positive volume and $\Sigma_t$ becomes $\Sigma_t^+$ from where the flow continues to evolve. We also remark that the assumption of $E_0$ being a minimizing hull is not restrictive as otherwise, $E_0$ jumps to $E_0'$ right at the start of the flow.\\ For later use, we  record the following compactness theorem, c.f. Lemma 4.4 in \cite{marquardt2017weak}.
\begin{lem}Let $(M_i,g_i)$ be a sequence of Riemannian manifolds such that $M_i$ are open sets satisfying $M_i\subset M$ as well as $M_i\to M$ and suppose that $g_i\to g$ in $C^1_{loc}$. Moreover, suppose that $u_i\in C_{loc}^{0,1}(M_i)$ is a sequence of weak solutions in $M_i$ such that $u_i\to u$ locally uniformly and that $|\overline{\nabla} u_i|$ is locally uniformly bounded. Then $u$ is a weak solution in $M$. Moreover, for almost every $t$, there holds $\Sigma^i_t\to \Sigma_t$ locally in $C^{1,\alpha}$ for any $0<\alpha<1/2$.
	\label{compactness lemma}
\end{lem}
As we will see in Section \ref{geroch monotonicity section}, our monotonicity calculation relies heavily on the fact that the leaves $\Sigma_t$ are connected free boundary surfaces if the same holds for $\Sigma$. Although this implication is generally not true, we can prove it in the special situation where $M\setminus E_0$ is an exterior region.
\begin{lem}
	Suppose that $M$ is a simply connected asymptotically flat half-space with one end and that $\partial M$ is connected. Suppose that $u$ is a proper weak solution with initial data $E_0$. Then $\Sigma_t$ is a connected free boundary surface if the same holds for $\Sigma=\tilde\partial E_0$.
	\label{connected free boundary}
\end{lem}
\begin{proof}
	Using the fact that $M$ is simply connected, we can argue as in the proof of Lemma 4.2 in \cite{huisken2001inverse} to show that $\Sigma_t$ remains connected. If $\Sigma_t$ is not a free boundary surface, then it follows that $\Sigma_t\cap \partial M=\emptyset$. Since $u$ is continuous, this  contradicts the fact that $\partial M$ is connected.
\end{proof}
Finally, we give an example which indicates that we cannot hope for any kind of reasonable monotonicity if the flow starts at a closed surface. Suppose that $M=\mathbb{R}^3_+$ and let $\Sigma_0$ be a small round sphere. One can check that $\Sigma_0$ is a strictly minimizing hull if its radius is sufficiently small and it follows that $\Sigma_t$ expands homothetically until a time $t_0>0$ when it gets close enough to the boundary such that it ceases to be a minimizing hull. At this point, a catenoidal bridge forms connecting $\Sigma_{t_0}$ to the boundary and the flow subsequentially continues as a free boundary inverse mean curvature flow. The usual Hawking mass $\hat m_H$ may be monotone but approaches $\infty$ as $t\to\infty$ because $\Sigma_t$ is asymptotic to a large hemisphere as $t\to\infty$, see Section \ref{asymptotic behavior section}. The modified Hawking mass on the other hand is decreasing, at least up to the time $t_0$. It is also not possible to allow the definition of the Hawking mass to account for the presence of a boundary since $\hat m_H(\Sigma_t)=0$ for any $t<t_0$ but $m_H(\Sigma^+_{t_0})<0$ as $m_H$ is negative unless $\Sigma^+_{t_0}$ is a hemisphere and it is easy to see that this cannot be the case. It is for this reason that we require $\Sigma$ to be a free boundary surface in Theorem \ref{main thm RPI}. 

\section{A new approximation scheme and the existence of weak solutions}
\label{new approximation scheme}
In this section, we will prove the existence of a proper weak solution $u$ provided $(M,g)$ is an $n-$dimensional asymptotically flat half-space\footnote{The definition is the same as the one for $n=3$, see also \cite{almaraz2014positive}.} with one end and $E_0\subset M$ a precompact set whose interior boundary consists of closed and free boundary surfaces of class $C^2$. There is no restriction on the dimension of $M$ and $\Sigma=\tilde\partial E_0$ is allowed to have several components. However, as the application of the flow in this article requires $\Sigma$ to be a connected free boundary surface and since the existence proof does not change when allowing several components, possibly some of them closed, we will only consider the case where $\Sigma$ is a connected free boundary surface of class $C^2$. \\
In order to prove the existence of a weak solution when $\partial M=\emptyset$, Huisken and Ilmanen used an elliptic regularization scheme where the approximate solutions correspond to exact translating smooth solutions of the inverse mean curvature flow in one dimension higher. The advantage of this technique is that it easily provides estimates for the approximate solutions and allows to perform  monotonicity calculations in the smooth setting before passing them to the limit. However, as we have discussed in the introduction, this approach does not seem to work well for manifolds with boundary due to the analytic structure of the approximate partial differential equation, compare \cite{marquardt2017weak}. For this reason, we use a new approximation scheme which still allows for enough geometric interpretability despite being designed for analytic convenience. We will discuss this in more depth in Subsection \ref{geometric interpretability}. We also remark that the approximation scheme used by Moser in \cite{moser2007inverse,moser2015geroch} and further developed by Kotschwar and Ni in \cite{kotschwar2009local} involving the $p$-Laplacian works well for manifolds with boundary from an analytic point of view, too. However, one may check that the approximate solutions correspond to exact solutions with respect to a conformal metric of $g$ whose conformal factor depends on the solution itself and is in general not even continuous. Consequently, there seems to be no hope in using this strategy for our purposes. 
\subsection{The approximate equation and gradient estimates}
Let $(M,g)$ be an $n-$dimensional asymptotically flat half-space, $n\geq 3$, with one end and connected boundary $\partial M$ and suppose that $E_0\subset M$ is compact and bounded by a free boundary surface $\Sigma$ of class $C^2$. Let $M'$ be the non-compact component of $M- \Sigma$. It follows that the boundary of $M'$ is given by $\hat\partial M'\cup \Sigma$ where we recall $\hat \partial M'$ to be the exterior boundary of $M'$.
Next, we choose a compact set $\Omega\subset M$ containing $E_0$ and a chart at infinity from $M\setminus \Omega$ to $\mathbb{R}^n_+\setminus B^n_{R_0}(0)$ for some number $R_0>0$. Slightly abusing notation, we call the corresponding pull-back metric $g$. Finally, we denote the normal of $\Sigma$ pointing outside of $E_0$ by $\nu$ and the normal of $\partial M$ pointing outside of $M$ by $\mu$. \\ In order to find a weak solution, we use a regularization scheme which is chosen in a way such that sub- and supersolutions of the approximate equation can be constructed easily and such that solutions of the regularization scheme are sufficiently smooth.  Given $\Gamma>0$ to be chosen, we define $x_0:=-\Gamma e_3$.
 Let $\epsilon>0$ and $R_\epsilon\geq R_0$ be a large constant to be chosen. Let us define $U_\epsilon := (B^n_{R_\epsilon}(x_0)\cap (\mathbb{R}^n_+\setminus B^n_{R_0}(0)))\cup \Omega $ where we identify the Euclidean domains with subsets of $M$ using the asymptotic chart. After possibly increasing $\Gamma$, we may arrange that $\tilde\partial U_\epsilon$ meets $\partial M$ at an acute angle for any $R_\epsilon\geq R_0$. We choose a smooth perturbation $\Sigma_\epsilon$ of $\Sigma$ in direction $-\nu$ such that $\Sigma_\epsilon\to \Sigma$ in $C^2$ as $\epsilon\to 0$ and such that the contact angle between $\Sigma_\epsilon$ and $\partial M$ is positive but strictly less than $\pi/2$. This procedure, which was also used by Marquardt, see \cite{marquardt2017weak}, is done to ensure the existence of a solution which is $C^1$ up to the corner $\partial \Sigma_\epsilon$. As before, $\Sigma_\epsilon$ divides $M$ into two components and we denote the one intersecting $M'$ by $\tilde  M'_\epsilon$. Likewise, $\tilde \partial U_\epsilon$ divides $ \tilde M_\epsilon'$ in two components and we denote the interior of the precompact one by $M'_\epsilon$. Finally,  let $S_\epsilon=\partial M\cap \overline{M'_\epsilon}$ such that $\partial M'_\epsilon=\tilde \partial U_\epsilon\cup S_\epsilon\cup\Sigma_\epsilon$. 
Given $R_0>0$ and $\epsilon>0$, let $\gamma>1,\tau>0, R_\epsilon>R_0$ and consider the following mixed boundary value problem
\begin{align}
\overline{\operatorname{div}}\bigg(\frac{\overline{\nabla} u_{\epsilon,\gamma,\tau}}{\sqrt{\epsilon^2+|\overline{\nabla} u_{\epsilon,\gamma,\tau}|^2}}\bigg)&=|\overline{\nabla} u_{\epsilon,\gamma,\tau}|^\gamma \text{ in } M'_\epsilon, \label{mbvp 1}\\
u_{\epsilon,\gamma,\tau}&=0 \text{ on } \Sigma_{\epsilon}, \\
u_{\epsilon,\gamma,\tau}&=\tau \text{ on } \tilde\partial  U_\epsilon, \\
\partial_\mu u_{\epsilon,\gamma,\tau}&=0 \text{ on } \label{mbvp 4}S_\epsilon.
\end{align}  
For ease of notation, we omit the subscripts ${\epsilon,\gamma,\tau}$ when there is no risk of confusion. The reader is encouraged to compare (\ref{mbvp 1}) with (\ref{IMCF level set equation}). The analytically inconvenient source term $\epsilon^2$ on the right-hand side has been removed. In order to re-regularize the equation, we have introduced the exponent $\gamma>1$.  We now prove gradient estimates for the function $u$. Some of the following ideas were inspired by the arguments in \cite{marquardt2017weak,kotschwar2009local}. \\
A straightforward computation using the asymptotic behaviour (\ref{asymptotic behavior 0}) reveals that, after choosing $\Gamma$ sufficiently large, there holds 
$$g(\mu,x-x_0)<0, \quad   
g(\mu,\overline\nabla |x-x_0|_e)<0,$$
for all $x\in\{x\in\mathbb{R}^n|x_n=0\}\setminus B^n_{R_0}(0)$.  We then extend the Euclidean distance function $r(x):=|x-x_0|_e$ to all of $M$ in any way such that $r\leq R_0$ in $\Omega$.\\ 
\begin{lem}
	Let $u\in C^2(M'_\epsilon)\cap C^1(\overline{M'_\epsilon})$ be a solution of (\ref{mbvp 1}-\ref{mbvp 4}). Then $0\leq u\leq \tau$. Moreover, there is a constant $R_0$ which only depends on the asymptotic behaviour (\ref{asymptotic behavior 0}-\ref{asymptotic behavior 1}) such that if $\epsilon<(8 R_0)^{-1}$ and $R_\epsilon =(4 \epsilon)^{-1}$ there holds
	\begin{align}
	u\geq \max\{0,\frac{1}{4}(\log(|x-x_0|_e)-\log(R_\epsilon)+4\tau)\},
	\end{align}
	provided $\tau<\frac{1}{4}(\log(R_\epsilon)-\log(R_0))$.	
	\label{subsolution}
\end{lem}
\begin{proof}
	The first statement is an immediate consequence of the maximum principle. In order to prove the more precise estimate, we define $\rho=\frac{1}{4}(\log(r)-\log(R_\epsilon)+4\tau)$ and note that $\rho$ is non-positive inside of $B^n_{R_0}(x_0)$. Moreover, $\rho(x)=\tau$ for $x\in \tilde\partial U_\epsilon$ and  $g(\mu,\overline{\nabla}|x-x_0|)<0$ on $S_\epsilon\setminus B^n_{R_0}(x_0)$ implies $g(\mu,{\overline{\nabla}} \rho)=\partial_\mu \rho(x) <0$ on $S_\epsilon\setminus B^n_{R_0}(x_0)$. It follows that $\rho - u$ does not attain a positive maximum on $S_\epsilon\setminus B^n_{R_0}(x_0)\cup B^n_{R_0}(x_0)\cup\tilde\partial U_\epsilon$. In order to complete the proof, we now exclude that $\rho - u$ attains a positive maximum in the interior. As usual, we indicate geometric expressions with respect to the Euclidean metric by the subscript $e$. Then, using $\overline{\operatorname{div}}(\cdot)=|\det{g}|^{-1/2}\overline {\operatorname{div}_e}(|\det{g}|^{1/2}\cdot)$ we find
	\begin{align*}
	\overline{\operatorname{div}}\bigg(\frac{\overline{\nabla} \rho}{\sqrt{\epsilon^2+|\overline{\nabla}\rho|_e^2}}\bigg)=\partial_i\bigg(&\frac{\partial_i\rho}{\sqrt{\epsilon^2+|\overline{\nabla}_e \rho|_e^2}}-(g_e^{ij}-g^{ij})\frac{\partial_j\rho}{\sqrt{\epsilon^2+|\overline{\nabla} \rho|^2}}\\&-\partial_i\rho\frac{(g_e^{jl}-g^{jl})\partial_j\rho\partial_l\rho}{\sqrt{\epsilon^2+|\overline{\nabla} \rho|^2}\sqrt{\epsilon^2+|\overline{\nabla}_e \rho|_e^2}}\bigg)+\mathcal{O}(|x-x_0|_e^{-2})
	\\=\partial_i\bigg(&\frac{\partial_i\rho}{\sqrt{\epsilon^2+|\overline{\nabla}_e \rho|_e^2}}\bigg)+\mathcal{O}\bigg(\frac{|\overline{\nabla}^2_e\rho|_e}{|\overline{\nabla}_e \rho|_e|x-x_0|_e}\bigg)+\mathcal{O}(|x-x_0|_e^{-2})
	\\=\partial_i\bigg(&\frac{\partial_i\rho}{\sqrt{\epsilon^2+|\overline{\nabla}_e \rho|_e^2}}\bigg)+\mathcal{O}(|x-x_0|_e^{-2}).
	\end{align*}
	Here, we used $g^{ij}-g_e^{ij}=\mathcal{O}(|x-x_0|_e^{-1})$ and $\partial_lg^{ij}=\mathcal{O}(|x-x_0|_e^{-2})$. Likewise, it is easy to see that $$|\overline{\nabla} \rho|^\gamma=|\overline{\nabla}_e \rho|_e^\gamma+\mathcal{O}(|x-x_0|_e^{-2}).$$
	We continue to compute
	\begin{align*}
	&4\overline{\operatorname{div}}_e\bigg(\frac{\overline{\nabla}_e \rho}{\sqrt{\epsilon^2+|\overline{\nabla}_e \rho|_e^2}}\bigg)-4|\overline{\nabla}_e \rho|^\gamma \\\geq& \overline{\operatorname{div}}_e\bigg(\frac{x-x_0}{|x-x_0|_e}\frac{4}{\sqrt{1+16\epsilon^2|x-x_0|_e^2}}\bigg)-\frac{1}{|x-x_0|_e} 
	\\=&\frac{8}{|x-x_0|_e\sqrt{1+16\epsilon^2|x-x_0|_e^2}}- \frac{64\epsilon^2}{({1+16\epsilon^2|x-x_0|_e^2)}^\frac{3}{2}}-\frac{1}{|x-x_0|_e}
	\\\geq &\frac{4\sqrt{2}}{|x-x_0|_e}-\frac{4}{|x-x_0|_e}-\frac{1}{|x-x_0|_e}
	\\\geq &\frac{1}{|x-x_0|_e}. 
	\end{align*} 
	In the first equality we used $\gamma>1$ and assumed $R_0\geq1$. In the second inequality we used $1+16\epsilon^2|x-x_0|_e^2\leq 2$ and $\epsilon^2\leq 1/16 |x-x_0|_e^{-2}$, which are both consequences of $|x-x_0|_e\leq R_\epsilon$.
	Choosing $R_0$ large enough we deduce
	$$
	\overline{\operatorname{div}}\bigg(\frac{\overline{\nabla} \rho}{\sqrt{\epsilon^2+|\overline{\nabla} \rho|^2}}\bigg)-4| \overline{\nabla} \rho|^\gamma>0.
	$$
	It follows that the function $\rho -u$ cannot attain an interior maximum outside of $B^n_{x_0}(R_0)$ and consequently $\rho-u\leq 0$.
\end{proof}
\begin{coro}
	Under the assumptions of the previous lemma there holds 
	$$
	|\overline{\nabla} u|(p)\leq \frac{c}{4R_\epsilon}
	$$
	for any $p\in \tilde \partial U_\epsilon$. Here, $c$ depends on the $C^0-$data of the metric $g$ and if $g=g_e$, one may take $c=1$.
	\label{outer dirichlet est}
\end{coro}
\begin{proof}
	All tangential derivatives of $u$ vanish at $p$. In order to estimate the normal derivative, we can use the previous lemma.
\end{proof}
Next, we prove an interior estimate for $|\overline{\nabla} u|$. To this end, we need two preliminary lemmas. Let $\psi_\epsilon=\sqrt{|\overline{\nabla} u|^2+\epsilon^2}$ and $\zeta$ be a smooth function. We define
$$
L_\epsilon(\zeta):=\frac{\overline{\Delta} \zeta}{\psi_\epsilon}-\frac{\overline{\nabla}^2\zeta(\overline{\nabla} u,\overline{\nabla} u)}{\psi_\epsilon^3}
$$
to be the leading order part of the linearisation of (\ref{mbvp 1}). In the following, we will again omit the subscript $\epsilon$. Let $f=|\overline{\nabla} u|^2$ which means that $\psi=\sqrt{f+\epsilon^2}$.
\begin{lem}
	Let $u\in C^2(M'_\epsilon)$. At any point $p\in M'_\epsilon$ where $|\overline{\nabla} u|>0$ there holds
	$$
	L(f)=\frac{2|\overline{\nabla}^2 u|^2}{\psi}+2\frac{\operatorname{Rc}(\overline{\nabla} u,\overline{\nabla} u)}{\psi}+\frac{|\overline{\nabla} f|^2}{2\psi^3}-\frac{g(\overline{\nabla} f,\overline{\nabla} u)^2}{\psi^5}+\frac{f^{\gamma/2}g(\overline{\nabla} f,\overline{\nabla} u)}{\psi^2}+\gamma f^{\frac{\gamma}{2}-1}g(\overline{\nabla} f,\overline{\nabla} u).
	$$
\end{lem}
\begin{proof}
	We first note that the Schauder theory for elliptic equations implies that $u$ is smooth near $p$.
	The Bochner formula states that
	\begin{align*}
	\overline{\Delta} f=2|\overline{\nabla} ^2u|^2+2g(\overline{\nabla} \overline{\Delta} u, \overline{\nabla} u)+2\operatorname{Rc}(\overline{\nabla} u,\overline{\nabla} u).
	\end{align*}
	On the other hand, since $2\overline{\nabla}^2 u(\overline{\nabla} u,\overline{\nabla} u)=g(\overline{\nabla} f,\overline{\nabla} u)$ we find
	$$
	\frac{\overline{\Delta} u}{\psi}-\frac{g(\overline{\nabla} f,\overline{\nabla} u)}{2\psi^3}=|\overline{\nabla} u|^\gamma.
	$$
	Taking gradients on both sides and multiplying by $\overline{\nabla} u$ it is easy to see, for instance using normal coordinates, that
	$$
	\frac{g(\overline{\nabla} \overline{\Delta} u,\overline{\nabla} u)}{\psi}-\frac{\overline{\Delta} u g(\overline{\nabla} f,\overline{\nabla} u)}{2\psi^3}+\frac34\frac{g(\overline{\nabla} f,\overline{\nabla} u)^2}{\psi^5}-\frac{g(\overline{\nabla} f,\overline{\nabla} f)}{4\psi^3}-\frac{\overline{\nabla}^2 f(\overline{\nabla} u,\overline{\nabla} u)}{2\psi^3}=\frac{\gamma}{2} f^{\frac{\gamma}{2}-1}g(\overline{\nabla} f,\overline{\nabla} u).
	$$
	It follows that
	$$
	\frac{g(\overline{\nabla} \overline{\Delta} u,\overline{\nabla} u)}{\psi}-\frac{f^{\gamma/2} g(\overline{\nabla} f,\overline{\nabla} u)}{2\psi^2}+\frac{g(\overline{\nabla} f,\overline{\nabla} u)^2}{2\psi^5}-\frac{g(\overline{\nabla} f,\overline{\nabla} f)}{4\psi^3}-\frac{\overline{\nabla}^2 f(\overline{\nabla} u,\overline{\nabla} u)}{2\psi^3}=\frac{\gamma}{2} f^{\frac{\gamma}{2}-1}g(\overline{\nabla} f,\overline{\nabla} u).
	$$
	The claim follows easily from these identities.
\end{proof}
In order to capitalize on this identity, we rewrite the Hessian term.
\begin{lem}
	Let $u\in C^2(M'_\epsilon)$ be a solution of (\ref{mbvp 1}) and $p\in M'_\epsilon$ be a point where $|\overline{\nabla} u|\geq 2\epsilon$. Then
	$$
	|\overline{\nabla}^2 u|^2\geq \frac{1}{16}\frac{|\overline{\nabla} f|^2}{f}+\frac{f^\gamma\psi^2}{2(n-1)}+\frac{g(\overline{\nabla} f, \overline{\nabla} u)^2}{8(n-1)\psi^4}.
	$$
\end{lem}
\begin{proof}
	We choose normal coordinates $\{\partial_i\}$ centred at $p$ such that $\partial_1=\overline{\nabla} u/|\overline{\nabla} u|$. Then, $\partial_jf=2\partial_{j}\partial_1u\partial_1u$. In particular, $\partial_j\partial_1u=2^{-1}\partial_jff^{-\frac{1}{2}}$ and $\partial_1u=f^{1/2}$. The equation (\ref{mbvp 1}) satisfied by $u$ can now be written as
	$$
	\sum_{j=2}^n\partial_j\partial_ju=-\partial_1\partial_1u+\frac{\partial_1f\partial_1u}{2\psi^2}+f^{\gamma/2}\psi.
	$$
	A standard trace inequality implies that
	\begin{align*}
	|\overline{\nabla}^2 u|^2\geq& |\partial_1\partial_1u|^2+2\sum_{j=2}^n |\partial_j\partial_1u|^2+\frac{1}{n-1}\bigg(\sum_{j=2}^n|\partial_j\partial_ju|^2\bigg)\\
	=&\frac{n}{n-1}|\partial_1\partial_1 u|^2+2\sum_{j=2}^n |\partial_j\partial_1u|^2+\frac{f^\gamma\psi^2}{n-1}+\frac{g(\overline{\nabla} f, \overline{\nabla} u)^2}{4(n-1)\psi^4}+\frac{f^{\frac{\gamma}{2}}g(\overline{\nabla} f, \overline{\nabla} u)}{(n-1)\psi}\\&-\frac{2f^{\frac{\gamma}{2}}\psi \partial_1\partial_1u}{n-1}-\frac{\partial_1f\partial_1u\partial_1\partial_1u}{(n-1)\psi^2}
	\\=&\frac{1}{n-1}\bigg(\frac{n}{4f}-\frac{1}{2\psi^2}\bigg)|\partial_1f|^2+\frac12\sum_{j= 2}^n \frac{|\partial_jf|^2}{f}+\frac{f^\gamma\psi^2}{n-1}+\frac{g(\overline{\nabla} f, \overline{\nabla} u)^2}{4(n-1)\psi^4}\\&+\frac{f^{\frac{\gamma}{2}}g(\overline{\nabla} f, \overline{\nabla} u)}{(n-1)\psi}-\frac{f^{\frac{\gamma}{2}-1}\psi g(\overline{\nabla} f, \overline{\nabla} u)}{n-1}
	\\\geq& \frac{1}{8}\frac{|\overline{\nabla} f|^2}{f}+\frac{f^\gamma\psi^2}{n-1}+\frac{g(\overline{\nabla} f, \overline{\nabla} u)^2}{4(n-1)\psi^4}-\frac{\epsilon^2f^{\gamma/2-1}}{(n-1)\psi}g(\overline{\nabla} f, \overline{\nabla} u)
	\\\geq& \frac{1}{16}\frac{|\overline{\nabla} f|^2}{f}+\frac{f^\gamma\psi^2}{2(n-1)}+\frac{g(\overline{\nabla} f, \overline{\nabla} u)^2}{4(n-1)\psi^4}. 	\end{align*}
	In the last step we used Young's inequality and the fact that $f\geq 4\epsilon^2$..
\end{proof}
\begin{coro}
	Let $u\in C^2(M'_\epsilon)$ be a solution of (\ref{mbvp 1}), $\gamma\leq 2$ and $\epsilon\leq 1/4$. Furthermore, let us assume that $\operatorname{Rc}\geq -\Gamma/2g$ for some constant $\Gamma\geq 0$. Let $p\in M'_\epsilon$ be a point where $|\overline{\nabla} u|\geq 2\epsilon$.  Then there exists a universal constant $c=c(n)$ such that
	$$
	L(f)\geq \frac{1}{2(n-1)}f^{\gamma+\frac{1}{2}}-\Gamma f^{\frac12}-c(n)\frac{|\overline{\nabla} f|^2}{f^{\frac32}}.
	$$
	\label{coro L est}
\end{coro}
\begin{proof}
	This follows from the previous two lemmas, Young's inequality and from the fact that $\psi^2$ and $f$ can be compared under the assumptions.
\end{proof}
We are now in the position to prove the following interior estimate.
\begin{lem}
	Let $u\in C^2(M'_\epsilon)$ be a solution of (\ref{mbvp 1}) and $p\in M'_\epsilon$. Let $\gamma<2,\epsilon<1/4,$, $R\leq 2\operatorname{dist}_g(p,\partial M'_\epsilon)$ and suppose that the sectional curvature of $(M,g)$ is bounded from below by $-\Gamma$.
	Then
	$$
	|\overline{\nabla} u|^{2\gamma}(p)  \leq c(n)(R^{-2}+\Gamma R^{-1}+\Gamma )+4^\gamma\epsilon^{2\gamma}.
	$$
	\label{interior estimate}
\end{lem}
\begin{proof}
	Let $d(y):=\operatorname{dist}_g(y,p)$ and pick a smooth function $\zeta\in C^\infty(\mathbb{R})$ such that $\zeta\leq 1$, $\zeta\equiv 1$ on $(-\infty, 1/2]$ and $\zeta \equiv 0$ on $[1,\infty)$. We may choose $\zeta$ in a way such that $\zeta'^2\zeta^{-1}\leq 10$ and $0\geq \zeta''\geq -10$. Next, we define $\rho:M'_\epsilon\to\mathbb{R}$ via
	$$
	\rho(y):=\zeta\bigg(\frac{d(y)}{R}\bigg).
	$$
	Evidently, there holds 
	$$
	\frac{|\overline{\nabla} \rho|^2}{\rho} \leq \frac{10}{R^2}.
	$$
	On the other hand, the Hessian Comparison Theorem implies that $|\overline{\nabla} ^2 d| \leq (1+\Gamma d)d^{-1}g$ and consequently
	$$
	|\overline{\nabla}^2 \rho|\leq \frac{30}{R^2}+\frac{10\Gamma}{R}.
	$$
	We now define the quantity $\upsilon:=\rho f$. $\upsilon$ attains an interior maximum at a point $y_0$ and we may assume that $f(y_0)\geq 4\epsilon^2$. In particular, $\upsilon$ is smooth near $y_0$ by standard elliptic regularity theory. Moreover, there holds $\rho\overline{\nabla} f=-f\overline{\nabla} \rho$. We write $L$ in normal coordinates centred at $y_0$ and denote the leading order coefficients by $L_{ij}$. Clearly, $|L_{ij}|\leq 2\psi^{-1}$. It follows by the ellipticity of $L$ that at $y_0$ there holds
	\begin{align*}
	0\geq L(\upsilon)&=\rho L(f)+L(\rho)f+2L_{ij}\partial_if\partial_j\eta \\
	&\geq \rho L(f)-f^{\frac12}\bigg(\frac{80}{R^2}+\frac{20\Gamma}{R}\bigg) \\
	&\geq \frac{1}{2(n-1)}\rho f^{\gamma+\frac12}-c(n)f^{\frac12}\bigg(R^{-2}+\Gamma R^{-1}+\Gamma\bigg).
	\end{align*}
	In the last step, we used the previous corollary. Since $\rho(y_0)\leq1$, the claim follows.
\end{proof}
Before we can prove a global gradient estimate, we need to estimate $|\overline{\nabla} u|$ on $\Sigma_\epsilon$. To this end, we construct suitable barriers.
\begin{lem}
	Let $u\in C^2(M'_\epsilon)\cap C^1({\overline{M'_\epsilon}})$ be a solution of (\ref{mbvp 1}-\ref{mbvp 4}) and $\epsilon>0$ be sufficiently small. Then there exists a constant $c$ depending only on $|A^\Sigma|_{C^0(\Sigma)}, |A^{\partial M}|_{C^0(\partial M)}, \operatorname{inj}(\Sigma),\operatorname{inj}(\partial M)$ and $|\operatorname{Rm}|_{C^0(B_{\operatorname{inj}(\Sigma),g}(\Sigma))}$ such that for any $p\in \Sigma_\epsilon$ there holds
	$$
	|\overline{\nabla} u|(p)\leq c.
	$$
	Here, the norms are taken with respect to the metric $g$.
\end{lem}
\begin{proof}
	This is a variation of Lemma 3.8 in \cite{marquardt2017weak}. According to Lemma \ref{subsolution} there holds $u\geq 0$. We proceed to construct a supersolution. Let $d^{\Sigma_\epsilon}$ be the distance function to $\Sigma_\epsilon$ and $d^{\partial M}$ be the distance function to $\partial M$. There exists a constant $\delta>0$ depending on the injectivity radii of $\Sigma$ and $\partial M$ such that $d^{\Sigma_\epsilon}$ is smooth in $M'_{\epsilon,\delta}:=\{d^{\Sigma_\epsilon}<\delta\}\cap M'_{\epsilon}$ and such that $d^{\partial M}$ is smooth in $\tilde M'_{\epsilon,\delta}:=\{d^{\partial M}<8\delta\}\cap M'_{\epsilon,\delta}$\footnote{It follows from asymptotical flatness that the injectivity radius of $\partial M$ is uniformly bounded from below and that such a $\delta$ exists.}. The Hessian of $d^{\Sigma_\epsilon}$ at a point $p$ is given by the second fundamental form of the level set of $d^{\Sigma_\epsilon}$ containing $p$. Evolving $\Sigma_\epsilon$ in normal direction with unit speed and using for instance Theorem 3.2 in \cite{gerhard1999geometric},  we can compute the evolution equation for $|A^{\Sigma_\epsilon}|^2$ and conclude that $|\overline{\nabla} ^2 d^{\Sigma_\epsilon}|\leq \Theta$ in $M'_{\epsilon,\delta}$ where $\Theta$ only depends on $|A^\Sigma|_{C^0(\Sigma)}$ and the curvature of $g$. Likewise, we may arrange that $|\overline{\nabla} ^2 d^{\partial M}|\leq \Theta$ in $\tilde M'_{\epsilon,\delta}$ where $\Theta$ now also depends on $|A^{\partial M}|_{C^0({\partial M})}$. Let $\Gamma>0$ to be chosen and $\rho(y):=\Gamma d^{\Sigma_{\epsilon}}(y)\zeta(d^{\partial M}(y))$ for some smooth function $\zeta\in C^\infty(\mathbb{R})$ satisfying $1\leq \zeta\leq 2$ with $\zeta(0)=2$. We compute at $S_\epsilon$
	$$
	\partial_\mu \rho=\Gamma(2\partial_\mu d^{\Sigma_{\epsilon}}-\zeta'(0)d^{\Sigma_\epsilon}).
	$$
	$\Sigma_\epsilon$ meets $S_{\epsilon}$ at an acute angle and $\Sigma$ meets $\partial M$ orthogonally. Moreover, $\Sigma_\epsilon$ converges in $C^2$ to $\Sigma$. It then follows by Taylor's theorem that after possibly shrinking $\delta$ there exists a constant $c_1>0$ depending on $|A^\Sigma|_{C^0(\Sigma)}$ and the curvature of of $g$ but not on $\epsilon$ such that $\partial_\mu d^{\Sigma_\epsilon}(y) \geq- c_1d^{\Sigma_\epsilon}(y)$ for any $y\in S_\epsilon\cap \overline{M'_{\epsilon,\delta}}$. Consequently, if $\zeta'(0)<-4c_1$, there holds $\partial_\mu \rho<0$ at $\in S_\epsilon\cap \overline{M'_{\epsilon,\delta}}$. Next, if we assume that $\zeta$ is constant on $[8\delta,\infty)$ then it follows that $\rho$ is smooth in $M'_{\epsilon,\delta}$. We may then compute that 
	$$
	|\overline{\nabla} \rho|^2=\Gamma^2(\zeta^2+2\zeta'\zeta d^{\Sigma_\epsilon}g(\overline{\nabla} d^{\Sigma_{\epsilon}},\overline{\nabla} d^{\partial M})+(d^{\Sigma_{\epsilon}}\zeta')^2)
	$$ 
	and if we also assume that $|\zeta'|\leq ({2\delta})^{-1}$, it follows that 
	\begin{align}
	\frac{\Gamma^2}{2}\leq |\overline{\nabla} \rho|^2\leq 3 \Gamma^2.
	\label{zeta 1}
	\end{align}
	in $M'_{\epsilon,\delta}$. 	Moreover, a straightforward computation reveals that
	\begin{align}
	\bigg|\operatorname{div}\bigg(\frac{\overline{\nabla} \rho}{\sqrt{\epsilon^2+|\overline{\nabla} \rho|^2}}\bigg)\bigg|\leq c(\Theta+1)(1+\delta^{-1})+\Theta|\zeta''|d^{\Sigma_\epsilon}\leq c(\Theta+1)(1+\delta^{-1})
	\label{zeta 2}
	\end{align}
	in $M'_{\epsilon,\delta}$ for some universal constant $c$ (which is in particular independent of $\epsilon$) provided $|\zeta''|\leq 10\delta^{-2}$. Choosing $\Gamma\geq 2+2c(\Theta+1)(1+\delta^{-1})$ and noting that $\gamma > 1$ we conclude from (\ref{zeta 1}) and (\ref{zeta 2}) that
	$$
	\overline{\operatorname{div}}\bigg(\frac{\overline{\nabla} \rho}{\sqrt{\epsilon^2+|\overline{\nabla} \rho|^2}}\bigg) -|\overline{\nabla} \rho|^{\gamma}<0.
	$$
	Summarizing, we have assumed that $\zeta(0)=2$, $1\leq \zeta\leq 2$, $\zeta'(0)<-4c_1$,  $|\zeta'|\leq(2\delta)^{-1}$, $|\zeta''|\leq 10\delta^{-2}$ and that $\zeta=1$ on $[8\delta,\infty)$. After arranging that $\delta^{-1}>8c_1$, one possibility is given by 
	$$
	\zeta(s):=\begin{cases} &1+\exp\bigg(1-\frac{\delta^2}{(\delta-\frac{s}{8})^2}\bigg) \qquad  \text{ if } s<8\delta, \\ &1 \qquad\qquad\qquad\qquad\qquad\quad \text{ if } s\geq 8\delta.
	\end{cases}
	$$
	Furthermore, we may  assume that $\Gamma\geq \delta^{-1}$ and consequently $\rho\geq 1 $ on $\{d^{\Sigma_{\epsilon}}=\delta\}$. We then define the function $\upsilon$ to be $\rho(1-\rho)^{-1}$ where $\rho<1$ and to be $\infty$ elsewhere. Clearly, $\upsilon=0$ on $\Sigma_\epsilon$ and it is easy to see that  $\partial_\mu\upsilon>0$ on $\{\rho<1\}\cap S_\epsilon$. A straightforward calculation reveals that
	$$
	\overline{\operatorname{div}}\bigg(\frac{\overline{\nabla} \upsilon}{\sqrt{\epsilon^2+|\overline{\nabla} \upsilon|^2}}\bigg)-|\overline{\nabla} \upsilon|^{\gamma}<0
	$$ 
	and we conclude that $\upsilon$ is a supersolution of (\ref{mbvp 1}-\ref{mbvp 4}). Thus,
	$$
	0\leq \partial_\nu u\leq \partial_\nu \upsilon= \partial_\nu \rho
	$$
	on $\Sigma_\epsilon$ and the claim follows.
\end{proof}
We are finally in the position to prove a global gradient bound.
\begin{lem}
	Let $u$ be a $C^2(M'_\epsilon)\cap C^1({\overline{M'_\epsilon}})$ solution of (\ref{mbvp 1}-\ref{mbvp 4}), $1<\gamma<2$ and $\epsilon$ sufficiently small. Let $\Sigma$ be of class $C^2$. Then there exists a constant $c$ which depends on $(M,g)$, $|A^\Sigma|_{C^0(\Sigma)}$ and $\operatorname{inj}(\Sigma)$ such that
	\begin{align}
	|\overline{\nabla} u|\leq c.
	\label{global grad estimate}
	\end{align}
	\label{lem global grad estimate}
	\begin{proof}
		Let $\Theta:=\sup_{\partial M}{|A^{\partial M}|}+1$ which is finite by (\ref{asymptotic behavior 0}). We pick a smooth function $\rho$ such that $\partial_\mu \rho=-\Theta \rho$ on $S_\epsilon$, $1\leq \rho\leq 2$ and such that the first and second derivatives of $\rho$ are uniformly bounded by a constant $c$ depending on the asymptotic behaviour of $(M,g)$ as well as a bound for $|A^{\partial M}|$ and $|\operatorname{Rm}|$ in a suitably large compact set. We define $\upsilon=f\rho$ and note that we can equivalently estimate $\upsilon$ to prove the lemma. An easy calculation using the boundary condition (\ref{mbvp 4}) reveals that $\partial_\mu f=-A^\Sigma (\overline{\nabla} u,\overline{\nabla} u)$ on $S_\epsilon$. Consequently, $\upsilon$ cannot attain its maximum on $S_\epsilon$. If $\upsilon$ attains its maximum at $\tilde \partial U_\epsilon$, we can use Corollary \ref{outer dirichlet est}. If $\upsilon$ attains its maximum on $\Sigma_\epsilon$, we can use the previous lemma. Finally, if $\upsilon$ attains an interior maximum at a point $p$ we conclude that $f(p)> 0$ and consequently that $u$ is smooth near $p$. Moreover, $\rho\overline{\nabla} f=-f\overline{\nabla} \rho$ and $L(\upsilon)\leq 0$ at $p$ since ${L}$ is elliptic. As in the proof of Lemma \ref{interior estimate} we conclude that
		\begin{align*}
		0\geq \frac{1}{2(n-1)}\rho f^{\gamma+\frac12}-cf^{\frac12},
		\end{align*}
		where $c$ depends on the Ricci curvature of $g$, $\Theta$ and $n$. The claim follows.
	\end{proof}
\end{lem}
\subsection{Existence of approximate solutions and passing to the limit}
In this subsection, we first show that solutions of the approximate problem (\ref{mbvp 1}-\ref{mbvp 4}) exist. Afterwards, we will prove that a proper weak solution of the free boundary inverse mean curvature flow starting at $E_0$ can be obtained as the pointwise limit of a suitable sequence of functions $u_{\epsilon,\gamma,\tau}$.
\begin{lem}
	Suppose that $\epsilon<1$ is sufficiently small and that $\gamma>1$ is sufficiently close to $1$ (depending on $\epsilon$). Let $ Z_\epsilon:=(\log(R_\epsilon)-\log(R_0))/4$, where $R_\epsilon$ is the constant from Lemma \ref{subsolution}, and assume that $0\leq \tau\leq Z_\epsilon$. Then there exists a unique solution $u_{\epsilon,\gamma,\tau}\in C^{3,\gamma-1}({M'_\epsilon}))\cap C^{1,\beta} (\overline{M'_\epsilon})$ of (\ref{mbvp 1}-\ref{mbvp 4}) where $\beta$ depends on $\gamma,\epsilon$. Moreover, for any $\delta>0$ and $0<\alpha<1$ there exists a constant $c$ such that
	\begin{align}
	|u_{\epsilon,\gamma,\tau}|_{C^{2,\alpha}(M'_{\epsilon,\delta})}\leq c
	\label{approximate c2 estimate}
	\end{align}
	where $M'_{\epsilon,\delta}:=\{p\in M'_\epsilon|\operatorname{dist}_g(p,\partial M\cap(\Sigma_\epsilon \cup \tilde \partial U_\epsilon))<\delta\}$. Here, $c$ depends on $\alpha,\delta,\epsilon$ but not on $\gamma$.
	\label{approximate existence}
\end{lem}
\begin{proof}
	The proof is based on a continuity method and very similar to the proof of Proposition 3.13 in \cite{marquardt2017weak}. We therefore only sketch the argument. Marquardt uses the theory of mixed boundary value problems for domains whose boundaries have edges which was developed by Lieberman in \cite{lieberman1986mixed,lieberman1989optimal}. Given a positive integer $k$, $\alpha\in(0,1)$ and a real number $b$, we define the weighted Hölder norms $|\psi|_{H^b_{k,\alpha}(M'_\epsilon)}:=\sup_{\delta>0}\delta^{b+k+\alpha}|\psi|_{C^{k,\alpha}(M'_{\epsilon,\delta})}$ and the weighted Hoelder spaces   $H^b_{k,\alpha}:=\{\psi\in C_{loc}^{k,\alpha}(M'_\epsilon)||\psi|_{H^b_{k,\alpha}(M'_\epsilon)}<\infty\}$. Recall that both $\Sigma_\epsilon,\hat\partial U_\epsilon$ meet $S$ at an acute angle. Consequently, the theory developed by Lieberman applies. 	Let $I:=\{\tau\in [0,Z_\epsilon] \text{ }| \text{ There exists a unique solution } u\in H^{-1-\beta}_{2,\alpha}(\Omega_\epsilon) \text{ of (\ref{mbvp 1}-\ref{mbvp 4})}\}$ for some $0<\alpha,\beta<1$ to be chosen. Next, we define the spaces
	$$
	A:=\{u\in H^{-1-\beta}_{2,\alpha}(M'_\epsilon) | u=0 \text{ on } \Sigma_\epsilon, \partial_\mu u=0 \text{ on } S_\epsilon \},
	\quad B:=B_1\times B_2:= H^{1-\beta}_{0,\alpha}(M'_\epsilon)\times H^{-1-\beta}_{2,\alpha}(\tilde \partial U_\epsilon)
	$$
	and the differential operator  $$L_{\epsilon,\gamma}(\psi)=\overline{\operatorname{div}}\bigg(\frac{\overline{\nabla} \psi}{\sqrt{\epsilon^2+|\overline{\nabla} \psi|^2}}\bigg)-|\overline{\nabla} \psi|^\gamma.$$ We then define $\mathcal{L}_{\epsilon,\gamma}:A\to B, \psi\mapsto (L_{\epsilon,\gamma}(\psi),\pi(\psi))$, where $\pi:A\to B_2$ is the projection of $A$ to $B_2$. As $\gamma>1$, it is easy to see that $L_{\epsilon,\gamma}$ is $C^1$ provided $\alpha\leq \operatorname{min}\{\gamma-1,\beta\}$.  We now show that $I=[0,Z_\epsilon]$ if $\beta$ is sufficiently small. According to the maximum principle, solutions to (\ref{mbvp 1}-\ref{mbvp 4}) are unique, hence $0\in I$. We proceed to show that $I$ is open. Let $\tau\in I$ and $u$ be the corresponding solution. The linearization of $L_{\epsilon,\gamma}$ at $u$, denoted by $L^u_{\epsilon,\gamma}$, is a uniformly elliptic operator as $u$ enjoys uniform gradient bounds by the previous lemma. One may check easily that its second order coefficients are in $H^{0}_{0,\alpha}(M'_\epsilon)$ and that its first order coefficients are in $H^{1-\beta}_{0,\alpha}(M'_\epsilon)$ provided $\alpha\leq \gamma-1$\footnote{ This restriction is due to the term  $\gamma \nabla u|\nabla u|^{\gamma-2}\in H^{1-\beta}_{0,\gamma-1}(M'_\epsilon)$.}. Choosing $\beta$ sufficiently small (depending on $\epsilon$ or more precisely the contact angle between $\Sigma_\epsilon, \hat \partial U_\epsilon$ and $S$), it then follows that the linear theory for mixed boundary value problems can be applied, see Theorem A.14 in \cite{marquardt2012inverse}. It follows that $L^u_{\epsilon,\gamma}$ is invertible and the inverse function theorem consequently implies that $\tau'\in I$ if $\tau'\in[0,Z_\epsilon]$ and if $|\tau'-\tau|$ is sufficiently small. A compactness argument similar to the one used in the proof of Proposition 3.13 of \cite{marquardt2017weak} also shows that $I$ is closed. Finally, according to Lemma \ref{lem global grad estimate}, solutions of (\ref{mbvp 1}-\ref{mbvp 4}) enjoy uniform gradient bounds and the operator $L_{\epsilon,\gamma}$ is consequently uniformly elliptic. An estimate of de Giorgi-Moser-Nash type, see \cite{gilbarg2015elliptic} Theorem 13.2, then implies uniform $C^{1,\alpha}(M'_{\epsilon,\delta})$ estimates which only depend on $(M,g)$, $\epsilon$ and $\delta$. (\ref{approximate c2 estimate}) now follows from the Schauder estimates for elliptic equations.
\end{proof}
\begin{coro}
	Let  $\epsilon$ be sufficiently small. Then there exists a unique solution $u_\epsilon:=u_{\epsilon,1,Z_\epsilon}$ of (\ref{mbvp 1}-\ref{mbvp 4}) which is in  $C^{2,\alpha}(\overline{M'_\epsilon}\setminus(\partial M\cap(\Sigma_{\epsilon}\cup \tilde\partial U_\epsilon)))\cap C^{0,1}(\overline{M'_\epsilon})$ for any $0<\alpha<1$ and  satisfies the estimate (\ref{global grad estimate}).
\end{coro}
\begin{proof}
	The existence of $u_\epsilon$ follows from compactness and the previous lemma, letting $\gamma\to 1$. The gradient estimate follows from the fact that the Lipschitz constant is lower semi-continuous with respect to uniform $C^0-$convergence.
\end{proof}
\begin{lem}
	The functions $u_\epsilon$  converge locally uniformly in $M'$ and weakly in $W^{1,\infty}(M')$ to a proper weak solution $u\in C_{loc}^{0,1}(M')$ of the free boundary inverse mean curvature flow with initial data $E_0$ in the sense of Definition \ref{weak solution}.
	\label{lem existence of weak solution}
\end{lem}
\begin{proof}
	As $\epsilon$ tends to 0,  $M'_\epsilon$ converges to $M'$ and local uniform convergence along a subsequence $\epsilon_i$ follows from the uniform gradient bound (\ref{global grad estimate}). Likewise, we deduce weak local $W^{1,\infty}-$convergence. The existence of subsolutions, c.f. Lemma \ref{subsolution}, guarantees that $u$ is proper and we may extend $u$ in any way to $E_0$ such that $E_0=\{u<0\}$. Let us verify that $u$ is a weak solution. Let $\Omega$ be a compact set in $M'$ that does not touch $\Sigma$ and let $v$ be a locally Lipschitz function such that $\{v\neq u\} \subset\subset \Omega$. We assume for now that $v<u+1$. Let $\rho$ be a smooth function  equal to $1$ in  $\{v\neq u\}$ and supported in $\Omega$. We define $v_i:=\rho v+(1-\rho)u_i,$ where $u_i:=u_{\epsilon_i}$, multiply (\ref{mbvp 1}) by $v_i-u_i$ and compute, using the divergence theorem while keeping in mind the Neumann condition (\ref{mbvp 4}),
	\begin{align*}
	&-\int_\Omega \rho \frac{\overline{\nabla} u_i \cdot \overline{\nabla} v}{\sqrt{|\overline{\nabla} u_i|^2+\epsilon_i^2}}\text{d}vol+\int_\Omega \rho \frac{|\overline{\nabla} u_i|^2}{\sqrt{|\overline{\nabla} u_i|^2+\epsilon_i^2}}\text{d}vol-\int_\omega (v-u_i) \frac{\overline{\nabla} u_i \cdot \overline{\nabla} \rho}{\sqrt{|\overline{\nabla} u_i|^2+\epsilon_i^2}}\text{d}vol\\=&\int_\omega |\overline{\nabla} u_i|(\rho v-\rho u_i)\text{d}vol.
	\end{align*}
	The third term on the left converges to $0$ by dominated convergence. Moreover, we can estimate the first term from below by
	$
	-\int_\Omega \rho |\overline{\nabla} v|\text{d}vol.
	$
	On the other hand,
	\begin{align}
	\int_\Omega \rho \frac{|\overline{\nabla} u_i|^2}{\sqrt{|\overline{\nabla} u_i|^2+\epsilon_i^2}}\text{d}vol=\int_\Omega \rho \sqrt{|\overline{\nabla} u_i|^2+\epsilon_i^2} \text{d}vol-\int_\Omega \rho \frac{\epsilon_i^2}{\sqrt{|\overline{\nabla} u_i|^2+\epsilon_i^2}}\text{d}vol.
	\end{align}
	The second term on the right converges to $0$, the first term can be estimated using $\sqrt{|\overline{\nabla} u_i|^2+\epsilon_i^2}\geq |\overline{\nabla} u_i|$. This leaves us at
	$$
	\limsup_{i\to\infty}\int_\Omega \rho |\overline{\nabla} u_i| (1+u_i-v)\text{d}vol\leq \int_\Omega |\overline{\nabla} v|\text{d}vol.
	$$
	The integrand on the left hand side eventually becomes positive so the claim follows by lower semi continuity (notice that if $\rho\neq 1$, then $\overline{\nabla} v=\overline{\nabla} u$). The condition $v<u+1$ can be removed as in Theorem 2.1 in \cite{huisken2001inverse}. According to Lemma \ref{lemma properties of weak solutions}, proper weak solutions are unique and full convergence of the sequence follows.
\end{proof}
\subsection{Geometric interpretation of the approximate solutions}
\label{geometric interpretability}
Despite the analytic modification, equation (\ref{mbvp 1}) can still be interpreted as a geometric flow on a cylinder over $M$. In order to see this, we consider the space $M'_\epsilon\times\mathbb{R}$ equipped with the metric $\tilde g=g+dz^2$, choose $\epsilon$ and $\gamma$ sufficiently close to $0$ and $1$, respectively, such that  (\ref{mbvp 1}-\ref{mbvp 4}) can be solved and define the function $\tilde u_{\epsilon,\gamma}:=u_{\epsilon,\gamma,Z_\epsilon}+\epsilon z$.  Keeping in mind that the speed of the level set flow associated with $\tilde u_{\epsilon,\gamma}$ is given by $|\overline{\nabla} \tilde u_{\epsilon,\gamma}|^{-1}$, a straightforward computation shows that (\ref{mbvp 1}) is the level set formulation of the flow equation
$$
\frac{dx}{dt}=\frac{\nu}{\sqrt{\epsilon^2+H^{2/\gamma}}}.
$$
In order to be more precise, we let $t_0>0$ and define   $I_{t_0,\epsilon}:=(-Z_\epsilon/(4\epsilon),t_0/(2\epsilon))$. Moreover, given $t>0$ we define $\tilde \Sigma^{\epsilon,\gamma}_t:=\{\tilde u_{\epsilon,\gamma}=t\}$ and $\tilde \Sigma^{\epsilon}_t:=\tilde \Sigma^{\epsilon,1}_t$. The following holds.
\begin{lem}
	Let $n\leq7,$ $\epsilon>0$ be sufficiently close to $0$, $\gamma>1$ be sufficiently close to $1$ (depending on $\epsilon>0$), $t_0>0$ and $t_0\leq  t \leq Z_\epsilon/2$. Then $\Sigma_{ t}^{\epsilon,\gamma}\cap (M'_{\epsilon}\times I_{t_0,\epsilon})$ is a hypersurface of class $C^{3}$ and $\partial \Sigma_{t}^{\epsilon,\gamma}\cap (M'_{\epsilon}\times I_{t_0,\epsilon})$ is either empty or meets $\partial M\times \mathbb{R}$ orthogonally. 
	\begin{itemize}
		\item \textbf{Geometric Flow.} 
		The outward normal of $\Sigma_{t}^{\epsilon,\gamma}\cap (M'_{\epsilon}\times I_{t_0,\epsilon})$ is given by $\nu=\overline{\nabla }\tilde u_{\epsilon,\gamma}/|\overline{\nabla}\tilde u_{\epsilon,\gamma}|$ and the surfaces $\Sigma_{t}^{\epsilon,\gamma}\cap (M'_{\epsilon}\times I_{t_0,\epsilon})$ flow according to the (degenerate) parabolic equation
		\begin{align}
		\frac{dx}{dt}=\frac{\nu}{\sqrt{\epsilon^2+H^{2/\gamma}}}. \label{approximate flow equation}
		\end{align}
		\item 
		\textbf{Estimates.}
		For any $0<\alpha<1$, the local $C^{2,\alpha}-$estimates of $\Sigma_{t}^{\epsilon,\gamma}\cap (M'_{\epsilon}\times I_{t_0,\epsilon})$ depend on $t_0,\epsilon$, but not on $\gamma$. For any $0<\alpha<1/2$, the surfaces $\Sigma_{t}^{\epsilon,\gamma}\cap (M'_{\epsilon}\times I_{t_0,\epsilon})$ enjoy local $C^{1,\alpha}-$estimates which depend on $t_0$ but are independent of $\Sigma,\epsilon,\gamma$. 
		\item \textbf{Area estimate.}
		Given $0<t_1<Z_\epsilon$ and $t_0< t<t_1$, there is a constant $c$ depending on $t_0$ and $t_1$ but not on $\epsilon$ such that $|\tilde \Sigma^{\epsilon,\gamma}_{t}\cap(M'_\epsilon \times[-10,10])|\leq c$. 
		\item \textbf{Mean curvature estimate.} The mean curvature satisfies the uniform estimate
		\begin{align}
		|H^{\tilde \Sigma^{\epsilon,\gamma}_t}|_{C^0(\tilde \Sigma^{\epsilon,\gamma}_t\cap (M'_\epsilon\times I_{t_0,\epsilon}))}\leq c, \label{uni mc est}
		\end{align}
		where $c$ depends on $(M,g)$, $|A^\Sigma|_{C^0(\Sigma)}$ and $  \operatorname{inj}(\Sigma)$.
	\end{itemize}
	\label{regularity of approximate level sets}
\end{lem}
\begin{proof}
	Clearly, $\tilde u_{\epsilon,\gamma}$ has no critical points. On the other hand, the hypothesis implies that  $\{\tilde u= t\}\cap \{M\times I_{t_0,\epsilon}\}\subset\{t_0/2<u<(3/4)Z_\epsilon\}\times I_{t_0,\epsilon}=:U_{t_0,\epsilon}$. Consequently, the regular value theorem, Lemma \ref{lem global grad estimate}, Lemma \ref{approximate existence}, (\ref{approximate c2 estimate}), the equations (\ref{mbvp 1}-\ref{mbvp 4})  and the identity
	$$
	H^{\Sigma^{\epsilon,\gamma}_t}=|\overline{\nabla }u_{\epsilon,\gamma,Z_\epsilon}|^{\gamma},
	$$
	which also follows from (\ref{mbvp 1}),
	imply all statements apart from the $C^{1,\alpha}-$estimates and the area estimate. In order to prove these, let $\tilde v$ be a smooth $C^2-$function defined on $M'_{\epsilon}\times \mathbb{R}$ such that $\{\tilde v\neq \tilde u_{\epsilon,\gamma}\}\subset \tilde \Omega \subset \subset (M'_{\epsilon}-\tilde\partial M'_{\epsilon})\times \mathbb{R}$ for some compact set $\tilde \Omega$. Multiplying the equation
	$$
	\overline{\operatorname{div}}\bigg(\frac{\overline{\nabla}\tilde u_{\epsilon,\gamma}}{|\overline{\nabla}\tilde u_{\epsilon,\gamma}|}\bigg)=(|\overline{\nabla} \tilde u_{\epsilon,\gamma}|^2-\epsilon^2)^{\frac\gamma2}
	$$
	by $\tilde v-\tilde u_{\epsilon,\gamma}$, integrating over $\tilde \Omega$, keeping in mind the Neumann condition (\ref{mbvp 4}) and using the Cauchy-Schwarz inequality we  find
	$$
	\int_{\tilde \Omega} (|\overline{\nabla} \tilde u_{\epsilon,\gamma}|+\tilde u_{\epsilon,\gamma}({|\overline{\nabla} \tilde u_{\epsilon,\gamma}|^{2}-\epsilon^2})^\frac{\gamma}{2})\text{d}vol \leq \int_{\tilde \Omega}
	(|\overline{\nabla} \tilde v|+\tilde v({|\overline{\nabla} \tilde u_{\epsilon,\gamma}|^{2}-\epsilon^2})^\frac{\gamma}{2})\text{d}vol.
	$$
	Then, we may argue as in Lemma 2.1 in \cite{huisken2001inverse} to see that
	\begin{equation}\begin{aligned}
	&|\tilde\Sigma^{\epsilon,\gamma}_{ t}\cap \tilde U_{t_0,\epsilon}|-\int_{E_{ t}^{\epsilon,\gamma}\cap \tilde U_{t_0,\epsilon}} ({|\overline{\nabla} \tilde u_{\epsilon,\gamma}|^{2}-\epsilon^2})^\frac{\gamma}{2}\text{d}vol\\ \leq &|\tilde \partial^*F\cap \tilde U_{t_0,\epsilon} |-\int_{F\cap \tilde U_{t_0,\epsilon}} ({|\overline{\nabla} \tilde u_{\epsilon,\gamma}|^{2}-\epsilon^2})^\frac{\gamma}{2}\text{d}vol,
	\end{aligned} \label{equivalent approximate}
	\end{equation}
	for all finite perimter sets $F$ such that $F\Delta E_{ t}^{\epsilon,\gamma}\subset \subset \tilde U_{t_0,\epsilon}$.
	Here, $\tilde E_{ t}^{\epsilon,\gamma}:=\{\tilde u_{\epsilon,\gamma}< t\}$ and  $\tilde U_{t_0,\epsilon}:=\{t_0/4<u_{\epsilon,\gamma}<Z_\epsilon/2+1\}\times(-Z_\epsilon/(4\epsilon)-1,t_0/(2\epsilon)+1)$.  On the other hand, it follows from the local uniform convergence to a proper solution, c.f. Lemma \ref{lem existence of weak solution}, that after possibly shrinking $\epsilon$ and choosing $\gamma$ sufficiently close to $1$ (depending on $\epsilon$), there is a small constant $\delta>0$ which depends on $t_0$ but not on $\epsilon$ such that the collar $B^{n+1}_{\delta,\tilde g}(\tilde\Sigma^{\epsilon,\gamma}_{ t}\cap (M'_{\epsilon,\gamma}\times I_{t_0,\epsilon})):=\{p'\in M\times \mathbb{R}|\operatorname{dist}_{\tilde g}(p',\tilde\Sigma^{\epsilon,\gamma}_{ t}\cap (M'_{\epsilon}\times I_{t_0,\epsilon}))<\delta\}$ is compactly contained in $\tilde U_{t_0,\epsilon}$. Thus, if $B^{n+1}_{\delta,\tilde g}(p):=\{p'\in M\times\mathbb{R}|\operatorname{dist}_{\tilde g}(p',p)<\delta\}$ we find that $\tilde \Sigma^{\epsilon,\gamma}_{ t}$ is $c_0(|\overline{\nabla }u_{\epsilon,\gamma,Z_\epsilon}|_{L^\infty(M)},\delta,n),\delta-$almost minimal in the sense that
	$$
	|\tilde \Sigma^{\epsilon,\gamma}_{ t}\cap B^{n+1}_{\delta',\tilde g}(p)|\leq |\tilde \partial^*F\cap B^{n+1}_{\delta',\tilde g}(p) |+c_0\delta'^{n+1}
	$$
	for all $p\in \tilde \Sigma_{t}^{\epsilon,\gamma}\cap U_{t_0,\epsilon}$, $\delta'<\delta$ and $F$ as above. On the other hand, we have seen that the mean curvature $H^{\Sigma_{t}^{\epsilon,\gamma}}$ is uniformly bounded. One may now argue as in \cite{tamanini1981boundaries} and \cite{gruter1986allard} to prove the regularity statement, noting that the assumption $n\leq 7$ rules out the existence of a small singular set.
	Finally, the area estimate follows from (\ref{equivalent approximate}) and a straightforward comparison argument using a large, compact set to compare with and the uniform bound for $|\overline{\nabla} u_{\epsilon,\gamma,Z_\epsilon}|_{L^\infty(M)}$.
\end{proof}
\begin{rema}
	The stated dependency on $t_0$ is not optimal. However, we prefer to use the presented version of the lemma to avoid regularity issues near the corner $\partial \Sigma$.
\end{rema}
Using a modification of a lemma proved by Volkmann in \cite{volkmann2015free}, we also obtain the following length estimate for the free boundary of the level sets.
\begin{lem} Let $(\hat M,\hat g)$ be a Riemannian manifold with non-empty boundary $\partial \hat M$ and $\hat \Sigma\subset\hat M$ be a compact free boundary surface. Then there is a constant $c$ depending on $(M,g)$ such that
	$$
	|\partial \hat \Sigma|\leq c \int_{\hat \Sigma} (1+|H|)\text{d}vol.
	$$
	Moreover, let
	$t_0>0$ and $t_0\leq  t \leq Z_\epsilon/2$. Let $\epsilon>0$ be sufficiently small and $\gamma>1$ be sufficiently close to $1$ (depending on $\epsilon$). There is a constant $c$ depending on $(M,g)$ such that
	$$
	|\partial \tilde \Sigma^{\epsilon,\gamma}_{ t}\cap(M'_{\epsilon}\times[-8,8])|\leq c|\tilde \Sigma^{\epsilon,\gamma}_{t}\cap(M'_{\epsilon}\times[-10,10])|.
	$$
	\label{boundary length estimate}
\end{lem}
\begin{proof}The first claim can be proven in exactly the same way as Lemma 2.18 in \cite{volkmann2015free}. The idea is to estimate $| \partial \hat\Sigma|$ with the help of the first variational formula for the area using a regularized distance function to $\partial\hat M$, similar to $\rho$ in Lemma \ref{lem global grad estimate}, as a calibration.
	For the second claim, we multiply the calibration by a suitable cut-off function with respect to the $z-$variable and use the mean curvature estimate from the previous lemma.
\end{proof}
\section{Monotonicity of the modified Hawking mass}
In this section, we prove the monotonicity of the modified Hawking mass. As we are primarily interested in exterior regions of  asymptotically flat half-spaces, we assume that $(M,g)$ is an asymptotically flat half-space with one end. The Gauss-Bonnet theorem will play a crucial part in the argument and we therefore assume that $n=\dim(M)=3$ and that $\Sigma_t$ as well as $\Sigma_t^+$ are connected free boundary surfaces for all $t>0$. We will first prove an approximate growth formula for an approximate Willmore energy in the more regular case $\epsilon>0,\gamma>1$ and then pass to the limits $\gamma\to1$, $\epsilon\to 0$. The exponential area growth then implies the monotonicity of the modified Hawking mass.
\label{geroch monotonicity section}
\subsection{The approximate growth formula}
In this subsection, we aim to derive an approximate growth formula for the approximate Willmore energy. In order to motivate what follows, we first assume that $\tilde \Sigma_t$ is a smooth family of strictly mean convex, connected free boundary surfaces in $(M,g)$ evolving by (\ref{approximate flow equation}) with $\gamma=1$. We observe that lower order quantities may temporarily evolve in a very different way compared to the exact free boundary inverse mean curvature flow. For instance, if $H\equiv\delta\epsilon$ for some small constant $\delta>0$, then $\partial_t|\tilde \Sigma_t|=\delta(1+\delta^2)^{-\frac12}|\tilde \Sigma_t|$, which is in sharp contrast to the exponential area growth $\partial_t|\Sigma_t|=|\Sigma_t|$ valid for the free boundary inverse mean curvature flow. On the other hand, it turns out that there exists an approximate second order quantity satisfying a growth formula along (\ref{approximate flow equation}) which is very similar to the one of the usual Willmore energy along the smooth free boundary inverse mean curvature flow. More precisely, we notice that the integrand of the Willmore energy is given by $H^2$ and that the derivative of the function $s\mapsto s^2/2$ evaluated at $H$ is exactly the inverse of the speed of the inverse mean curvature flow. Likewise, if we define 
\begin{align}
\psi_\epsilon(s):=s\sqrt{\epsilon^2+s^2}+\epsilon^2\log(\sqrt{\epsilon^2+s^2}+s)-\epsilon^2\log(\epsilon), \label{psi def}
\end{align}
then $f_\epsilon(s):=\psi'_\epsilon(s)/2=\sqrt{\epsilon^2+s^2}$ evaluated at $H$ is exactly the inverse of the speed of the flow (\ref{approximate flow equation}). We thus define the approximate Willmore energy to be
$$\mathcal{W}_{\epsilon}(\tilde \Sigma):=\frac14\int_{\tilde \Sigma} \psi_\epsilon (H)\text{d}vol.$$
It is well-known that the mean curvature of a geometric flow with normal speed $f^{-1}_\epsilon(H)$ evolves according to the evolution equation
\begin{align}
\partial_t H=-\Delta\bigg(\frac{1}{f_\epsilon(H)}\bigg) -\frac{|A|^2}{f_\epsilon(H)}-\frac{\operatorname{Rc}(\nu,\nu)}{f_\epsilon(H)}.
\end{align}
On the other hand, recalling that $\mu$ is the outward normal of $\partial M$, it follows from differentiating the relation $g(\nu,\mu)=0$ on $\partial \tilde \Sigma_t$ that
\begin{align}
A^{\partial M}(\nu,\nu)=-\partial_\mu H\frac{f'_\epsilon(H)}{f_{\epsilon}(H)}. \label{approx bc condition}
\end{align}
Keeping in mind that $\psi'_\epsilon=2f_\epsilon$, $\partial_t(\text{d}vol)=Hf^{-1}_\epsilon\text{d}vol$ and integrating by parts we thus find
\begin{equation}
\begin{aligned}
\partial_t\int_{\tilde \Sigma_t}\psi_\epsilon(H)\text{d}vol=&-\int_{\tilde \Sigma_t}\bigg(2\frac{|\nabla f_\epsilon(H)|^2}{f^2_\epsilon(H)}+2|A|^2+2\operatorname{Rc}(\nu,\nu)-\frac{\psi_\epsilon(H)H}{f_\epsilon(H)}\bigg)\text{d}vol \\&-2\int_{\partial\tilde \Sigma_t}A^{\Sigma}(\nu,\nu)\text{d}vol. 
\end{aligned}
\label{lead to}
\end{equation}
The Gauss equation can be written as \begin{align}\operatorname{Rc}(\nu,\nu)=-K+\frac{\operatorname{Sc}}{2}+\frac{H^2}{2}-\frac{|A|^2}{2}\label{Gauss equation},
\end{align} 
while the free boundary condition implies that $A^{\partial M}(\nu,\nu)=H^{\partial M}-k_g$, where $k_g$ is the geodesic curvature of $\partial \Sigma$. Now, if we assume that the dominant energy condition $\operatorname{Sc},H^{\partial M}\geq 0$ holds, then it follows from the identity $|A|^2=H^2/2+|\Acirc|^2$ and the Gauss-Bonnet theorem that
\begin{align}
\partial_t\int_{\tilde \Sigma_t}\psi_\epsilon(H)\text{d}vol\leq 4\pi\chi(\tilde \Sigma_t)-\int_{\tilde \Sigma_t}\bigg( \frac{3}{2}H^2-\frac{\psi_\epsilon(H)H}{f_\epsilon(H)}\bigg)\text{d}vol.  \label{approximate hm 1}
\end{align}
The uniform area and mean curvature estimates from Lemma \ref{regularity of approximate level sets} together with the fact that each $\tilde \Sigma_t$ is a connected free boundary surface then imply that 
$$
\partial_t\int_{\tilde \Sigma_t}\psi_\epsilon(H)\text{d}vol\leq 4\pi- \frac12\int_{\tilde \Sigma} H^2+c\epsilon,
$$
for some constant $c$ independent of $\epsilon$. On the other hand, if $\Sigma_t$ is an exact solution of the free boundary inverse mean curvature flow, the exponential area growth implies
$$
\partial_t m_H(\Sigma_t)= \frac{(2|\Sigma_t|)^{\frac12}}{(16\pi)^{\frac32}}\bigg(4\pi-\frac12\int_{\Sigma_t} H^2\text{d}vol-\partial_t\int_{\Sigma_t} H^2\text{d}vol\bigg). 
$$
In light of the strong analytic control from Lemma \ref{regularity of approximate level sets} we might therefore hope to obtain the desired monotonicity of the modified Hawking mass in the limit. Of course, we face several obstacles trying to make this strategy rigorous. The level sets $\tilde \Sigma^{\epsilon,\gamma}_t$ are non-compact, in general not smooth and three rather than two-dimensional.  On the other hand, the definition of $\tilde u_{\epsilon,\gamma}$ implies that the level sets $\tilde \Sigma^{\epsilon,\gamma}_t$ converge to a cylinder over $\Sigma_t$ as $\gamma\to 1$ and $\epsilon\to0$. This suggests that  the problem can be localised with respect to the $z-$variable, c.f. \cite{huisken2001inverse}. \\ We now make this idea precise and fix two positive times $0<t_0<t_1$.  Let $\epsilon>0$ and $\gamma=\gamma(\epsilon)>1$ be chosen such that Lemma \ref{approximate existence} and Lemma \ref{regularity of approximate level sets} can be applied. As before, we define $\tilde u_{\epsilon,\gamma}:=u_{\epsilon,\gamma,Z_\epsilon}+\epsilon z$ and after possibly decreasing $\epsilon>0$, we may arrange that $[-10,10]\subset I_{t_0,\epsilon}$. We then pick  a smooth, non-negative and non-zero function $\zeta\in C^\infty(\mathbb{R})$ which is supported in $[-10,10]$. On $[0,\infty)$, we define $\psi_{\epsilon,\gamma}$ to be the anti-derivative of $2f_{\epsilon,\gamma}$ with the initial condition $\psi_{\epsilon,\gamma}(0)=0$ where $f_{\epsilon,\gamma}(s):=\sqrt{s^{\frac{2}{\gamma}}+\epsilon^2}$ is the inverse speed function of the flow (\ref{approximate flow equation}). It follows that $\psi_{\epsilon,\gamma}$ is of class $C^3$ and standard stability results for ordinary differential equations imply that $\psi_{\epsilon,\gamma}\to \psi_{\epsilon}$ locally uniformly in $C^3([0,\infty))$. We first consider the more regular case $\gamma>1$. Since the level sets $\tilde \Sigma^{\epsilon,\gamma}_t$ are in general not of class $C^4$, additional care is required. Let $t\in[t_0,t_1]$. According to Lemma \ref{regularity of approximate level sets}, $\tilde \Sigma^{\epsilon,\gamma}_{t}\cap\operatorname{spt}(\zeta)$ is of class $C^3$. On the other hand, it follows from (\ref{mbvp 1}) that the mean curvature $H$ of $\tilde \Sigma^{\epsilon,\gamma}_{t}$ is given by $H=|\overline{\nabla} u_{\epsilon,\gamma,Z_\epsilon}|^{\gamma}\geq 0$ and that $\tilde \Sigma^{\epsilon,\gamma}_{t}$ is smooth outside of the closed set $ C_{\epsilon,\gamma}:=\{|\overline{\nabla}u_{\epsilon,\gamma,Z_\epsilon}|=0\}\times \mathbb{R}$. Clearly, $H$ vanishes on $C_{\epsilon,\gamma}$. In fact, we can essentially ignore the singular set $C_{\epsilon,\gamma}$ as we shall now see.
\begin{lem} Let $0<t_0<t_1$, $\zeta\in C_c^{\infty}((-10,10))$ and $\epsilon>0,\gamma(\epsilon)>1$ be chosen sufficiently close to $0$ and $1$, respectively. Then the function $t\mapsto \int_{\Sigma^{\epsilon,\gamma}_t} \zeta \psi_{\epsilon,\gamma}(H)\text{d}vol$ is continuously differentiable in $[t_0,t_1]$ and for any $t\in[t_0,t_1]$ there holds
	\begin{align*}
	\partial_t\int_{\Sigma^{\epsilon,\gamma}_t} \zeta \psi_{\epsilon,\gamma}(H)\text{d}vol=&\int_{\tilde \Sigma^{\epsilon,\gamma}_t}\zeta\bigg(-2\frac{|\nabla f_{\epsilon,\gamma}(H)|^2}{f^2_{\epsilon,\gamma}(H)}-2|A|^2-2\operatorname{Rc}(\nu,\nu)+\frac{\psi_{\epsilon,\gamma}(H)H}{f_{\epsilon,\gamma}(H)}\bigg)\text{d}vol\\
	&+\int_{\tilde \Sigma^{\epsilon,\gamma}_t}\bigg(\frac{\psi_{\epsilon,\gamma}(H)}{f_{\epsilon,\gamma}(H)}g(\overline{\nabla}\zeta,\nu)-\frac{1}{f_{\epsilon,\gamma}(H)}g({\nabla}f_{\epsilon,\gamma}(H),\nabla \zeta)\bigg)\text{d}vol
	\\&-\int_{\partial \Sigma^{\epsilon,\gamma}_t} 2\zeta A^{\partial M}(\nu,\nu)\text{d}vol 
	\end{align*}
\end{lem}
\begin{proof}
	Let us choose a function $\tilde\rho\in C^{1}(\mathbb{R})$ such that $0\leq \tilde \rho \leq 1$,  $\tilde\rho'\geq 0$, $\tilde\rho_{|(-\infty,1/2] }\equiv 0$ and $\tilde\rho_{|[1,\infty) }\equiv 1$. Let $\delta>0$ and define $\rho_\delta(p):=\tilde\rho(\operatorname{dist}_{\tilde g}({\tilde C_{\epsilon,\gamma},p})/\delta)$. It follows that $\rho_{\delta}$ is Lipschitz and $|\overline{\nabla}\rho_\delta|\leq c_0/\delta$ for some constant $c_0$. Consequently, it follows that there is a constant $c_1$ independent of $t,\delta$ such that 
	\begin{align}
	\bigg|\int_{\tilde \Sigma^{\epsilon,\gamma}_t} |\overline{\nabla}\rho_\delta|\zeta\text{d}vol\bigg|\leq c_1.
	\label{cut off estimate}
	\end{align}
	Next, we define the functions $\upsilon_\delta,\upsilon$ on $[t_0,t_1]$ via
	\begin{align}
	\upsilon_\delta(t):=\int_{\tilde \Sigma_t^{\epsilon,\gamma}}\psi_{\epsilon,\gamma}(H)\zeta\rho_{\delta}\text{d}vol, \qquad \upsilon(t):=\int_{\tilde \Sigma_t^{\epsilon,\gamma}}\psi_{\epsilon,\gamma}(H)\zeta\text{d}vol.
	\end{align}
	Since $H$ vanishes on $C_{\epsilon,\gamma}$, it follows from bounded convergence that $\upsilon_\delta\to\upsilon$ pointwise. On the other hand, $\tilde \Sigma_t^{\epsilon,\gamma}\cap \operatorname{spt}(\zeta)\cap \operatorname{spt}(\rho_\delta)$ is smooth which implies that $\upsilon_\delta(t)$ is differentiable and we find, similarly to the computations that lead to (\ref{lead to}), that
	\begin{align*}
	\partial_t\upsilon_\delta(t)=&\int_{\tilde \Sigma^{\epsilon,\gamma}_t}\rho_\delta\zeta\bigg(-2\frac{|\nabla f_{\epsilon,\gamma}(H)|^2}{f^2_{\epsilon,\gamma}(H)}-2|A|^2-2\operatorname{Rc}(\nu,\nu)+\frac{\psi_{\epsilon,\gamma}(H)H}{f_{\epsilon,\gamma}(H)}\bigg)\text{d}vol\\&-2\int_{\partial \Sigma^{\epsilon,\gamma}_t} \rho_\delta\zeta A^{\partial M}(\mu,\mu)\text{d}vol \\
	&+\int_{\tilde \Sigma^{\epsilon,\gamma}_t}\rho_\delta\bigg(\frac{\psi_{\epsilon,\gamma}(H)}{f_{\epsilon,\gamma}(H)}g(\overline{\nabla}\zeta,\nu)-\frac{1}{f_{\epsilon,\gamma}(H)}g({\nabla}f_{\epsilon,\gamma}(H),\nabla \zeta)\bigg)\text{d}vol
	\\
	&+\int_{\tilde \Sigma^{\epsilon,\gamma}_t}\zeta\bigg(\frac{\psi_{\epsilon,\gamma}(H)}{f_{\epsilon,\gamma}(H)}g(\overline{\nabla}\rho_{\delta},\nu)-\frac{1}{f_{\epsilon,\gamma}(H)}g({\nabla}f_{\epsilon,\gamma}(H),\nabla \rho_\delta)\bigg)\text{d}vol.
	\end{align*}
	The last line converges to $0$ locally uniformly in $t$ as $\delta\to0$ which follows from (\ref{cut off estimate}), $\psi_{\epsilon,\gamma}(0)=f'_{\epsilon,\gamma}(0)=0$ and the fact that $\tilde u_{\epsilon,\gamma}$ enjoys locally uniform $C^3-$bounds. The remaining terms converge locally uniformly in $t$ to the corresponding terms with $\rho_\delta$ replaced by the indicator function of $\tilde \Sigma^{\epsilon,\gamma}_t\setminus C_{\epsilon,\gamma}$. In order to complete the proof, we first note that $A^{\partial M}({\mu,\mu})=-\partial_\mu Hf'_{\epsilon,\gamma}(H)f_{\epsilon,\gamma}^{-1}(H)=0$ on $C_{\epsilon,\gamma}$, where we used (\ref{approx bc condition}). Moreover, there holds $f'_{\epsilon,\gamma}(H)=\psi_{\epsilon,\gamma}(H)=0$ on $C_{\epsilon,\gamma}$ so it remains to check that $|A|^2=\operatorname{Rc}(\nu,\nu)=0$ almost everywhere on $\tilde \Sigma^{\epsilon,\gamma}_t\cap C_{\epsilon,\gamma}$. To this end, we simply observe that $\nu=\overline{\nabla} \tilde u/|\overline{\nabla} \tilde u|$ and consequently $\nu=\partial_z$ on $\tilde \Sigma^{\epsilon,\gamma}_t\cap C_{\epsilon,\gamma}$. The claim now follows since $\tilde g$ is a product metric.
\end{proof}
In the next lemma, we pass to the limit $\gamma\searrow 1$.
\begin{lem}
	Let $0<t_0<t_1$, $\zeta\in C^{\infty}_c((-10,10))$ and $\epsilon>0$ be sufficiently small. Then, there holds $\log(f^\gamma_\epsilon(H))\in W^{1,2}(\tilde \Sigma^\epsilon_t)$ for almost every $t\in[t_0,t_1]$ and 
	\begin{equation}
	\begin{aligned} 
	&\int_{\tilde \Sigma^\epsilon_{t_1}} \zeta \psi_\epsilon(H)\text{d}vol-\int_{\Sigma^\epsilon_{t_0}} \zeta \psi_\epsilon(H) \text{d}vol \\
	\leq &\int_{t_0}^{t_1} \bigg[\int_{\tilde \Sigma^{\epsilon}_t}\zeta\bigg(-2\frac{|\nabla f_{\epsilon}(H)|^2}{f^2_{\epsilon}(H)}-2|A|^2-2\operatorname{Rc}(\nu,\nu)+\frac{\psi_{\epsilon}(H)H}{f_{\epsilon}(H)}\bigg)\text{d}vol\\
	&+\int_{\tilde \Sigma^{\epsilon}_t}\bigg(\frac{\psi_{\epsilon}(H)}{f_{\epsilon}(H)}g(\overline{\nabla}\zeta,\nu)-\frac{1}{f_{\epsilon}(H)}g({\nabla}f_{\epsilon}(H),\nabla \zeta)\bigg)\text{d}vol
	-\int_{\partial \Sigma^{\epsilon}_t} 2\zeta A^{\partial M}(\mu,\mu)\text{d}vol \bigg]\text{d}t. 
	\end{aligned}
	\label{approximate growth}
	\end{equation}
\end{lem}
\begin{proof}
	If $\gamma>1$, the corresponding claim follows from the previous lemma and integration with respect to the $t-$variable. Let $\gamma_i>1$ be a sequence converging to $1$. We may apply the previous lemma with $\zeta$ replaced by $\tilde \zeta\in C^{\infty}_c((-10,10))$ such that $\tilde \zeta=1$ on $\operatorname{spt}(\zeta)$ and again integrate with respect to the $t-$variable. Using Young's inequality, the uniform (in terms of $\gamma$) area and length bounds, c.f. Lemma \ref{regularity of approximate level sets} and Lemma \ref{boundary length estimate}, as well as the uniform mean curvature estimate (\ref{uni mc est}) we find that there is a constant $c$ independent of $i\in\mathbb{N}$ such that
	\begin{align*}
	\int_{t_0}^{t_1}\int_{\tilde \Sigma^{\epsilon,\gamma_i}_t} \frac{|\nabla f_{\epsilon,\gamma_i}(H)|^2}{f^2_{\epsilon,\gamma_i}(H)}\text{d}vol\text{d}t\leq c. 
	\end{align*}
	In particular, Fatou's Lemma implies that
	\begin{align}
	\liminf_{\gamma_i\to 1}\int_{\tilde \Sigma^{\epsilon,\gamma_i}_t} \frac{|\nabla f_{\epsilon,\gamma_i}(H)|^2}{f^2_{\epsilon,\gamma_i}(H)}\text{d}vol<\infty
	\label{fegamma bound}
	\end{align}
	for almost every $t\in[t_0,t_1]$. On the other hand, it follows from Lemma \ref{regularity of approximate level sets} that $\tilde \Sigma^{\epsilon,\gamma}_t\cap\operatorname{spt}(\zeta)\to \tilde \Sigma^{\epsilon}_t\cap\operatorname{spt}( \zeta)$ in $C^{2}$ uniformly in  $t\in[t_0,t_1]$.
	Consequently, all terms except the ones involving $\nabla f_{\epsilon,\gamma}(H)$ pass to the limit. Moreover, for any $t\in[t_0,t_1]$, it follows that there is some nearby smooth surface $\tilde \Sigma$ such that $\tilde \Sigma^{\epsilon,\gamma_i}_t\cap\operatorname{spt}(\tilde \zeta)$, and $\tilde \Sigma^{\epsilon}_t\cap\operatorname{spt}(\tilde \zeta)$ can be written simultaneously as a normal graph over $\tilde \Sigma$ with $C^2-$convergence of the graph function. It now follows from (\ref{fegamma bound}) and the Rellich-Kochandrov theorem that a suitable subsequence, depending on $t$, of $f_{\epsilon,\gamma}(H^{\tilde \Sigma^{\epsilon,\gamma_i}_t})$ converges weakly to $f_\epsilon(H^{\tilde \Sigma^{\epsilon}_t})$ in $W^{1,2}(\tilde \Sigma)$ for almost every $t\in[t_0,t_1]$. Then, lower semi-continuity and Fatou's lemma imply that
	$$
	\int_{t_0}^{t_1}\int_{\tilde \Sigma^{\epsilon}_t} \frac{|\nabla f_{\epsilon}(H)|^2}{f^2_{\epsilon}(H)}\text{d}vol\text{d}t\leq\liminf_{i\to\infty}\int_{t_0}^{t_1}\int_{\tilde \Sigma^{\epsilon,\gamma_i}_t} \frac{|\nabla f_{\epsilon,\gamma_i}(H)|^2}{f^2_{\epsilon,\gamma_i}(H)}\text{d}vol\text{d}t.
	$$ 
	Finally, in order to deal with the term 
	$$
	\int_{\tilde \Sigma^{\epsilon}_t}f_{\epsilon,\gamma_i}^{-1}(H)g({\nabla}f_{\epsilon,\gamma_i}(H),\nabla \zeta)\text{d}vol
	$$
	we can integrate by parts, use uniform $C^2-$convergence and then perform another integration by parts  on the level sets $\tilde \Sigma^\epsilon_t$ where there holds $\log f_{\epsilon}(H)\in W^{1,2}(\tilde \Sigma^\epsilon_t\cap\operatorname{spt}(\zeta))$. As this is true for almost every $t\in[t_0,t_1]$ by (\ref{fegamma bound}), the claim follows.
\end{proof}

\subsection{Passing the approximate growth formula to the limit} 
Using the ideas developed in \cite{huisken2001inverse}, we now pass (\ref{approximate growth}) to the limit $\epsilon\to 0$. As in the previous section, let $\zeta,\tilde\zeta$ be non-negative, non-zero cut-off functions with respect to the $z-$variable such that $$\operatorname{spt}(\zeta)\subset[-5,5]\subset\{\tilde\zeta=1\}\subset\operatorname{spt}(\tilde \zeta)\subset (-10,10)$$ and $0<t_0<t_1$. In the following estimates, the constants may depend on $\zeta,\tilde \zeta, t_0^{-1}$ as well as $t_1$ but not on $\epsilon$. As before,  Lemma \ref{regularity of approximate level sets} and Lemma \ref{boundary length estimate} imply that
\begin{align}
|\tilde \Sigma^\epsilon_t\cap\operatorname{spt}(\tilde\zeta)|+|\partial \tilde \Sigma^\epsilon_t\cap\operatorname{spt}(\tilde \zeta)|\leq c. \label{area bounds}
\end{align}
Similarly, we recall the uniform mean curvature estimate (\ref{uni mc est})
\begin{align}
0\leq H^{\tilde \Sigma^\epsilon_t}\leq c. \label{mean curvature bound}
\end{align}
It follows that
\begin{align}
|H\epsilon^2\log(\sqrt{\epsilon^2+H^2}+H)(\epsilon^2+H^2)^{-1/2}|\leq c\epsilon 
\end{align}
and consequently, recalling (\ref{psi def}), that
\begin{align}
\psi_\epsilon(H)\frac{H}{f_\epsilon}= H^2+\mathcal{O}(\epsilon). \label{pass ident 0}
\end{align}
Likewise, there holds
\begin{align}
\frac{\psi_\epsilon(H)}{f_\epsilon(H)}=H+\mathcal{O}(\epsilon). \label{pass ident 1}
\end{align}
Returning to (\ref{approximate growth}) with $\tilde\zeta$ instead of $\zeta$, we can use Young's inequality, the area and length estimates (\ref{area bounds}), the mean curvature bound (\ref{mean curvature bound}) and the fact that $|A^{\partial M}|$ is uniformly bounded to deduce that
$$
\int_{t_0}^{t_1}\int_{\tilde \Sigma^\epsilon_t\cap(M\times[-5,5])}\bigg( \frac{|\nabla f_\epsilon(H)|^2}{f_\epsilon(H)^2}
+|\nabla f_\epsilon(H)|^2 +|A|^2\bigg)\text{d}vol\leq c.$$
We pick a subsequence $\epsilon^i\to0$ and obtain using Fatou's lemma that
\begin{align}
\liminf_{i\to\infty} \int_{\tilde \Sigma^{\epsilon_i}_t\cap(M\times[-5,5])}\bigg( \frac{|\nabla f_{\epsilon_i}(H)|^2}{f_{\epsilon_i}(H)^2}
+|\nabla f_{\epsilon_i}(H)|^2 +|A|^2\bigg)\text{d}vol<\infty \label{pointwise bounds}
\end{align}
for almost every $t\in(t_0,t_1)$.  If we make the same considerations with a larger time interval containing $[t_0,t_1]$, we conclude that (\ref{pointwise bounds}) holds at $t_0$ and $t_1$, too, for almost every choice of $0<t_0<t_1$. Let us assume from now on that we have chosen such $t_0,t_1$. Combining (\ref{approximate growth}), (\ref{area bounds}), (\ref{mean curvature bound}) and  (\ref{pass ident 0}-\ref{pass ident 1})  we obtain 
\begin{equation}
\begin{aligned}
&\int_{\tilde \Sigma^{\epsilon_i}_{t_1}} \zeta H^2\text{d}vol-\int_{\tilde \Sigma^{\epsilon_i}_{t_0}} \zeta H^2 \text{d}vol -c{\epsilon_i} \\
\leq &\int_{t_0}^{t_1} \bigg[\int_{\tilde \Sigma^{{\epsilon_i}}_t}\zeta\bigg(-2\frac{|\nabla f_{{\epsilon_i}}(H)|^2}{f^2_{{\epsilon_i}}(H)}-2|A|^2-2\operatorname{Rc}(\nu,\nu)+H^2\bigg)\text{d}vol\\
&+\int_{\tilde \Sigma^{{\epsilon_i}}_t}\bigg(Hg(\overline{\nabla}\zeta,\nu)-\frac{1}{f_{{\epsilon_i}}(H)}g({\nabla}f_{{\epsilon_i}}(H),\nabla \zeta)\bigg)\text{d}vol
-\int_{\partial \Sigma^{{\epsilon_i}}_t} 2\zeta A^{\partial M}(\mu,\mu)\text{d}vol \bigg]\text{d}t. 
\end{aligned}
\label{first growth formula}
\end{equation}
We now pass this inequality to the limit term by term. We may slightly abbreviate the arguments whenever they are very similar to the ones presented in \cite{huisken2001inverse}.
\begin{lem}
	For almost every $0<t_0<t_1$ there holds
	\begin{align*}
	&	\lim_{i\to\infty} \bigg(\int_{\tilde \Sigma^{\epsilon_i}_{t_1}} \zeta H^2\text{d}vol- \int_{\tilde \Sigma^{\epsilon_i}_{t_0}} \tilde \zeta H^2\text{d}vol +\int_{t_0}^{t_1} \int_{\tilde \Sigma^{\epsilon_i}_t}\zeta H^2\text{d}vol\text{d}t\bigg)\\&= \int_{\Sigma_{t_1}\times\mathbb{R}} \zeta H^2\text{d}vol- \int_{\Sigma_{t_0}\times\mathbb{R}} \zeta H^2\text{d}vol + \int_{t_0}^{t_1}\int_{\Sigma_t\times\mathbb{R}}\zeta H^2\text{d}vol\text{d}t.
	\end{align*}
	\label{h2 limit}
\end{lem}
\begin{proof} This is very similar to the argument in \cite{huisken2001inverse}.
	First, it follows from (\ref{area bounds}), (\ref{mean curvature bound}) and the bounded convergence theorem that it suffices to show that
	\begin{align}
	\lim\limits_{i\to\infty}\int_{\Sigma^{\epsilon_i}_t}\zeta H^2\text{d}vol= \int_{\Sigma_t\times\mathbb{R}}\zeta H^2\text{d}vol
	\label{to show}
	\end{align}
	for almost every $t\in[t_0,t_1]$ including $t_0,t_1$. Using Young's inequality, (\ref{area bounds}) and (\ref{mean curvature bound}) we deduce from (\ref{first growth formula}) that the function 
	$$
	t\mapsto \int_{\tilde \Sigma^{\epsilon_i}_t}\zeta H^2\text{d}vol-c_0t
	$$
	is decreasing in $[t_0,t_1]$ if $c_0$ is chosen sufficiently large. It then follows from the compactness theorem for BV-functions, see for instance \cite{pallara2000functions}, that we may choose a subsequence such that  the limit of $\int_{\Sigma^{\epsilon_i}_t}\zeta H^2\text{d}vol$ exists for all but 
	countably many $t\in[t_0,t_1]$ and we may assume that $t,t_0,t_1$ are outside of this exceptional set. It is thus sufficient to prove (\ref{to show}) along a suitable subsequence. \\
	According to Lemma \ref{regularity of approximate level sets} and the locally uniform convergence of $u_{\epsilon,1,Z_\epsilon}$ to the weak solution $u$, c.f. Lemma \ref{lem existence of weak solution}, it follows that  $\tilde \Sigma^{\epsilon_i}_t\cap\operatorname{spt} \zeta$ converges to $(\Sigma_t\times\mathbb{R})\cap \operatorname{spt}\zeta$ in $C^{1,\alpha}$ where $0<\alpha<1/2$. Thus, we may simultaneously write all of these surfaces as the normal graph over a nearby smooth surface $\tilde \Sigma_t$ where the graph functions are denoted by $\upsilon_i$ and $\upsilon$, respectively. In particular, $\upsilon_i\to \upsilon$ in $C^{1,\alpha}$. After choosing another subsequence, we may then conclude using (\ref{pointwise bounds}) and the Rellich-Kochandrov theorem that $f_{\epsilon_i}(H^{\tilde \Sigma^{\epsilon_i}_t})\to \rho$ in $L^q(\tilde \Sigma_t)$ for every $q<\infty$, where $\rho\in L^q(\tilde \Sigma_t)$ is a yet unknown function. Choosing another subsequence, we may assume that this convergence holds pointwise almost everywhere. Taking $q=2$, we deduce by Cauchy-Schwarz that $f^2(H^{\tilde \Sigma^{\epsilon_i}_t})\to \rho^2$ in $L^1(\tilde \Sigma_t)$, which together with the pointwise bound (\ref{mean curvature bound}) implies that $(H^{\tilde \Sigma^{\epsilon_i}_t})^2\to \rho^2$ in $L^1(\tilde \Sigma_t)$ and in particular, after choosing yet another subsequence, $H^{\tilde \Sigma^{\epsilon_i}_t}\to \rho$ pointwise almost everywhere in $\tilde \Sigma_t$. Using the bounded convergence theorem, we conclude that $H^{\tilde \Sigma^{\epsilon_i}_t}\rightharpoonup \rho$ in $L^q(\tilde \Sigma_t)$. Now we can use Remark \ref{weak mean curvature rema} to identify $\rho$ to be the generalized mean curvature of $\Sigma_t\times{\mathbb{R}}$, which is of course equal to $H^{\Sigma_t}$. (\ref{to show}) now follows from the bounded convergence theorem. 
\end{proof}
\begin{lem}
	There holds
	$$
	\int_{t_0}^{t_1}\int_{\partial\tilde \Sigma^{\epsilon_i}_t}\zeta A^{\partial M}(\nu,\nu)\text{d}vol\to \int_{t_0}^{t_1}\int_{\partial \Sigma_t\times{\mathbb{R}}} \zeta A^{\partial M}(\nu,\nu)\text{d}vol$$ as well as 
	$$ 
	\int_{t_0}^{t_1}\int_{\tilde \Sigma^{\epsilon_i}_t} \zeta\operatorname{Rc}(\nu,\nu)\text{d}vol\to \int_{t_0}^{t_1}\int_{ \Sigma_t\times{\mathbb{R}}}\zeta \operatorname{Rc}(\nu,\nu)\text{d}vol.
	$$
	as $i\to\infty$.
	\label{limit 2}
\end{lem}
\begin{proof}
	This follows from the $C^{1,\alpha}-$convergence of the level sets, (\ref{area bounds}) and the bounded convergence theorem.
\end{proof}
Before we proceed, we note that according to Lemma \ref{lemma properties of weak solutions} there holds $H^{\Sigma_t}=|\overline{\nabla }u|$. On the other hand, the co-area formula implies that $|\overline{\nabla }u|^{-1}(p)<\infty$ for almost every $t>0$ and almost every $p\in{\Sigma_t}$. Consequently, the function $H^{-1}$ is well-defined on almost every level set $\Sigma_t$.
\begin{lem}
	There holds
	$$
	\int_{ \Sigma_t\times\mathbb{R}} \frac{|\nabla H|^2}{H^2}\zeta\text{d}vol\leq \liminf_{i\to\infty} \int_{\tilde \Sigma^{\epsilon_i}_t} \frac{|\nabla  f_{\epsilon_i}(H)|^2}{f_{\epsilon_i}(H)^2}\zeta\text{d}vol
	$$
	for almost every $t\in[t_0,t_1]$. \label{limit 3}
\end{lem}
\begin{proof}
	This is a variation of Lemma 5.2 in \cite{huisken2001inverse}. As in the proof of Lemma 
	\ref{h2 limit} we may simultaneously write  $\tilde \Sigma^{\epsilon_i}_t\cap\operatorname{spt}(\zeta)$ and $ (\Sigma_t\times\mathbb{R)}\cap \operatorname{spt}(\zeta)$ as the normal graph over some nearby surface $\tilde \Sigma_t$ with $C^1-$convergence of the graph functions and we may assume that all surfaces involved consist of one connected component. Let 
	$$y_i:=\int_{\tilde \Sigma^{\epsilon_i}_t\cap\operatorname{spt}(\zeta)}\log f_{\epsilon_i}(H)\text{d}vol.
	$$
	According to (\ref{mean curvature bound}),(\ref{area bounds}) and (\ref{pointwise bounds}) we may assume that, after choosing another subsequence,  for almost every $t$ there holds $\log f_{\epsilon_i}(H^{\tilde \Sigma^{\epsilon_i}_t})-y_i\to \rho$ in $L^2{(\tilde \Sigma_t)}$ for some function $\rho \in L^2{(\tilde \Sigma_t)}$ and that $y_i\to y\in[-\infty,\infty)$. $y=-\infty$ implies after choosing another subsequence that $H^{\tilde \Sigma^{\epsilon_i}_t}\to0$ almost everywhere and it follows from Remark \ref{weak mean curvature rema} that $H^{\Sigma_t}\equiv 0$. According to the discussion preceding the lemma, this cannot happen for almost every $t>0$. In the remaining case, that is $y>-\infty$, it follows that $\operatorname{log}f_{\epsilon_i}(H^{\tilde \Sigma^{\epsilon_i}_t})$ converges in $L^2(\tilde \Sigma_t)$. The bounded convergence theorem then implies that, after choosing another pointwise converging subsequence, $H^{\tilde \Sigma^{\epsilon_i}_t}$ converges in $L^2(\tilde \Sigma_t)$, too. Using Remark \ref{weak mean curvature rema}, we may then identify the limit as in the proof of Lemma \ref{h2 limit} and the claim follows from lower semi-continuity.
\end{proof}
\begin{lem}
	There holds
	$$
	\int_{t_0}^{t_1}\int_{\tilde \Sigma^{\epsilon_i}_t}g(\overline{\nabla}\zeta,\nu)H\text{d}vol\text{d}t\to 
	0, \qquad \int_{t_0}^{t_1}\int_{\tilde \Sigma^{\epsilon_i}_t} \frac{g(\nabla \zeta,\nabla f_{\epsilon_i}(H))}{f_{\epsilon_i}(H)}\text{d}vol\text{d}t\to 0
	$$
	as $i\to\infty$. \label{limit 4}
\end{lem}
\begin{proof}
	The argument is completely verbatim to Lemma 5.3 and the reasoning on p.399 in  \cite{huisken2001inverse}. The central observation consists in the fact that  both $\nu$ and $\nabla f_{\epsilon_i}(H)$ become perpendicular to $\partial_z$ as   $\tilde \Sigma^{\epsilon_i}_t$ converges to the cylinder $\Sigma_t\times\mathbb{R}$.
\end{proof}
In order to deal with the $|A|^2$ term, we need to define a weak notion of the second fundamental form for submanifolds with boundary. To this end, we follow the idea in \cite{huisken2001inverse} which is based on work of Hutchinson, see \cite{hutchinson1986second}. Let us recall that the first variational formula of the area functional for a smooth surface with boundary ${\tilde\Sigma}$ states that
$$
\int_{\tilde\Sigma} Hg(\nu,X)\text{d}vol=\int_{\tilde\Sigma} {\operatorname{div}}(X)\text{d}vol-\int_{\partial {\tilde\Sigma}} g(X,\mu)\text{d}vol,
$$ 
where $\mu$ is the outward conormal of $\partial {\tilde\Sigma}$ and $X$ any $C^1-$vector field on $M$. 
Let us choose a local coordinate frame $\{\partial_i\}$ and assume that $q=(q_{ij})_{ij}$ is a 
compactly supported, symmetric two-tensor on $M$. Then, plugging $X=g^{jk}q_{lk}\nu^l\partial_j$
into the first variational formula we find
\begin{align}
\int_{\tilde\Sigma} Hq(\nu,\nu)\text{d}vol=\int_{\tilde\Sigma} ({\operatorname{div}}(q)(\nu)+g(q,A))\text{d}vol-\int_{\partial {\tilde\Sigma}} q(\nu,\mu)\text{d}vol.
\label{int by parts formula}
\end{align}
Hence, if ${\tilde\Sigma}$ is a free boundary $C^1-$submanifold that possesses a locally integrable generalized mean curvature, we say that ${\tilde\Sigma}$ has a weak second fundamental form $A$ if $A$ is a section of $\operatorname{Sym}(T{\tilde\Sigma})$ for which (\ref{int by parts formula}) holds for any compactly supported choice of $q$.\\
If ${\tilde\Sigma}$ is a $W^{2,2}\cap C^{1,\alpha}-$surface, we can locally write it as the normal graph over a nearby smooth surface $\hat \Sigma$ and call the graph function $\upsilon$. We may then mollify the graph function to obtain a family of smooth surfaces that converge to ${\tilde\Sigma}$ in $W^{2,2}\cap C^{1,\alpha}$. By slightly perturbing the approximating surfaces  near $\partial M$, we may assume that they meet $\partial M$ orthogonally. As the second fundamental form can be expressed by the Hessian of $\upsilon$ plus some lower order correction terms, see (7.10) in \cite{huisken2001inverse} for instance, it follows that ${\tilde\Sigma}$ has a weak second fundamental form $A\in L^2({\tilde\Sigma},\operatorname{Sym}(T{\tilde\Sigma}))$. Moreover, if $A$ exists, we find by inserting $q=\rho g$ for some cut-off function $\rho$ that indeed $H=\operatorname{tr}_{{\tilde\Sigma}}A$ almost everywhere. \\
On the other hand, it follows from Remark \ref{weak mean curvature rema}, the Riesz representation theorem and weak convergence that if there is a sequence of free boundary surfaces $\Sigma_i$ with weak second fundamental form converging to ${\tilde\Sigma}$ locally in $C^1$ and if $A^{\Sigma_i}$ is uniformly bounded in $L^2(\Sigma_i,\operatorname{Sym}(T\Sigma_i))$, then $A^{{\tilde\Sigma}}\in L^2({\tilde\Sigma})$ exists with weak convergence
\begin{align*}
\int_{\Sigma_i} g(A^{\Sigma_i},q)\text{d}vol\to \int_{{\tilde\Sigma}} g(A^{\tilde\Sigma},q)\text{d}vol
\end{align*}
as well as
\begin{align}
\int_{{\tilde\Sigma}} |A^{{\tilde\Sigma}}|^2\text{d}vol \leq \liminf_{i\to\infty} \int_{\Sigma_i} |A^{\Sigma_i}|^2\text{d}vol. \label{limit 5}
\end{align}
Moreover, as the Hessian of the graph function $\upsilon$ and $A$ agree up to lower order terms, it follows that ${\tilde\Sigma}$ is of class $W^{2,2}$. \\
In any case, if ${\tilde\Sigma}$ possesses a weak second fundamental form, then we may diagonalize $A$ almost everywhere in ${\tilde\Sigma}$ with orthonormal eigenvectors $e_1,e_2$ as well as eigenvalues $\kappa_1,\kappa_2$ and define the weak Gauss curvature $K$ to be
\begin{align}
K=\operatorname{Rm}(e_1,e_2,e_2,e_1)+\kappa_1\kappa_2. \label{gauss curvature}
\end{align}
It follows that $K\in L^1({\tilde\Sigma})$. Through approximation by smooth free boundary surfaces one then obtains the following Gauss-Bonnet type lemma.
\begin{lem}
	Let ${\tilde\Sigma}$ be a compact free boundary surface with weak second fundamental form in  $L^2$. Then the weak Gaussian curvature $K\in L^1({\tilde\Sigma})$ satisfies
	$$
	\int_{\tilde\Sigma} K\text{d}vol+\int_{\partial {\tilde\Sigma}}\operatorname{tr}_{\partial {\tilde\Sigma}} A^{\partial M}\text{d}vol=2\pi\chi({\tilde\Sigma}).
	$$
	\label{weak gauss bonnet}
\end{lem}
Using the area bounds and the mean curvature estimate (\ref{uni mc est}) we also  conclude the following.
\begin{coro}
	For all $t\in[0,t_1]$ there holds
	\begin{align}
	\int_{\Sigma_t} |A|^2 \text{d}vol <c,
	\end{align}
	where $c$ is a constant depending on $t_1$.
	\label{weak a2 coro}
\end{coro}
\begin{proof}
	This is very similar to Lemma 5.5 in \cite{huisken2001inverse}.
	First, it follows  from (\ref{pointwise bounds}), the convergence of the level sets and the  discussion preceding Lemma \ref{weak gauss bonnet} that $\Sigma_t$ possesses a weak second fundamental form $A\in L^2(\Sigma_t,\operatorname{Sym}(T\Sigma_t))$ for almost every $t\in[0,t_1]$. Moreover, the pointwise almost everywhere decomposition 
	$|A|^2=\kappa_1^2+\kappa_2^2=H^2-2\kappa_1\kappa_2$ holds, where, as before, $\kappa_i$ denote the principal curvatures. On the other hand, Lemma \ref{lemma properties of weak solutions} implies that for every $t\in[0,t_1]$ there is a number $\epsilon_t>0$ such that $|\chi(\Sigma_{t'})|\leq \operatorname{max}\{|\chi(\Sigma_t)|,|\chi(\Sigma^+_t|)\}$ for all $t'\in(t-\epsilon,t+\epsilon)$. A standard covering argument then yields that $|\chi(\Sigma_t)|\leq c_1$ for all $t\in[0,t_1]$, where $c_1$ only depends on $t_1$. Since $(M,g)$ has locally bounded geometry, it now follows from (\ref{gauss curvature}), Lemma \ref{weak gauss bonnet} and (\ref{area bounds}) that
	$$
	\bigg|\int_{\Sigma_t} \kappa_1\kappa_2\text{d}vol\bigg|<c_2
	$$
	where $c_2$ depends on $t_1$.
	The claim now follows for almost every $t$ since $H$ is uniformly bounded, see (\ref{mean curvature bound}). Finally, we conclude the claim for all $t$ by lower semi-continuity, see (\ref{limit 5}).
\end{proof}
We are now in the position to pass all quantities to the limit to obtain the following growth formula for the Willmore energy.
\begin{lem}
	Let $(M,g)$ be an asymptotically flat half-space, $E_0\subset M$ be a precompact subset and $\Sigma=\tilde \partial E_0$ be a free boundary surface of class $C^1$ with weak second fundamental form in $L^2$. Let $E_t$ be the precompact weak solution of the free boundary inverse mean curvature flow with initial data $E_0$. Then for every $0\leq t_0< t_1$ there holds
	\begin{equation}
	\begin{aligned}
	\int_{\Sigma_{t_0}} H^2\text{d}vol\geq& \int_{\Sigma_{t_1}} H^2\text{d}vol +\int_{t_0}^{t_1}\int_{\Sigma_t}\bigg(\frac12 H^2-4\pi\chi(\Sigma_t)+2\frac{|\nabla H|^2}{H^2}+\frac12 |\Acirc|^2+\operatorname{Sc}\bigg)\text{d}vol\text{d}t
	\\& +\int_{t_0}^{t_1}\int_{\partial \Sigma_t} H^{\partial M}\text{d}vol\text{d}t.
	\end{aligned}
	\label{h2 growth formula}
	\end{equation}
	\label{h2 growth lemma}
\end{lem}
\begin{proof}
	We first prove the assertion for almost every $0<t_0<t_1$. In this case, we first pass (\ref{first growth formula}) to the limit using Lemma \ref{h2 limit}, \ref{limit 2} and \ref{limit 3} as well as Lemma \ref{limit 4} and the lower semi-continuity (\ref{limit 5}) together with Fatou's lemma. We may then integrate the $z-$variable using Fubini's theorem. Finally, we use the Gauss equation (\ref{Gauss equation}), $A^{\partial M}(\nu,\nu)= H^{\partial M}-\operatorname{tr}_{\partial \Sigma}A^{\partial M}$, Lemma \ref{weak gauss bonnet} as well as the identity $|A|^2=H^2/2+|\Acirc|^2$. In order to get the claim for almost every $t_0>0$ and every $t_1>t_0$, we use lower semi-continuity and the fact that $\Sigma_s\to\Sigma_t$ in $C^1$ as $s\nearrow t$, compare Lemma \ref{regularity of approximate level sets}.  \\
	In order to prove the claim for $t_0=0$, we first assume that $\Sigma$ is smooth. It follows from (\ref{mc minimizing hull}) and the fact that forming the strictly minimizing hull does not increase the interior area of the boundary that
	$$
	\int_{\tilde \partial E_0'} H^2\text{d}vol\leq \int_{\Sigma} H^2\text{d}vol.
	$$
	Since the precompact weak flow starting at $E_0'$ coincides with the precompact weak flow starting at $E_0$ we may thus replace $E_0$ by $E_0'$ to prove the lemma. Let $\Sigma'=\tilde \partial E_0'$. According to Lemma \ref{approximation lemma} below there are two possibilities, the first one being that $\Sigma'$ is a smooth minimial surface which holds for instance if $\Sigma$ is a smooth minimal surface, too. In this case, we choose a sequence of functions $\rho_i$ defined on $M$ such that $\rho_i\to1$ in $C^{\infty}(M)$, $\rho_i=1$ on $\overline{E_0'}$, $\rho_i>1$ on $M-E_0'$, $\partial_\nu \rho_i>1$ on $\Sigma'$ and $\partial_\mu \rho_i=0$ on $\partial M$. We define $g_i:=\rho_ig$ and it follows that $E_0'$ is also a strictly minimizing hull in $(M,g_i)$ and that $\Sigma'$ is strictly mean convex with respect to $g_i$. According to Lemma \ref{regularity of approximate level sets}, the precompact weak flow starting at $\Sigma'$, denoted by $\Sigma^i_t$, remains smooth for a short-time and a similar calculation as the one that lead to (\ref{approximate hm 1}) shows that (\ref{h2 growth formula}) holds for the flow $\Sigma^i_t$ with $\tilde t_0=0$ and $\tilde t_1>0$ as long as the flow remains smooth until $\tilde t_1$. Combining this with the  fact that (\ref{h2 growth formula}) holds for almost every $\hat t_0>0$ and every $\hat t_1>\hat t_0$ yields for every $0<t_0$
	\begin{align*}
	o(1)\geq& \int_{\Sigma^i_{t_0}} H^2 \text{d}vol+\int_{0}^{t_0}\int_{\Sigma^i_t}\bigg(\frac12 H^2-4\pi\chi(\Sigma^i_t)+2\frac{|\nabla H|^2}{H^2}+\frac12 |\Acirc|^2+\operatorname{Sc}\bigg)\text{d}vol\text{d}t
	\\&+\int_{0}^{t_0}\int_{\partial \Sigma^i_t} H^{\partial M}\text{d}vol\text{d}t
	\\\geq& \int_{\Sigma^i_{t_0}} H^2 
	-ct_0,
	\end{align*}
	where the last inequality follows from discarding the non-negative terms, the exponential area growth, the length estimate Lemma \ref{boundary length estimate}, the uniform mean curvature bound and the fact that $-\chi(\Sigma^i_t)$ is uniformly bounded from below. To see the last claim, we may use Lemma \ref{weak gauss bonnet}, the exponential area growth, the length estimate Lemma \ref{boundary length estimate}, the mean curvature bound Lemma \ref{lem global grad estimate} and the inequality $-2\kappa_1\kappa_2\geq-H^2$. Now, choosing $t_0$ outside of the exceptional set in the compactness result Lemma \ref{compactness lemma} it follows from Lemma \ref{compactness lemma} and lower semi-continuity, c.f. Remark \ref{weak mean curvature rema}, that
	$$
	\int_{\Sigma_{t_0}} H^2\text{d}vol\leq ct_0.
	$$ 
	The claim now follows for $t_0=0$ by choosing an appropriate sequence $t_0\searrow 0$ in (\ref{h2 growth formula}) and monotone convergence, using the exponential area growth, the boundary length estimate and the fact that $-4\pi \chi(\Sigma_t), \operatorname{Sc}, H^{\partial M}$ are all uniformly bounded from below.	\\
	The second possibility is that $\Sigma'$ is a weakly mean convex free boundary surface of class $C^{1,1}$ and that $H^{\Sigma'}>0$ on a positive measure set. Using Lemma \ref{approximation lemma} below we may then choose a sequence $\Sigma^i$ of smooth, outward minimizing, strictly mean convex free boundary surfaces approximating $\Sigma'$ in $C^1$ from the inside such that
	$$
	\int_{\Sigma^i} H^2\text{d}vol\to \int_{\Sigma'} H^2\text{d}vol.
	$$ 
	The $C^1-$convergence also implies that there is a $C^1-$diffeomorphism of $M$ mapping $\Sigma^i$ to $\Sigma'$ such that the induced metric $g_i$ converges to $g$ in $C^1$. We may then argue as in the first case. \\
	If $\Sigma$ is only $C^1$ with weak second fundamental form in $L^2$, then we can use an approximation argument as in \cite{huisken2001inverse} to conclude the claim. Finally, in order to prove the claim for all $t_0$, we restart the flow at $t_0$ using uniqueness and that $\Sigma_{t_0}$ is of class $C^1$ according to Lemma \ref{regularity of approximate level sets} and has weak second fundamental form in $L^2$ according to Corollary \ref{weak a2 coro}.
\end{proof}
We now prove the missing regularity lemma.
\begin{lem}
	1. Suppose that $E$ is a precompact strictly minimizing hull such that $\tilde \partial E$ is a free boundary surface of class $C^{1,1}$. Then either $\tilde \partial E$ is a smooth free boundary minimal surface or there is a sequence of  strictly minimizing hulls $E^i\subset E$ such that $\tilde \partial E^i\to\tilde \partial E$ in $C^1$, $\tilde \partial E^i$ is a strictly mean convex free boundary surface, $|A^{\tilde \partial E^i}|_{L^\infty(\partial E^i)}$ is uniformly bounded and
	$$
	\int_{\tilde \partial E^i} H^2\text{d}vol\to \int_{\tilde \partial E} H^2 \text{d}vol.
	$$
	2. Suppose that $E$ is precompact and that $\tilde \partial E$ is a free boundary surface of class $C^2$. Then $\tilde \partial E'$ is of class $C^{1,1}$ with estimates only depending on the $C^2-$data of $\tilde \partial E$, the $C^1-$data of $g$ and $|A^{\partial M}|_{L^\infty(\partial M)}$.
	\label{approximation lemma}
\end{lem}
\begin{proof}
	The first statement is very similar to Lemma 5.6 in \cite{huisken2001inverse} and we only sketch the argument. Since $\tilde \partial E$ is in $C^{1,1}$ we can use mollification to find a sequence of free boundary surfaces $\Sigma^i$ approximating $\tilde \partial E$ from the inside in $C^1\cap W^{2,2}$ such that $|A^{ \Sigma^i}|_{L^\infty(\Sigma^i)}\leq c_0$ and
	$$
	\int_{\Sigma^i} H^2\text{d}vol \to \int_{\tilde \partial E} H^2\text{d}vol, \qquad 
	\int_{\Sigma^i} H_-^2\text{d}vol \to \int_{\tilde \partial E} H_-^2\text{d}vol.
	$$
	The free boundary mean curvature flow was introduced by Stahl in \cite{stahl1996convergence,stahl1996regularity}, see also \cite{buckland2005mean}. It is a smooth flow of free boundary surfaces evolving with outward normal speed equal to $-H$. Adapting the techniques developed by Stahl, it is easy to see that there is a smooth free boundary mean curvature flow $\Sigma^i_s$, $s\geq 0$, starting at $\Sigma^i$ and existing for a short time $\epsilon_i>0$. Corollary 3.5 in \cite{huisken1986contracting} implies that
	\begin{align}
	\partial_s |A|^2=\Delta|A|^2-2|\nabla A|^2+2|A|^4+A*A*\operatorname{Rm}+A*\nabla \operatorname{Rm}
	\end{align}
	while one may easily adapt Proposition 2.3 in \cite{stahl1996convergence} to show that
	\begin{align}
	\partial_\mu |A|^2= \nabla^{\partial M} A^{\partial M}*A+A^{\partial M}*A*A+\operatorname{Rm}*A*A
	\end{align}
	as long as the flow exists. We define the function $\rho:=e^{c_1d}|A|^2$, where $d$ is some regularized distance function to $\partial M$. Choosing $c_1$ large enough (depending on the ambient curvature and $\partial M$), it follows that $|A|$ must be uniformly bounded if $\rho$ attains its maximum on the boundary and that otherwise
	$$
	\partial _s \operatorname{max}_{\Sigma^i_s} |A|^2\leq c(|A|^2+|A|^4).	
	$$
	It follows from standard comparison arguments for ordinary differential equations that $|A|^2$ remains uniformly bounded for a short time $\tilde \epsilon>0$ and one may then adapt the higher order estimates developed in \cite{huisken2001inverse} to conclude that $\epsilon_i>\epsilon>0$ for all $i\in\mathbb{N}$ and some short time $\epsilon>0$. Moreover, it follows that higher order estimates, which are uniform in $i\in\mathbb{N}$, hold  on every time interval $(\hat\epsilon,\epsilon]$, where $\hat\epsilon>0$. As in \cite{huisken2001inverse} one may then pass to the limit to obtain a free boundary mean curvature flow starting at $\tilde\partial E$ and the remaining claims follow from a growth estimate for the integrals of $H^2$ and $H_-^2$ as well as the strong maximum principle. \\
	We also only give a sketch of the proof of the second statement as the arguments are very similar to the case without boundary. First, using the results in \cite{gruter1986allard} and \cite{tamanini1981boundaries} it follows that $\Sigma':=\tilde \partial E'$ is of class $C^{1,\alpha}$ for all $\alpha<1/2$ with corresponding estimates. Moreover, we may argue as in Lemma 6.2 in \cite{huisken2001inverse} that $\Sigma'$ can be approximated in $C^1$ by smooth free boundary surfaces with uniformly bounded mean curvature and it then follows from Lemma \ref{weak gauss bonnet} and the discussion preceding it that $\Sigma'$ is in fact of class $W^{2,2}$. Now,  \cite{gerhardt1973regularity} implies that $\Sigma'$ is in $C^{1,1}$ away from the boundary with corresponding interior estimates. In order to show regularity at the boundary, let us pick $p\in \partial \Sigma'$. We may then choose coordinates  $\Psi:B^3_{\delta,+}:=B^3_\delta(0)\cap \{x_1\geq 0\}\to M$ by first choosing orthonormal unit length directions $e_2,e_3$ at $p$ tangential to $\partial M$ such that $e_2$ is tangential to $\partial\Sigma'$, then picking a local parametrization $\Phi$ of $\partial M$ such that $\partial_2\Phi=e_2$ and $\partial_3\Phi=e_3$ at $p$ and finally defining
	$$\Psi(x_1,x_2,x_3)=\exp_{\Phi(x_2,x_3)}(-x_1\mu).$$
	Given $\epsilon>0$, we may shrink $\delta$ to arrange that
	$$
	(1-\epsilon)\operatorname{Id}\leq g \leq (1+\epsilon)\operatorname{Id} 
	$$
	on $B^3_{+,\delta}$, $g_{13}=g_{23}=0$, $g_{33}=1$ and $\partial_1g_{i1}=0$ on $\partial M$ for $i\in\{1,2,3\}$  as well as $g=\operatorname{Id}$ at $p$. Next, we can assume that $\Sigma'\cap \Psi(B^3_{\delta,+})$ can be written as the graph of a function $\rho\in C^{1,\alpha}\cap W^{2,2}( D^{2}_{ \delta}(0)\cap \{x_1\geq 0\})$ with $\rho(0,0)=0$ and $\nabla_e\rho(0,0)=0$. The free boundary condition says that $\partial_1 \rho(0,x_2)=0$ for all $x_2$. We may also assume that $\tilde\partial E\cap \Psi(B^3_{\delta,+})$ is given as the graph of a function $\psi$ over the same disk such that $\rho\geq \psi$. On the other hand, if $\upsilon\in C^{1,\alpha}\cap W^{2,2}( D^{2}_{ \delta}(0)\cap \{x_1\geq 0\})$ is another function satisfying $\upsilon\geq \psi$, it follows that $|\operatorname{graph}(\upsilon)|_g\geq |\operatorname{graph}(\rho)|_g$ where we used that $E'$ is the strictly minimizing hull of $E$. Moreover, we may reflect the metric $g$ across $\{x_1=0\}$ to obtain a $C^{0,1}-$metric $\tilde g$ on $B^3_\delta(0)$ and similarly reflect $\rho$ to obtain a function $\tilde \rho\in W^{2,2}\cap C^{1,\alpha}(D^2_\delta(0))$. Again, it follows that $|\operatorname{graph}\tilde \upsilon|_{\tilde g}\geq |\operatorname{graph}\tilde \rho|_{\tilde g}$ for all $\tilde \upsilon \in W^{2,2}\cap C^{1,\alpha}(D^2_\delta(0))$ such that $\tilde \upsilon \geq \tilde \psi$, where $\tilde \psi$ is the reflection of $\psi$. Computing the derivative of the area of the graph of $\tilde \rho+s(\tilde \upsilon-\tilde \rho)$ at $s=0$ one concludes that $\tilde \rho$ satisfies an elliptic, non-linear variational inequality. Now one may argue as in Proposition 3.2 in \cite{focardi2017classical} to conclude that $\tilde \rho$ actually solves a non-linear elliptic equality of the form 
	\begin{align}
	a_{\alpha\beta}\partial_\alpha\partial_\beta \tilde \rho+b_\alpha \partial_\alpha\tilde\rho =\zeta, \label{nl el equation}
	\end{align} 
	where $\zeta\in L^\infty(D^2_\delta(0))$ with estimates depending on $|\tilde \psi|_{C^{1,1}(D^2_\delta(0))}$ (note that $\tilde \psi\in C^{1,1}(D^2_\delta(0))$ since $\tilde \partial E$ is a free boundary surface). Using the estimates for $|\rho|_{C^{1,\alpha}( D^{2}_{ \delta}(0)\cap \{x_1\geq 0\})}, |g|_{C^1( B^{3}_{ \delta}(0)\cap \{x_1\geq 0\})}$ one may then check that although $\tilde{g}$ is not quite smooth, the coefficients are still regular enough so that one can apply the $L^p-$theory, see for instance Lemma 9.17 in \cite{gilbarg2015elliptic}, to conclude that $\rho \in W^{2,p}(( D^{2}_{ \delta}(0)\cap \{x_1\geq 0\}))$ for all $p<\infty$. It remains to improve this to $C^{1,1}-$regularity. In order to estimate $\partial_2\partial_2 \rho$, we can now argue exactly as in \cite{gerhardt1973regularity}. The argument is completely verbatim with the only difference that Theorem 4.1 in \cite{stampacchia1965probleme} is now applied to a function $\upsilon$ on a half disk and that the right-hand side of the estimate stated in the theorem  now contains the maximum of $\upsilon$ on the non-Neumann part of the boundary. The proof is essentially the same with the only difference that we now use the Sobolev inequality on the domain $H_0^{1,2}((D^2_+\setminus I)\cup  \{x_1=0\})$, where $I=\{\rho=\psi\}$ is the coincidence set, and verify by reflection that the Sobolev constant does not depend on $I$. In order to control the other components of $\nabla_e^2  \rho$, we may appeal to the equation (\ref{nl el equation}) and the uniform ellipticity.
\end{proof}
Combining the growth formula (\ref{h2 growth formula}) with the exponential area growth we are now able to prove the monotonicity of the modified Hawking mass in the special situation where $\tilde \partial M$ consists of one free boundary component.
\begin{coro}
	Let $(M',g)$ be an exterior region whose interior boundary consists of a connected free boundary surface $\Sigma$ bounding the set $E_0$. Suppose that $(M,g)$ satisfies the dominant energy condition $\operatorname{Sc},H^{\partial M}\geq 0$. Let $E_t$ be the precompact solution of the weak inverse mean curvature flow starting at $E_0$. Then the quantity $m_H(\Sigma_t)$ is non-decreasing. More precisely, there holds
	\begin{equation}
	\begin{aligned}
	m_H&(\Sigma_{t_1})\geq m_H(\Sigma_{t_0})\\&+\int_{t_0}^{t_1}\frac{(2|\Sigma_t|)^{\frac12}}{(16\pi)^{\frac32}}\bigg[4\pi(1-\chi(\Sigma_t))+\int_{\Sigma_t}\bigg(2\frac{|\nabla H|^2}{H^2}+\frac12 |\Acirc|^2+\operatorname{Sc}\bigg)\text{d}vol+\int_{\partial \Sigma_t} H^{\partial M}\text{d}vol\bigg]\text{d}t
	\end{aligned}
	\label{geroch monotonicity formula}
	\end{equation}
	for every $0\leq t_0<t_1$.
	\label{geroch monotonicity}
\end{coro}
\begin{proof}
	Since $M'$ is an exterior region, $E_0$ is a strictly minimizing hull, see Lemma \ref{topological structure}. It then follows from Lemma \ref{regularity of approximate level sets} that $|\Sigma_t|=e^t|\Sigma|$. Combining this with the growth formula (\ref{h2 growth formula})  we obtain (\ref{geroch monotonicity formula}). On the other hand, since $M'$ is an exterior region and since $\Sigma$ is a connected free boundary surface, it follows from Lemma \ref{connected free boundary} that $\chi(\Sigma_t)\leq 1$ for all $t>0$. The claim now follows from the dominant energy condition.
\end{proof}
\begin{rema}
	If $(M',g)$ is an exterior region with more than one interior boundary component, at least one of them being a free boundary surface, one may construct a weak flow starting at one of the free boundary components in a way such that the modified Hawking mass remains non-decreasing. To this end, one evolves the flow $E_t$ starting from the chosen component until the time $t_0>0$ when  it meets one of the sets $E$ enclosed by another interior boundary component. Then, one replaces $E_t$ by $F_t:=(E_t\cup E)'$. It is easy to see that this replacement can only increase the modified Hawking mass and it is possible to restart the flow from $F_t$ in a way such that the modified Hawking mass remains non-decreasing (note that $\tilde\partial F_t$ is again a connected free boundary surface). This procedure terminates after the finite number of sets enclosed by interior boundary components have been swallowed and the flow subsequently continues as a usual weak inverse mean curvature flow. We refer to Section 6 in \cite{huisken2001inverse} for details and note that the reasoning presented there can be applied verbatim to the  case treated in this article, that is, $\partial M\neq \emptyset$.
	\label{geroch mon rema}
\end{rema}
\section{Asymptotic Behaviour}\
\label{asymptotic behavior section}
In this section, we study the behaviour of the weak flow in the asymptotic region. Consequently, we assume that $(M,g)$ is an asymptotically flat half-space with one end and that $u$ is a proper weak solution of the free boundary inverse mean curvature flow with connected level sets. Using a weak blowdown argument similar to the one used in \cite{huisken2001inverse}, we show that the leaves $\Sigma_t=\{u=t\}$ become close to round hemispheres in $C^{1,\alpha}$. We then proceed to show that the free boundary Hawking mass is asymptotic to the ADM-mass. Combining this with the results from the previous section we are then in the position to prove the main result of this article.
\subsection{Asymptotic estimates}
We start with a more precise asymptotic decay estimate for the gradient of the weak solution $u$. This will then allow us to perform the weak blowdown.
\begin{lem}
	Let $(M,g)$ be an asymptotically flat half-space with one end and $u\in C_{loc}^{0,1}(M)$ a proper weak solution of the free boundary inverse mean curvature flow with connected level sets. Then there exists  a constant $c$ and a compact set $\Omega\subset\subset M$ depending on the asymptotic behaviour of $g$ as well as the constant from Lemma \ref{lem global grad estimate}  such that $M-\Omega$ is diffeomorphic to $\mathbb{R}^3_+\setminus B^3_{1}(0)$ and such that for all $x\in \mathbb{R}^3_+\setminus B^3_{1}(0)$ there holds
	$$
	|\overline{\nabla} u|(x)\leq \frac{c}{|x|_e}.  
	$$
	\label{asymptotic grad est}
\end{lem}
\begin{proof}
	Let $\gamma>1, \epsilon>0$. Without loss of generality, we may assume that $u=u_{\epsilon,\gamma,Z_\epsilon}$.  We choose $\Omega\subset\subset M$ such that $M\setminus \Omega$ is covered by the chart at infinity. We may assume that this chart maps $\tilde \partial \Omega$ to $ \partial B^3_{R_0}(0)\cap \mathbb{R}^3_+$ for some constant $R_0>0$ to be chosen. It follows from (\ref{asymptotic behavior 0}-\ref{asymptotic behavior 1})  that 
	\begin{align} |A^{\partial M}|\leq \frac{c}{|x|_e}, \qquad |\mu+e_3|\leq \frac{c}{|x|_e}, 
	\label{asymptotic normal} \end{align}
	for all $x\in(\mathbb{R}^2\times\{0\})\setminus B^3_{R_0}(0)$.
	Moreover, Lemma \ref{lem global grad estimate} implies that there is a constant $c_0>0$ such that $|\overline{\nabla} u|_{L^\infty(M)}\leq c_0$. We proceed to estimate $|\overline{\nabla} u|$ at a point $x_0\in\mathbb{R}^3_+\setminus B^3_{2R_0}(0)$. We choose a smooth, decreasing function $\rho:\mathbb{R}\to [0,\infty)$  satisfying $\rho(s)\equiv 1$ for $s\leq 1$, $\rho'\geq -2$ as well as $\rho(s)\leq 0$ for $s\geq 2$ and another smooth function $\zeta$ such that $\zeta\leq 2$, $\zeta'\geq 0$, $\zeta(0)=\zeta'(0)=1$ as well as $|\zeta'|,|\zeta''|\leq 2$. Furthermore, let $\Gamma>0$ be a positive constant to be chosen and define\footnote{For $Q$ to be well-defined and smooth one may have to decrease $\epsilon>0$ depending on $|x_0|_e$ such that $\operatorname{spt}(\tilde \rho )\subset M'_\epsilon$. We use the Euclidean distance function since the estimates for its Hessian are independent of the ambient curvature.}  the function $Q:\mathbb{R}^3_+ \setminus B^3_{R_0}(0)\to \mathbb{R}$ via $$Q(x):=\rho^2\bigg(\frac{4\operatorname{dist}_e(x,x_0)}{|x_0|_e}\bigg)\zeta\bigg(\frac{\Gamma x_3}{|x_0|_e}\bigg)|\overline{\nabla} u|^2(x)=:\tilde\rho^2(x)\tilde\zeta(x) |\overline{\nabla} u|^2(x).$$
	Since $Q$ is non-positive outside of a compact set, it follows that $Q$ attains a non-negative maximum at some point $x\in\mathbb{R}_+^3\setminus B^3_{R_0}(0)$ and we may assume that $Q(x)>0$. The case $|x|_e=R_0$ is excluded since this implies that $\operatorname{dist}_e(x,x_0)\geq  |x_0|_e/2$ and consequently $Q(x)\leq 0$. In fact, we may assume that
	\begin{align}
	\frac12|x_0|_e\leq |x|_e\leq 2|x_0|_e. \label{x x0 comp}
	\end{align}
	Furthermore, we may assume that $\tilde\rho(x)\geq |x_0|_e^{-1}$ because otherwise we may simply estimate
$$
	|\overline{\nabla} u|^2(x_0)\leq Q(x)\leq \frac{2 c_0^2}{|x_0|_e^2}.
	$$
	Next, let us assume that $x\in \mathbb{R}^2\times\{0\}$. Using (\ref{asymptotic behavior 0}-\ref{asymptotic behavior 1}), (\ref{asymptotic normal}), $\rho'\leq 0$ and $g_e(x-x_0,-e_3)>0$  it follows that   $\rho'\partial_\mu \operatorname{dist}_e(\cdot,x_0)|_x\leq c|x|_e^{-1}$. Similarly, we obtain $\partial_\mu x_3\leq -1/2$.  On the other hand, the free boundary condition implies $\partial_\mu |\overline{\nabla} u|^2=-2A^{\partial M} (\overline{\nabla} u,\overline{\nabla} u)$. Consequently, we estimate at $x$ 
	\begin{align*}
	\partial_\mu Q&\leq -2A^{\partial M}(\overline{\nabla}u,\overline{\nabla}u) +\frac{\Gamma}{2|x_0|_e} \tilde \rho^2|\overline{\nabla u}|^2+\frac{c}{|x|_e|x_0|_e}\tilde\rho |\overline{\nabla} u|^2  \\& 
	\leq |\overline{\nabla} u|^2\tilde \rho \bigg(\frac{c}{|x_0|_e}\tilde \rho -\frac{\Gamma}{|x_0|_e}\tilde \rho +\frac{c}{|x_0|_e^2} \bigg)
\\&	<0,
	\end{align*}
	provided $\Gamma$ is sufficiently large. In the second step, we used (\ref{asymptotic normal}) and (\ref{x x0 comp}) and in the last step we used $\tilde\rho(x)\geq|x_0|_e^{-1}$.  As this is a contradiction, it follows that $Q$ attains an interior maximum. We can then argue as in the proof of Lemma \ref{interior estimate}, using however an explicit estimate for the hessian of the Euclidean distance function and the estimate $\operatorname{Rc}(\overline{\nabla}u,\overline{\nabla u})\geq -c|\overline{\nabla u}|^2/|x|_e^2$ in Corollary \ref{coro L est}, to deduce that 
	$$ |\overline{\nabla} u|^2(x_0)|\leq Q(x)\leq \frac{c}{|x_0|_e^2}.$$
\end{proof}
We now proceed with the weak blowdown. To this end, we assume that $\Omega\subset\subset M$ is a compact set such that $M-\Omega$ is covered by the asymptotic chart and diffeomorphic to $\mathbb{R}^3_+\setminus B^3_{1}(0)$. We fix a constant $\delta>0$ and define the rescaled quantities $M^\delta:=\delta (M\setminus \Omega)$, $g^\delta(x):=\delta^2 g(\delta^{-1}x)$, $u^\delta:=u(\delta^{-1}x)$ and $E^\delta_t:=\delta E_t$. It is then easy to see that $u^\delta$ is a weak solution of the free boundary inverse mean curvature flow in $(M^\delta,g^\delta)$. The following two lemmas are very similar to Lemma 7.1 and Lemma 7.4 in \cite{huisken2001inverse} and we only sketch the arguments.
\begin{lem}
	Let $(M,g)$, $\Omega\subset M$ and $u$ be as in the previous lemma. There are constants $c_\delta\to\infty$ such that $u^\delta-c_\delta\to 2 \log(|x|)$ locally uniformly in $\mathbb{R}_+^3\setminus\{0\}$. Moreover, the expanding hemisphere solution $x\mapsto2\log(|x|)$ is the only precompact weak free boundary solution in $\mathbb{R}^3_+\setminus\{0\}$.
	\label{weak blow down}
	\begin{proof}
		Since $u$ is proper, it follows that there is a time $t_0>0$ such that for any $t\geq t_0$ $\Sigma_t$ is contained in $M\setminus \Omega$. A computation similar to Lemma \ref{subsolution} using the asymptotic behaviour (\ref{asymptotic behavior 0}) shows that we can choose numbers $\Lambda,\Theta>0$ and a time $t_1>0$ such that for all $t>t_1$ the family of expanding hemispheres $B^3_{e^{\Theta t}}(- \Lambda e_3)\cap\mathbb{R}_+^3$ is a subsolution. Consequently, we assume that $t>\operatorname{max}\{t_0,t_1\}$ and define the eccentricity of $\Sigma_t$ to be the quotient $R_t/r_t$ where $R_t$ is the smallest number such that $E_t\subset B^3_{R_t}(-\Lambda e_3)\cap\mathbb{R}_+^3$ and $r_t$ is the largest number such that $B^3_{r_t}(-\Lambda e_3)\cap\mathbb{R}_+^3\subset E_t$. Using the expanding hemisphere barriers, it follows that $R_{t+s}\leq e^{\Theta s}R_t$ for any $s\geq 0$. On the other hand, let $t_0$ be sufficiently large such that $r_t\geq R_0$. The definition of $r_t$ demands that there must be a point $x\in \partial B^3_{r_t}(-\Lambda e_3)\cap\mathbb{R}_+^3$ such that $u(x)=t$. Using the previous lemma and integration it follows that there exists a constant $c_0$ independent of $t$ such that $u>t-c_0$ everywhere on 
		$\partial B^3_{r_t}(-\Lambda e_3)\cap \mathbb{R}_+^3$. As $\Sigma_{t-c_0}$ is connected, we conclude that $\Sigma_{t-c_0}\subset B^3_{r_t}(-\Lambda e_3)\cap\mathbb{R}_+^3$. It follows that \begin{align}
		R_t\leq e^{\Theta c_0}R_{t-c_0}\leq e^{\Theta c_0}r_t.
		\label{eccentricity est}
		\end{align}
		Now let $\delta_i>0$ by any sequence converging to zero. The eccentricity estimate (\ref{eccentricity est}) and the gradient estimate from the previous lemma are scale invariant. The Arzela-Ascoli theorem and the compactness result Lemma \ref{compactness lemma} then imply that there is a subsequence, labelled the same, and a sequence $c_{\delta_i}\to\infty$ such that $u^{\delta_i}-c_{\delta_i}$ converges to a weak solution $v$ locally uniformly in $\mathbb{R}^3_+\setminus\{0\}$ with local $C^{1,\alpha}$ convergence of the level sets. Using the eccentricity estimate (\ref{eccentricity est}), we can now argue as in \cite{huisken2001inverse} to deduce that $v$ is non-constant and is in fact a precompact solution with $t$ ranging from $-\infty$ to $\infty$. In order to complete the proof, it remains to show that the hemisphere solution is the only precompact weak solution in $\mathbb{R}^3_+\setminus\{0\}$.  We briefly sketch the argument. \\In the situation of the exact Euclidean half-space, we may choose $\Lambda=0$ and $\Theta=1/2$. In fact, the expanding hemispheres are exact solutions and using them as barriers it follows that the eccentricity is non-increasing. It is easy to see that the eccentricity approaches $1$ as $t\to\infty$ and we would like to show that it is constant and equal to $1$. If not, then we can again perform a blowdown, this time using a sequence $\delta_i\to\infty$ to obtain another weak solution defined on $\mathbb{R}^3_+-\{0\}$ with constant eccentricity strictly larger than $1$. It follows that a level set $\Sigma_{t_0}$ lies between $\partial B^3_r(0)\cap\mathbb{R}^3_+$ and $\partial B^3_R(0)\cap\mathbb{R}^3_+$, has eccentricity $R/r>1$ and is not equal to either of these hemispheres. However, this implies that we may slightly perturb $\partial B^3_R(0)\cap\mathbb{R}^3_+$ inwards such that the perturbation is still  a strictly mean convex, star shaped free boundary surface lying on one side of $\Sigma_{t_0}$. We may also perturb $\partial B^3_r(0)\cap\mathbb{R}^3_+$ outwards in a similar way. It then follows from \cite{marquardt2013inverse} that the smooth inverse mean curvature flow starting at either of these perturbations exists for all times and consequently coincides with the weak solution starting at the respective perturbation, c.f. Lemma \ref{regularity of approximate level sets}. By the strong maximum principle, both solutions are disjoint from the corresponding expanding hemisphere solutions. However, the weak maximum principle implies that $\Sigma_t$ remains between the two weak solutions starting at the perturbations and consequently the eccentricity must decrease, a contradiction.
	\end{proof}
\end{lem}
The next lemma constitutes the final ingredient for the proof of the main theorem. As usual, the subscript $e$ indicates that the respective quantity is computed with respect to the Euclidean background metric.
\begin{lem}
	Let $(M,g)$ be an asymptotically flat half-space with one end and $E_t$ be a precompact weak solution of the free boundary inverse mean curvature flow such that $\tilde\partial E_t$ is connected. There holds
	$$
	\lim_{t\to\infty} m_H(\Sigma_t)\leq m_{ADM}.
	$$
	\label{hawking asymptotic}
\end{lem}
\begin{proof}
	The proof is very similar to Lemma 7.4 in \cite{huisken2001inverse}. We therefore only sketch some details and point out the few differences. We define $r(t)$ by requiring that $|\Sigma_t|=2\pi r^2(t)$. In the notation of the previous lemma this means that $|\Sigma^{1/r(t)}_t|_{g^{1/r(t)}}=2\pi$. Moreover, the previous lemma together with Lemma \ref{compactness lemma} implies that $r(t)\to\infty$ as $t\to\infty$ and that $\Sigma_t^{1/r(t)}\to \mathbb{S}^2\cap \mathbb{R}^3_+$ in $C^{1,\alpha}$, where $0<\alpha<1/2$. In particular, $\partial \Sigma_t^{1/r(t)}\to \partial D^2_1(0)$ in $C^1$. Rescaling, we deduce that
	\begin{align}
	|\partial \Sigma_t|=2\pi r(t)+\mathcal{O}(1). \label{asymptotic boundary length}
	\end{align} 
	As in \cite{huisken2001inverse}, we can estimate the Willmore energy of $\Sigma_t$ by comparing it to the Euclidean Willmore energy and using the well-known estimate
	$$
	\int_{\tilde \Sigma} H_e^2\text{d}vol_e \geq 8\pi,
	$$
	valid for any free boundary disc $\tilde \Sigma$ in $\mathbb{R}^3_+$. This follows for instance from Lemma \ref{weak gauss bonnet} together with the fact that the plane is totally geodesic. We then find
	\begin{equation}
	\begin{aligned}
	\int_{\Sigma_t} H^2\text{d}vol \geq 8\pi +\int_{\Sigma_t}\bigg (&\frac12 H^2\operatorname{tr}_{\Sigma_t}q-2Hh(q,A) +H^2q(\nu,\nu)-2H\operatorname{tr}_{\Sigma_{t_i}}\overline{\nabla}q(\cdot,\cdot,\nu)\\&+H\operatorname{tr}_{\Sigma_{t_i}}\overline{\nabla}q(\nu,\cdot,\cdot)-c|q|^2|A|^2-c|\overline{\nabla }q|^2)\text{d}vol,
	\end{aligned}
	\label{wilm exp}
	\end{equation}
	where $h$ is the induced metric of $\Sigma_t$ and $q:=g-g_e$. Next, Lemma \ref{asymptotic grad est} and Lemma \ref{regularity of approximate level sets} imply
	\begin{align}
	H=|\overline{\nabla }u|\leq \frac{c}{|x|_e}\leq \frac{c}{r(t)} \label{h est}
	\end{align}
	on $\Sigma_t$. Using the Gauss equation (\ref{Gauss equation}), we may then deduce from Lemma \ref{weak gauss bonnet}, (\ref{asymptotic behavior 1}), (\ref{asymptotic normal})  and (\ref{asymptotic boundary length}) that
	$$
	\int_{s}^{s+1} \int_{\Sigma_t} |A|^2 \text{d}vol\text{d}t\leq c+c\int_s^{s+1}\int_{\partial \Sigma_t} |A^{\partial M}|\text{d}vol\text{d}t \leq c.
	$$
	It follows that there exists a sequence $t_i\to\infty$ such that 
	$$
	\limsup_{i\to\infty} \int_{\Sigma_{t_i}} |A|^2\text{d}vol <\infty.
	$$	
	In particular, (\ref{asymptotic behavior 0}), (\ref{wilm exp}) and (\ref{h est}) then imply
	$$
	\liminf_{i\to\infty}\int_{\Sigma_{t_i}} H^2\text{d}vol \geq 8\pi-\frac{c}{r}
	$$
	and consequently $\limsup_{i\to\infty}m_H(\Sigma_{t_i})<\infty$. We may then use the monotonicity formula Lemma \ref{geroch monotonicity} and the Rellich-Kochandrov theorem to select another subsequence such that $H^{\Sigma^{1/r(t_i)}_{t_i}}\to H^{\mathbb{S}^2\cap\mathbb{R}^3_+}=2$ in $L^2(\mathbb{S}^2\cap\mathbb{R}^3_+)$ and $A^{\Sigma^{1/r(t_i)}_{t_i}}\to h_{\mathbb{S}^2\cap\mathbb{R}^3_+}$ in $L^2(\mathbb{S}^2\cap\mathbb{R}^3_+)$.  Revisiting (\ref{wilm exp}) it follows that
	$$
	32\pi m_H(\Sigma_{t_i})\leq\int_{\Sigma_{t_i}}\bigg(\frac{2}{r(t_i)}\operatorname{tr}_{\Sigma_{t_i}}q-\frac{4}{r(t_i)}q(\nu,\nu)
	+4\operatorname{tr}_{\Sigma_{t_i}}\overline{\nabla}q(\cdot,\cdot,\nu)-2\operatorname{tr}_{\Sigma_{t_i}}\overline{\nabla}q(\nu,\cdot,\cdot)\bigg)\text{d}vol+o(1).	
	$$
	Using the integration by parts formula (\ref{int by parts formula}) and (\ref{asymptotic behavior 0})  we find 
	$$
	2\int_{\Sigma_{t_i}}\operatorname{tr}_{\Sigma_{t_i}}\overline{\nabla}q(\cdot,\cdot,\nu)\text{d}vol=\int_{\Sigma_{t_i}}\bigg(\frac{4}{r(t_i)}q(\nu,\nu)-\frac{2}{r(t_i)}\operatorname{tr}_{\Sigma_{t_i}}q\bigg)\text{d}vol +2\int_{\partial\Sigma_{t_i}} q(\nu,\mu)\text{d}vol+o(1).
	$$
	Moreover, the asymptotic decay implies $q(\nu,\mu)=-g_e(\nu,\mu)=g(\nu_e,-\mu_e)+{\mathcal{O}}(r(t_i)^{-2})$ and it follows that
	$$
	32\pi m_H(\Sigma_{t_i})\leq 2\int_{\Sigma_{t_i}}(\operatorname{tr}_{\Sigma_{t_i}}\overline{\nabla}g(\cdot,\cdot,\nu_e)-\operatorname{tr}_{\Sigma_{t_i}}\overline{\nabla}g(\nu_e,\cdot,\cdot))\text{d}vol_e+\int_{\partial \Sigma_{t_i}} g(\nu_e,-\mu_e)\text{d}vol+{o}(1).
	$$
	We may now apply the divergence theorem to both integrals and use the integrability of both $\operatorname{Sc}$ as well as  $H^{\partial M}$, the asymptotic decay of the metric (\ref{asymptotic behavior 0}) and the $C^{1,\alpha}-$convergence of $\Sigma^{r(t_i)}_{t_i}$ to $\mathbb{S}^2\cap\mathbb{R}^3_+$ to conclude that 
	\begin{align*}
	32\pi m_H(\Sigma_{t_i})\leq& 2\int_{\mathbb{S}^2_{r(t_i)}\cap\mathbb{R}^3_+}\bigg(\operatorname{tr}_{\mathbb{S}^2_{r(t_i)}\cap\mathbb{R}^3_+}\overline{\nabla}g(\cdot,\cdot,\nu_e)-\operatorname{tr}_{\mathbb{S}^2_{r(t_i)}\cap\mathbb{R}^3_+}\overline{\nabla}g(\nu_e,\cdot,\cdot)\bigg)\text{d}vol_e\\&+\int_{\partial(\mathbb{S}^2_{r(t_i)}\cap\mathbb{R}^3_+)} g(\nu_e,\partial_3)\text{d}vol+{o}(1),
	\end{align*}
	which should be compared with Lemma 7.3 in \cite{huisken2001inverse}. Recalling the definition of the ADM-mass (\ref{ADM mass}) we find
	$$
	\limsup_{i\to\infty} m_H(\Sigma_{t_i})\leq m_{ADM}.
	$$
	The claim now follows from Lemma \ref{geroch monotonicity}.
\end{proof}
\subsection{Proof of the main results}
\label{proof of the main result}
We are finally in the position to prove the main results.
\begin{proof}[Proof of Theorem \ref{main thm RPI}] Let $(M',g)$ be an exterior region and $\Sigma$ be one of its interior free boundary components. Using either Corollary \ref{geroch monotonicity} if $\Sigma$ is the only interior boundary component or otherwise arguing as in Remark \ref{geroch mon rema}, we find that there exists a weak free boundary inverse mean curvature flow starting at $\Sigma$ such that $m_H(\Sigma_t)$ is non-decreasing and such that $\Sigma_t$ is a connected free boundary surface for all times. Using the previous lemma we find
	$$
	\sqrt{\frac{|\Sigma|}{32\pi}}\leq \limsup_{t\to\infty}m_H(\Sigma_t)\leq m_{ADM}
	$$
	as claimed. The equality case is again very similar to \cite{huisken2001inverse} and we  only sketch the details. 	If equality holds, then it follows from Corollary \ref{geroch monotonicity}, lower semi-continuity and Lemma \ref{regularity of approximate level sets} that all leaves $\Sigma_t$ have constant mean curvature $H(t)>0$ almost everywhere and are topological discs. Moreover, it follows from elliptic regularity and Lemma \ref{regularity of approximate level sets} that all sets $\Sigma_t,\Sigma_t^+$ are smooth with locally uniform $C^k-$estimates. If there is a jump, then $\Sigma_t^+- \Sigma_t$ is non-empty for some $t$ and consequently $H^{\Sigma_t^+}=0$ by the constancy of the mean curvature and (\ref{mc minimizing hull}), contradicting the fact that $M'$ is free of free boundary minimal surfaces. Consequently, $H(t)>0$ for all $t>0$ and it follows from Lemma \ref{regularity of approximate level sets} that the entire flow is smooth. Moreover, it follows that all leaves have constant Gauss curvature and are totally umbilic while the scalar curvature of $M$ and the mean curvature of $\partial M$ vanish. One may now argue as in \cite{huisken2001inverse} to show that this can only hold if $(M',g)$ is the exterior region of the spatial Schwarzschild half-space.
\end{proof}
\begin{proof}[Proof of Corollary \ref{volkmann coro}.] This follows from Theorem \ref{main thm RPI} and Lemma \ref{mass vs extrinsic mass}.
\end{proof}


\begin{thebibliography}{ABLdL16}
	
	\bibitem[ABLdL16]{almaraz2014positive}
	S{\'e}rgio Almaraz, Ezequiel Barbosa, and Levi Lopes~de Lima.
	\newblock A positive mass theorem for asymptotically flat manifolds with a
	non-compact boundary.
	\newblock {\em Communications in Analysis and Geometry}, 24(4):673--715, 2016.
	
	\bibitem[ADM61]{arnowitt1961coordinate}
	Richard Arnowitt, Stanley Deser, and Charles~W Misner.
	\newblock Coordinate invariance and energy expressions in general relativity.
	\newblock {\em Physical Review}, 122(3):997, 1961.
	
	\bibitem[Bar86]{bartnik1986mass}
	Robert Bartnik.
	\newblock The mass of an asymptotically flat manifold.
	\newblock {\em Communications on pure and applied mathematics}, 39(5):661--693,
	1986.
	
	\bibitem[BM18]{barbosa2018positive}
	Ezequiel Barbosa and Adson Meira.
	\newblock A positive mass theorem and penrose inequality for graphs with
	noncompact boundary.
	\newblock {\em msp. org/pjm}, 294(2):257, 2018.
	
	\bibitem[Bra01]{bray2001proof}
	Hubert~L Bray.
	\newblock Proof of the riemannian penrose inequality using the positive mass
	theorem.
	\newblock {\em Journal of Differential Geometry}, 59(2):177--267, 2001.
	
	\bibitem[Buc05]{buckland2005mean}
	John~A Buckland.
	\newblock Mean curvature flow with free boundary on smooth hypersurfaces.
	\newblock {\em Journal f{\"u}r die reine und angewandte Mathematik},
	2005(586):71--90, 2005.
	
	\bibitem[Car16]{carlotto2016rigidity}
	Alessandro Carlotto.
	\newblock Rigidity of stable minimal hypersurfaces in asymptotically flat
	spaces.
	\newblock {\em Calculus of variations and partial differential equations},
	55(3):54, 2016.
	
	\bibitem[CGP10]{chrusciel2010mathematical}
	Piotr Chru{\'s}ciel, Gregory Galloway, and Daniel Pollack.
	\newblock Mathematical general relativity: a sampler.
	\newblock {\em Bulletin of the American Mathematical Society}, 47(4):567--638,
	2010.
	
	\bibitem[Cha18]{chai2018positive}
	Xiaoxiang Chai.
	\newblock Positive mass theorem and free boundary minimal surfaces.
	\newblock {\em arXiv preprint arXiv:1811.06254}, 2018.
	
	\bibitem[FGS17]{focardi2017classical}
	Matteo Focardi, Francesco Geraci, and Emanuele Spadaro.
	\newblock The classical obstacle problem for nonlinear variational energies.
	\newblock {\em Nonlinear Analysis: Theory, Methods \& Applications},
	154:71--87, 2017.
	
	\bibitem[Gal93]{galloway1993topology}
	Gregory~J Galloway.
	\newblock On the topology of black holes.
	\newblock {\em Communications in mathematical physics}, 151(1):53--66, 1993.
	
	\bibitem[Ger73a]{gerhardt1973regularity}
	Claus Gerhardt.
	\newblock Regularity of solutions of nonlinear variational inequalities.
	\newblock {\em Archive for Rational Mechanics and Analysis}, 52(4):389--393,
	1973.
	
	\bibitem[Ger73b]{geroch1973energy}
	Robert Geroch.
	\newblock Energy extraction.
	\newblock {\em Annals of the New York Academy of Sciences}, 224(1):108--117,
	1973.
	
	\bibitem[Ger90]{gerhardt1990flow}
	Claus Gerhardt.
	\newblock Flow of nonconvex hypersurfaces into spheres.
	\newblock {\em Journal of Differential Geometry}, 32(1):299--314, 1990.
	
	\bibitem[GJ86]{gruter1986allard}
	Michael Gr{\"u}ter and J{\"u}rgen Jost.
	\newblock Allard type regularity results for varifolds with free boundaries.
	\newblock {\em Annali della Scuola Normale Superiore di Pisa-Classe di
		Scienze}, 13(1):129--169, 1986.
	
	\bibitem[GLZ16]{guang2016curvature}
	Qiang Guang, Martin Man-chun Li, and Xin Zhou.
	\newblock Curvature estimates for stable free boundary minimal hypersurfaces.
	\newblock {\em Journal f{\"u}r die reine und angewandte Mathematik (Crelles
		Journal)}, 2016.
	
	\bibitem[Gr{\"u}87]{gruter1987optimal}
	Michael Gr{\"u}ter.
	\newblock Optimal regularity for codimension one minimal surfaces with a free
	boundary.
	\newblock {\em manuscripta mathematica}, 58(3):295--343, 1987.
	
	\bibitem[GT15]{gilbarg2015elliptic}
	David Gilbarg and Neil~S Trudinger.
	\newblock {\em Elliptic partial differential equations of second order}.
	\newblock Springer, 2015.
	
	\bibitem[HI01]{huisken2001inverse}
	Gerhard Huisken and Tom Ilmanen.
	\newblock The inverse mean curvature flow and the riemannian penrose
	inequality.
	\newblock {\em Journal of Differential Geometry}, 59(3):353--437, 2001.
	
	\bibitem[HP99]{gerhard1999geometric}
	Gerhard Huisken and Alexander Polden.
	\newblock Geometric evolution equations for hypersurfaces.
	\newblock In {\em Calculus of variations and geometric evolution problems},
	pages 45--84. Springer, 1999.
	
	\bibitem[Hui86]{huisken1986contracting}
	Gerhard Huisken.
	\newblock Contracting convex hypersurfaces in riemannian manifolds by their
	mean curvature.
	\newblock {\em Inventiones mathematicae}, 84(3):463--480, 1986.
	
	\bibitem[Hui88]{huisken1988expansion}
	G~Huisken.
	\newblock On the expansion of convex hypersurfaces by the inverse of symmetric
	curvature functions.
	\newblock {\em To appear}, 1988.
	
	\bibitem[Hut86]{hutchinson1986second}
	John~E Hutchinson.
	\newblock Second fundamental form for varifolds and the existence of surfaces
	minimising curvature.
	\newblock {\em Indiana Univ. Math. J.}, 35(1):45--71, 1986.
	
	\bibitem[JW77]{jang1977positive}
	Pong~Soo Jang and Robert~M Wald.
	\newblock The positive energy conjecture and the cosmic censor hypothesis.
	\newblock {\em Journal of Mathematical Physics}, 18(1):41--44, 1977.
	
	\bibitem[KN09]{kotschwar2009local}
	Brett Kotschwar and Lei Ni.
	\newblock Local gradient estimates of p-harmonic functions, 1/h-flow, and an
	entropy formula.
	\newblock {\em Ann. Sci. {\'E}c. Norm. Sup{\'e}r.(4)}, 42(1):1--36, 2009.
	
	\bibitem[Lie86]{lieberman1986mixed}
	Gary~M Lieberman.
	\newblock Mixed boundary value problems for elliptic and parabolic differential
	equations of second order.
	\newblock {\em Journal of Mathematical Analysis and Applications},
	113(2):422--440, 1986.
	
	\bibitem[Lie89]{lieberman1989optimal}
	Gary~M Lieberman.
	\newblock Optimal h{\"o}lder regularity for mixed boundary value problems.
	\newblock {\em Journal of Mathematical Analysis and Applications},
	143(2):572--586, 1989.
	
	\bibitem[LS16]{lambert2016inverse}
	Ben Lambert and Julian Scheuer.
	\newblock The inverse mean curvature flow perpendicular to the sphere.
	\newblock {\em Mathematische Annalen}, 364(3-4):1069--1093, 2016.
	
	\bibitem[LS17]{lambert2017geometric}
	Ben Lambert and Julian Scheuer.
	\newblock A geometric inequality for convex free boundary hypersurfaces in the
	unit ball.
	\newblock {\em Proceedings of the American Mathematical Society},
	145(9):4009--4020, 2017.
	
	\bibitem[LZ16]{li2016min}
	Martin Li and Xin Zhou.
	\newblock Min-max theory for free boundary minimal hypersurfaces i-regularity
	theory.
	\newblock {\em arXiv preprint arXiv:1611.02612}, 2016.
	
	\bibitem[Mar12]{marquardt2012inverse}
	Thomas Marquardt.
	\newblock {\em The inverse mean curvature flow for hypersurfaces with
		boundary}.
	\newblock PhD thesis, 2012.
	
	\bibitem[Mar13]{marquardt2013inverse}
	Thomas Marquardt.
	\newblock Inverse mean curvature flow for star-shaped hypersurfaces evolving in
	a cone.
	\newblock {\em Journal of Geometric Analysis}, 23(3):1303--1313, 2013.
	
	\bibitem[Mar17]{marquardt2017weak}
	Thomas Marquardt.
	\newblock Weak solutions of inverse mean curvature flow for hypersurfaces with
	boundary.
	\newblock {\em Journal f{\"u}r die reine und angewandte Mathematik (Crelles
		Journal)}, 2017(728):237--261, 2017.
	
	\bibitem[Mas74]{massari1974esistenza}
	Umberto Massari.
	\newblock Esistenza e regolarit{\`a} delle ipersuperfici di curvatura media
	assegnata in r n.
	\newblock {\em Archive for Rational Mechanics and Analysis}, 55(4):357--382,
	1974.
	
	\bibitem[MISY82]{meeks1982embedded}
	William Meeks~III, Leon Simon, and Shing-Tung Yau.
	\newblock Embedded minimal surfaces, exotic spheres, and manifolds with
	positive ricci curvature.
	\newblock {\em Annals of Mathematics}, pages 621--659, 1982.
	
	\bibitem[MIY82]{meeks1982classical}
	William~H Meeks~III and Shing-Tung Yau.
	\newblock The classical plateau problem and the topology of three-dimensional
	manifolds: the embedding of the solution given by douglas-morrey and an
	analytic proof of dehn's lemma.
	\newblock {\em Topology}, 21(4):409--442, 1982.
	
	\bibitem[Mos07]{moser2007inverse}
	Roger Moser.
	\newblock The inverse mean curvature flow and p-harmonic functions.
	\newblock {\em Journal of the European Mathematical Society}, 9(1):77--83,
	2007.
	
	\bibitem[Mos15]{moser2015geroch}
	Roger Moser.
	\newblock Geroch monotonicity and the construction of weak solutions of the
	inverse mean curvature flow.
	\newblock {\em Asian Journal of Mathematics}, 19(2):357--376, 2015.
	
	\bibitem[MY80]{meeks1980topology}
	William~H Meeks and Shing-Tung Yau.
	\newblock Topology of three dimensional manifolds and the embedding problems in
	minimal surface theory.
	\newblock {\em Annals of Mathematics}, 112(3):441--484, 1980.
	
	\bibitem[MY82]{meeks1982existence}
	William~W Meeks and Shing-Tung Yau.
	\newblock The existence of embedded minimal surfaces and the problem of
	uniqueness.
	\newblock {\em Mathematische Zeitschrift}, 179(2):151--168, 1982.
	
	\bibitem[PAF00]{pallara2000functions}
	L~Ambrosio-N Fusco-D Pallara, L~Ambrosio, and N~Fusco.
	\newblock {\em Functions of bounded variation and free discontinuity problems}.
	\newblock Oxford University Press, Oxford, 2000.
	
	\bibitem[Pen73]{penrose1973naked}
	Roger Penrose.
	\newblock Naked singularities.
	\newblock {\em Annals of the New York Academy of Sciences}, 224(1):125--134,
	1973.
	
	\bibitem[Pen82]{penrose1982some}
	Roger Penrose.
	\newblock Some unsolved problems in classical general-relativity.
	\newblock {\em Annals of Mathematics Studies}, pages 631--668, 1982.
	
	\bibitem[Per11]{perez2011nearly}
	Daniel~Raoul Perez.
	\newblock {\em On nearly umbilical hypersurfaces}.
	\newblock PhD thesis, University of Zurich, 2011.
	
	\bibitem[Ros08]{ros2008stability}
	Antonio Ros.
	\newblock Stability of minimal and constant mean curvature surfaces with free
	boundary.
	\newblock {\em Mat. Contemp}, 35:221--240, 2008.
	
	\bibitem[SSY75]{schoen1975curvature}
	Richard Schoen, Leon Simon, and Shing-Tung Yau.
	\newblock Curvature estimates for minimal hypersurfaces.
	\newblock {\em Acta Mathematica}, 134(1):275--288, 1975.
	
	\bibitem[Sta65]{stampacchia1965probleme}
	Guido Stampacchia.
	\newblock Le probl{\`e}me de dirichlet pour les {\'e}quations elliptiques du
	second ordre {\`a} coefficients discontinus.
	\newblock In {\em Annales de l'institut Fourier}, volume~15, pages 189--257,
	1965.
	
	\bibitem[Sta96a]{stahl1996convergence}
	Axel Stahl.
	\newblock Convergence of solutions to the mean curvature flow with a neumann
	boundary condition.
	\newblock {\em Calculus of Variations and Partial Differential Equations},
	4(5):421--441, 1996.
	
	\bibitem[Sta96b]{stahl1996regularity}
	Axel Stahl.
	\newblock Regularity estimates for solutions to the mean curvature flow with a
	neumann boundary condition.
	\newblock {\em Calculus of Variations and Partial Differential Equations},
	4(4):385--407, 1996.
	
	\bibitem[SY79]{schoen1979proof}
	Richard Schoen and Shing-Tung Yau.
	\newblock On the proof of the positive mass conjecture in general relativity.
	\newblock {\em Communications in Mathematical Physics}, 65(1):45--76, 1979.
	
	\bibitem[SZW91]{sternberg1991c1}
	Peter Sternberg, William~P Ziemer, and Graham Williams.
	\newblock C1, 1-regularity of constrained area minimizing hypersurfaces.
	\newblock {\em Journal of Differential Equations}, 94(1):83--94, 1991.
	
	\bibitem[Tam81]{tamanini1981boundaries}
	Italo Tamanini.
	\newblock {\em Boundaries of Caccioppoli sets with Holder-continuous normal
		vector}.
	\newblock Libera Universita di Trento. Dipartimento di Matematica, 1981.
	
	\bibitem[Urb90]{urbas1990expansion}
	John~IE Urbas.
	\newblock On the expansion of starshaped hypersurfaces by symmetric functions
	of their principal curvatures.
	\newblock {\em Mathematische Zeitschrift}, 205(1):355--372, 1990.
	
	\bibitem[Vol15]{volkmann2015free}
	Alexander Volkmann.
	\newblock {\em Free boundary problems governed by mean curvature}.
	\newblock PhD thesis, 2015.
	
	\bibitem[Wei18]{wei2018minkowski}
	Yong Wei.
	\newblock On the minkowski-type inequality for outward minimizing hypersurfaces
	in schwarzschild space.
	\newblock {\em Calculus of Variations and Partial Differential Equations},
	57(2):46, 2018.
	
\end{thebibliography}
\end{document}